\documentclass{siamart171218}

\usepackage{macros}
\usepackage[numbers]{natbib}
\usepackage{hyperref}
\usepackage{xspace}
\usepackage{graphicx}
\usepackage{float}
\usepackage{tikz}
\usepackage{overpic}

\newcommand{\subgrad}{\mathsf{g}}
\newcommand{\gradmap}{\mathsf{G}}
\newcommand{\meangradmap}{\wb{\gradmap}}

\newcommand{\toto}{\rightrightarrows}

\newcommand{\opt}{^\star}
\newcommand{\stepsize}{\alpha}
\newcommand{\ball}{\mathbb{B}}

\newcommand{\normalcone}{\mc{N}}

\newcommand{\statval}{s}
\newcommand{\statrv}{S}
\newcommand{\statdomain}{\mc{S}}
\newcommand{\xdomain}{X}

\newcommand{\noise}{\xi}
\newcommand{\cont}{\mathcal{C}}

\newcommand{\error}{\xi}
\newcommand{\stationary}{\xdomain\opt}
\newcommand{\continuous}{\mathcal{C}}

\newcommand{\regularizer}{\varphi}
\newcommand{\klow}{k_{<}}
\newcommand{\khigh}{k_{>}}
\newcommand{\bigindic}{\mathbb{I}}
\newcommand{\project}{\pi}

\providecommand{\opnorm}{\norm}
\newcommand{\lipc}{\beta}
\newcommand{\liph}{\gamma}
\newcommand{\lipf}{M}
\usepackage{enumitem}

\begin{document}

\title{Stochastic Methods for Composite and \\
  Weakly Convex Optimization Problems\thanks{
    Both authors were partially supported by NSF award CCF-1553086 
    and an Alfred P. Sloan fellowship.}}
\author{John C.\ Duchi\thanks{Department of Electrical Engineering,
    Department of Statistics, 
    Stanford University, CA 94305 (\email{jduchi@stanford.edu}).}
  \and Feng Ruan\thanks{Department of Statistics, Stanford University, 
    CA 94305 (\email{fengruan@stanford.edu})}. Supported 
  by an E.K.\ Potter Stanford Graduate Fellowship.}

\newcommand{\TheTitle}{Stochastic Methods for Composite Optimization Problems} 
\newcommand{\TheAuthors}{John C.\ Duchi and Feng Ruan}

\headers{\TheTitle}{\TheAuthors}

\maketitle


\begin{abstract}
  We consider minimization of stochastic functionals
  that are compositions of a (potentially) non-smooth convex
  function $h$ and smooth function $c$ and,
  more generally, stochastic weakly-convex functionals. We develop
  a family of stochastic methods---including
  a stochastic prox-linear algorithm
  and a stochastic (generalized) sub-gradient procedure---and
  prove that, under mild technical conditions, each converges
  to first-order stationary points of the stochastic objective.
  We provide experiments further investigating our methods
  on non-smooth phase retrieval problems; the experiments indicate
  the practical effectiveness of the procedures.
\end{abstract}


\section{Introduction}

Let $f : \R^d \to \R$ be the stochastic composite function
\begin{equation}
  \label{eqn:convex-composite-stochastic}
  f(x) \defeq \E_P[h(c(x; \statrv); \statrv)]
  = \int_\statdomain h(c(x; \statval); \statval) dP(\statval),
\end{equation}
where $P$ is a probability distribution on a sample space $\statdomain$ and
for each $\statval \in \statdomain$, the function $z \mapsto h(z; \statval)$
is closed convex and $x \mapsto c(x; \statval)$ is smooth.  In this paper,
we consider stochastic methods for minimization---or at least finding
stationary points---of such composite functionals.  The
objective~\eqref{eqn:convex-composite-stochastic} is an instance of the more
general problem of stochastic weakly convex optimization, where $f(x) \defeq
\E_P[f(x; \statrv)]$ and for each $x_0$ and $\statval \in \statdomain$,
there is $\lambda(\statval, x_0)$ such that $x \mapsto f(x; \statval) +
\frac{\lambda(\statval, x_0)}{2} \norm{x - x_0}^2$ is convex in a
neighborhood of $x_0$.  (We show later how
problem~\eqref{eqn:convex-composite-stochastic} falls in this framework.)
Such functions have classical and modern applications in
optimization~\cite{Ermoliev69, Drusvyatskiy18, PoliquinRo96,
  RockafellarWe98}, for example, in phase
retrieval~\cite{DuchiRu18a} problems or
training deep linear neural networks (e.g.~\cite{HardtMa17}).
We thus study the problem
\begin{equation}
  \label{eqn:problem}
  \begin{split}
    \minimize_x ~~
    & f(x) + \regularizer(x)
    = \E_P[f(x; \statrv)] + \regularizer(x) \\
    \subjectto ~~& x\in \xdomain, 
  \end{split}
\end{equation}
where $\xdomain \subset \R^d$ is a closed convex set and $\regularizer :
\R^d \to \R$ is a closed convex function.

Many problems are representable in the
form~\eqref{eqn:convex-composite-stochastic}. Taking the function $c$ as the
identity mapping, classical regularized stochastic convex optimization
problems fall into this framework~\cite{NemirovskiJuLaSh09}, including
regularized least-squares and the Lasso~\cite{HastieTiFr09, Tibshirani96b},
with $\statval = (a, b) \in \R^d \times \R$ and $h(x; \statval) = \half (a^T
x - b)^2$ and $\varphi$ typically some norm on $x$, or supervised learning
objectives such as logistic regression or support vector
machines~\cite{HastieTiFr09}.  The more general
settings~(\ref{eqn:convex-composite-stochastic}--\ref{eqn:problem}) include
a number of important non-convex problems. Examples include non-linear least
squares~\cite[cf.][]{NocedalWr06}, with $\statval = (a, b)$ and $b \in \R$,
the convex term $h(t; \statval) \equiv h(t) = \half t^2$ independent of the
sampled $\statval$, and $c(x; \statval) = c_0(x; a) - b$ where $c_0$ is some
smooth function a modeler believes predicts $b$ well given $x$ and data $a$.
Another compelling example is the (robust) phase retrieval
problem~\cite{CandesStVo13,SchechtmanElCoChMiSe15}---which we explore in
more depth in our numerical experiments---where the data $\statval = (a, b)
\in \R^d \times \R_+$, $h(t; \statval) \equiv h(t) = |t|$ or $h(t; \statval)
\equiv h(t) = \half t^2$, and $c(x; \statval) = (a^T x)^2 - b$. In the case
that $h(t) = |t|$, the form~\eqref{eqn:convex-composite-stochastic} is an
exact penalty for the solution of a collection of quadratic equalities
$(a_i^T x)^2 = b_i$, $i = 1, \ldots, N$, where we take $P$ to be point
masses on pairs $(a_i, b_i)$.

Fletcher and Watson~\cite{FletcherWa80, Fletcher82} initiated work on
composite problems, to
\begin{equation}
  \label{eqn:convex-composite-deterministic}
  \minimize_x ~~ h(c(x)) + \regularizer(x), ~
  \subjectto ~x\in \xdomain
\end{equation}
for fixed convex $h$, smooth $c$, convex $\regularizer$ and convex
$\xdomain$. A motivation of this earlier work is nonlinear programming
problems with the constraint that $x \in \{x : c(x) = 0\}$, in
which case taking $h(z) = \norm{z}$ functions
as an exact penalty~\cite{HiriartUrrutyLe93ab} for the constraint
$c(x) = 0$.  A more recent line of work, beginning with \citet{Burke85} and
continued by (among others) Druvyatskiy, Ioffe, Lewis, Pacquette, and
Wright~\cite{LewisWr08, DrusvyatskiyIoLe16, DrusvyatskiyLe18,
  DrusvyatskiyPa16}, establishes convergence
rate guarantees for methods that sequentially minimize
convex surrogates for
problem~\eqref{eqn:convex-composite-deterministic}.

Roughly, these papers construct a model of the composite function $f(x) =
h(c(x))$ as follows. Letting $\nabla c(x)$ be the transpose of the Jacobian
of $c$ at $x$, so $c(y) = c(x) + \nabla c(x)^T (y - x) +
o(\norm{y - x})$, one defines the ``linearized'' model of $f$ at $x$ by
\begin{equation}
  \label{eqn:linear-f}
  f_x(y) \defeq h(c(x) + \nabla c(x)^T (y - x)),
\end{equation}
which is convex in $y$. When $h$ and $\nabla c$ are Lipschitzian,
then $|f_x(y) - f(x)| = O(\norm{x - y}^2)$, so that the
model~\eqref{eqn:linear-f} is second-order accurate, which motivates the
following \emph{prox-linear} method. Beginning from some $x_0 \in \xdomain$,
iteratively construct
\begin{equation}
  \label{eqn:prox-linear-step-deterministic}
  x_{k + 1} = \argmin_{x \in \xdomain}
  \left\{ f_{x_k}(x) + \regularizer(x) + \frac{1}{2 \stepsize_k}
  \norm{x - x_k}^2 \right\},
\end{equation}
where $\stepsize_k > 0$ is a stepsize that may be chosen by a
line-search. For small $\stepsize_k$, the
iterates~\eqref{eqn:prox-linear-step-deterministic} guarantee decreasing
$h(c(x_k)) + \regularizer(x_k)$, the sequence of
problems~\eqref{eqn:prox-linear-step-deterministic} are convex, and
moreover, the iterates $x_k$ converge to stationary points of
problem~\eqref{eqn:convex-composite-deterministic}
\cite[\S~5]{DrusvyatskiyLe18}. The prox-linear method is effective so long
as minimizing the models $f_{x_k}(x)$ is reasonably computationally easy.
More generally, minimizing a sequence of models $f_{x_k}$ of $f$ centered
around the iterate $x_k$ is natural, with examples including Rockafellar's
proximal point algorithm~\cite{Rockafellar76} and general sequential convex
programming approaches, such as trust region and other Taylor-like
methods~\cite{NocedalWr06,ConnGoTo00,DrusvyatskiyLe18,Drusvyatskiy18}.

In our problem~\eqref{eqn:problem}
where $f(x) = \E[f(x; \statrv)]$ for $f(\cdot, \statval)$ weakly convex
or composite, the
iterates~\eqref{eqn:prox-linear-step-deterministic} may be computationally
challenging. Even in the case in which $P$ is discrete so that
problem~\eqref{eqn:convex-composite-stochastic} has the form $f(x) =
\frac{1}{n} \sum_{i = 1}^n h_i(c_i(x))$, which is evidently of the
form~\eqref{eqn:convex-composite-deterministic}, the iterations generating
$x_k$ may be prohibitively expensive for large $n$. When $P$ is continuous 
or is unknown, because we can only simulate draws $\statrv \sim P$
or in statistical settings where the only access to $P$ is via
observations $\statrv_i \sim P$, then the
iteration~\eqref{eqn:prox-linear-step-deterministic} is essentially
infeasible. Given the wide applicability of the stochastic composite
problem~\eqref{eqn:convex-composite-stochastic}, it is of
substantial interest to develop efficient online and stochastic methods to
(approximately) solve it, or at least to find local optima.


In this work, we develop and study stochastic model-based algorithms,
examples of which include a stochastic linear proximal algorithm, which is a
stochastic analogue of problem~\eqref{eqn:prox-linear-step-deterministic},
and a stochastic subgradient algorithm, both of whose definitions we give in
Section~\ref{sec:algorithm}. The iterations of such methods are often
computationally simple, and they require only individual samples $\statrv
\sim P$ at each iteration.  Consider for concreteness the case when $P$ is
discrete and supported on $i = 1, \ldots, n$ (i.e.\ $f(x) = \frac{1}{n}
\sum_{i=1}^n h_i(c_i(x))$). Then instead of solving the non-trivial
subproblem~\eqref{eqn:prox-linear-step-deterministic}, the stochastic
prox-linear algorithm samples $i_0 \in [n]$ uniformly, then substitutes
$h_{i_0}$ and $c_{i_0}$ for $h$ and $c$ in the iteration. Thus, as long as
there is a prox-linear step for the individual compositions $h_i \circ c_i$,
the algorithm is easy to implement and execute.


The main result of this paper is that the stochastic model-based methods we
develop for the stochastic composite and weakly-convex optimization problems
are convergent. More precisely, assuming that (i) with probability one, the
iterates of the procedures are bounded, (ii) the objective function $F +
\bigindic_X$ is coercive and (iii) second moment conditions on
local-Lipschitzian and local-convexity parameters of the random functions
$f(\cdot, \statval)$, any appropriate model-based stochastic minimization
strategy has limit points taking values $f(x)$ in the set of stationary
values of the function. If the image of non-critical points of the objective
function is dense in $\R$, the methods converge to stationary points of the
(potentially) non-smooth, non-convex objective~\eqref{eqn:problem}
(Theorem~\ref{theorem:informal-theorem-stationary-cluster-points} in
Sec.~\ref{sec:algorithm} and
Theorem~\ref{theorem:stationary-cluster-points-extension} in
Sec.~\ref{sec:as-convergence}). As gradients $\nabla f(x)$ may not exist
(and may not even be zero at stationary points because of the non-smoothness
of the objective), demonstrating this convergence provides some challenge.
To circumvent these difficulties, we show that the iterates are
asymptotically equivalent to the trajectories of a particular ordinary
differential inclusion~\cite{AubinCe84} (a non-smooth generalization of
ordinary differential equations (ODEs)) related to
problem~\eqref{eqn:convex-composite-stochastic}, building off of the
classical ODE method~\cite{Ljung77,KushnerYi03,Borkar08} (see
Section~\ref{sec:functional-convergence}). By developing a number of
analytic properties of the limiting differential inclusion using the weak
convexity of $f$, we show that trajectories of the ODE must converge
(Section~\ref{sec:diff-inclusion-properties}). A careful stability analysis
then shows that limit properties of trajectories of the ODE are preserved
under small perturbations, and viewing our algorithms as noisy discrete
approximations to a solution of the ordinary differential inclusion gives
our desired convergence (Section~\ref{sec:as-convergence}).

Our results do not provide rates of convergence for the stochastic
procedures, so to investigate the properties of the methods we propose, we
perform a number of numerical simulations in Section~\ref{sec:experiments}.
We focus on a discrete version of
problem~\eqref{eqn:convex-composite-stochastic} with the robust phase
retrieval objective $f(x; a, b) = |(a^T x)^2 - b|$, which facilitates
comparison with deterministic
methods~\eqref{eqn:prox-linear-step-deterministic}. Our experiments
extend our theoretical predictions, showing the advantages of
stochastic over deterministic procedures for some
problems, and they also show that the stochastic prox-linear method may be
preferable to stochastic subgradient methods because of robustness
properties it enjoys (which our simulations verify, though our theory does
not yet explain).

\paragraph{Related and subsequent work}
The stochastic subgradient method has a substantial history.  Early work
due to Ermoliev and Norkin~\cite{Ermoliev69,
  ErmolievNo98, ErmolievNo03}, \citet{Gupal79}, and \citet{Dorofeyev85}
identifies the basic assumptions sufficient for stochastic gradient methods
to be convergent. \citet{Ruszczynski87} provides a convergent gradient
averaging-based optimization scheme for stochastic weakly convex
problems. Our analytical approach is based on differential equations and
inclusions, which have a long history in the study of stochastic
optimization
methods, where researchers have used a limiting differential equation
or inclusion to exhibit convergence of stochastic approximation
schemes~\cite{Ljung77,AubinCe84,Kunze00,Borkar08}; more recent work
uses differential equations to model accelerated gradient
methods~\cite{SuBoCa14, WibisonoWiJo16}.  Our approach gives similar
convergence results to those for stochastic subgradient methods, but allows
us to study and prove convergence for a more general collection of
model-based minimization strategies.  Our results do not provide convergence
rates, which is possible when the compositional
structure~\eqref{eqn:convex-composite-stochastic} leaves the problem
\emph{convex}~\cite{WangLiFa16, WangFaLi17}; the problems we consider are
typically non-smooth and non-convex, so that these approaches do not apply.

Subsequent to the initial appearance of the current paper on the
\texttt{arXiv} and inspired by our work,\footnote{based on personal
  communication with Damek Davis and Dmitriy Drusvyatskiy} Davis,
Drusvyatskiy, and Grimmer have provided convergence rates for variants of
stochastic subgradient, prox-linear, and related methods~\cite{DavisGr17,
  DavisDr18, DavisDr18a}.  Here they show that the methods we develop in
this paper satisfy non-asymptotic convergence guarantees.  To make this
precise, let $F_\lambda(x) = \inf_{y \in X} \{f(y) + \varphi(y) +
\frac{\lambda}{2} \norm{y - x}^2\}$ be the Moreau envelope of the
objective~\eqref{eqn:problem}, which is continuously differentiable and for
which $\nabla F_\lambda$ being small is a proxy for near-stationarity of $x$
(see~\cite{Drusvyatskiy18, DavisDr18, DavisDr18a}). Then they show that,
with appropriate stepsizes, they can construct a (random) iterate
$\what{x}_k$ such that $\E[\norm{\nabla F_\lambda(\what{x}_k)}^2] = O(1 /
\sqrt{k})$. These convergence guarantees extend the
with probability 1 convergence results we provide.

\paragraph{Notation and basic definitions}
We collect here our (mostly standard) notation and basic definitions that we
require. For $x, y\in \R$, we let $x \wedge y = \min\{x, y\}$.  We let
$\ball$ denote the unit $\ell_2$-ball in $\R^d$, where $d$ is
apparent from context, and $\norm{\cdot}$ denotes the operator $\ell_2$-norm
(the standard Euclidean norm on vectors).  For a set $A \subset
\R^d$ we let $\norm{A} = \sup_{a \in A} \norm{a}$.  We say $f : \R^d \to \R
\cup \{+\infty\}$ is \emph{$\lambda$-weakly convex} (also known as
lower-$\mc{C}^2$ or semiconvex~\cite{RockafellarWe98, BolteDaLeMa10}) near
$x$ if there exists $\epsilon > 0$ such that for all $x_0 \in \R^d$,
\begin{equation}
  \label{eqn:weak-convexity}
  y \mapsto f(y) + \frac{\lambda}{2} \ltwo{y - x_0}^2,
  ~~ y \in x + \epsilon \ball
\end{equation}
is convex (the vector $x_0$ is immaterial in~\eqref{eqn:weak-convexity}, as
holding at one $x_0$ is equivalent)~\cite[Ch.~10.G]{RockafellarWe98}.  For a
function $f : \R^d \to \R \cup \{+\infty\}$, we let $\partial f(x)$ denote
the Fr\'{e}chet (or regular~\cite[Ch.~8.B]{RockafellarWe98}) subdifferential
of $f$ at the point $x$,
\begin{equation*}
  \partial f(x) \defeq \left\{g \in \R^d :
  f(y) \ge f(x) + \<g, y - x\> + o(\norm{y - x})
  ~ \mbox{as~} y \to x \right\}.
\end{equation*}
The Fr\'{e}chet subdifferential and subdifferential coincide for
convex $f$~\cite[Ch.~8]{RockafellarWe98},
and for weakly convex $f$, $\partial f(x)$ is non-empty
for $x$ in the relative interior of $\dom f$.
The Clarke
directional derivative of a function $f$ at the point $x$ in direction $v$
is
\begin{equation*}
  f'(x; v) \defeq \liminf_{t \downarrow 0, v' \to v} \frac{f(x + tv) - f(x)}{t},
\end{equation*}
and recall~\cite[Ex.~8.4]{RockafellarWe98} that
$\partial f(x) = \{w \in \R^d : \<v, w\> \le f'(x; v) ~ \mbox{for~all~} v\}$.

\newcommand{\fndist}{\mathsf{d}}

We let $\continuous(A, B)$ denote the continuous functions from
$A$ to $B$. Given a sequence of functions $f_n : \R_+ \to \R^d$, we say that
$f_n \to f$ in $\continuous(\R_+, \R^d)$ if $f_n \to f$ uniformly on
all compact sets, that is, for all $T < \infty$ we have
\begin{equation*}
  \lim_{n \to \infty} \sup_{t \in [0, T]} \norm{f_n(t) - f(t)} = 0.
\end{equation*}
This is equivalent to convergence in
$\fndist(f, g) \defeq \sum_{t = 1}^\infty 2^{-t} \sup_{\tau \in [0, t]}
\norm{f(\tau) - g(\tau)} \wedge 1$,
which shows the standard result that $\continuous(\R_+, \R^d)$ is a
Fr\'{e}chet space.
For a closed convex set $X$, we let $\bigindic_X$ denote the $+\infty$-valued
indicator for $X$, that is, $\bigindic_X(x) = 0$ if $x \in X$ and $+\infty$
otherwise.
The normal cone to $X$ at $x$ is
\begin{equation*}
  \normalcone_X(x) \defeq \{v \in \R^d : \<v, y - x\> \le 0 ~ \mbox{for~all~}
  y \in X \}.
\end{equation*}
For closed convex $C$, $\project_C(x) \defeq \argmin_{y \in C} \norm{y -
  x}$ denotes projection of $x$ onto $C$.


\section{Algorithms and Main Convergence Result}
\label{sec:algorithm}

In this section, we introduce the family of algorithms we study
for problem~\eqref{eqn:problem}. In analogy with
the update~\eqref{eqn:prox-linear-step-deterministic},
we first give a general form
of our model-based approach, then exhibit three examples
that fall into the broad scheme.
We iterate
\begin{equation}
  \begin{split}
    & \text{Draw}~~ \statrv_k \simiid P \\
    & x_{k + 1} \defeq \argmin_{y \in \xdomain}
    \left\{f_{x_k}(y; \statrv_k) + \regularizer(y)
    + \frac{1}{2 \stepsize_k} \norm{y - x_k}^2 \right\}.
  \end{split}
  \label{eqn:model-based-minimization}
\end{equation}
In the iteration~\eqref{eqn:model-based-minimization}, the function
$f_{x_k}(\cdot; \statval)$ is an approximation, or model, of $f(\cdot;
\statval)$ at the point $x_k$,
and $\stepsize_k > 0$ is a stepsize sequence.


For the model-based strategy~\eqref{eqn:model-based-minimization} to be
effective, we require that $f_x(\cdot; \statval)$ satisfy a few essential
properties on its approximation quality.
\begin{enumerate}[label=C.(\roman*),leftmargin=*]
\item \label{item:convex-model}
  The function $y \mapsto f_x(y; \statval)$ is convex
  and subdifferentiable on its domain
\item \label{item:equal-model}
  We have $f_x(x; \statval) = f(x; \statval)$
\item \label{item:subgrad-model}
  At $y = x$ we have the containment
  \begin{equation*}
    \left.\partial_y f_x(y; \statval)\right|_{y = x}
    \subset \partial_x f(x; \statval).
  \end{equation*}
\end{enumerate}
\noindent
In addition to
conditions~\ref{item:convex-model}--\ref{item:subgrad-model},
we require one additional technical condition on the models,
which quantitatively guarantees they locally
almost underestimate $f$.
\begin{enumerate}[label=C.(\roman*),leftmargin=*]
  \setcounter{enumi}{3}
\item \label{item:upper-approximation}
  There exists $\epsilon_0 > 0$ such that $0 < \epsilon \le \epsilon_0$
  implies that for all
  $x_0 \in \xdomain$ there exists $\delta_\epsilon(x_0; \statval) \ge 0$
  with
  \begin{equation*}
    f(y; \statval) \ge f_x(y; \statval) - \half
    \delta_\epsilon(x_0;\statval) \norm{y - x}^2
  \end{equation*}
  for $x, y \in x_0 + \epsilon \ball$, where $\E[\delta_\epsilon(x_0;
    \statrv)] < \infty$.
\end{enumerate}

\subsection{Examples}
\label{sec:example}
We give four example algorithms for
problems~\eqref{eqn:convex-composite-stochastic} and~\eqref{eqn:problem},
each of which consists of a local model $f_x$ satisfying
conditions~\ref{item:convex-model}--\ref{item:upper-approximation}.  The
conditions~\ref{item:convex-model}--\ref{item:subgrad-model} are immediate,
while we defer verification of condition~\ref{item:upper-approximation} to
after the statement of
Theorem~\ref{theorem:informal-theorem-stationary-cluster-points}. The first
example is the natural generalization of the classical subgradient
method~\cite{Ermoliev69}.

\begin{example}[Stochastic subgradient method]
  \label{example:sgd}
  For this method,
  we let $\subgrad(x; \statval) \in \partial f(x; \statval)$ be a
  (fixed)
  element of the Fr\'{e}chet subdifferential of $f(x; \statval)$;
  in the case of the composite objective~\eqref{eqn:convex-composite-stochastic}
  this is $\subgrad(x; \statval) \in \nabla c(x; \statval)
  \partial h(c(x;\statval);\statval)$.
  Then the model~\eqref{eqn:model-based-minimization} for
  the stochastic (regularized and projected) subgradient
  method is
  \begin{equation*}
    f_{x}(y; \statval) \defeq f(x; \statval)
    + \<\subgrad(x; \statval), y - x\>.
  \end{equation*}
  The properties~\ref{item:convex-model}--\ref{item:subgrad-model}
  are immediate.
\end{example}

The stochastic prox-linear method applies to the structured family of
convex composite problems~\eqref{eqn:convex-composite-stochastic},
generalizing the deterministic prox-linear method~\cite{Burke85,
  DrusvyatskiyLe18, DrusvyatskiyPa16}.

\begin{example}[Stochastic prox-linear method]
  \label{example:stochastic-prox-linear}
  Here, we have $f(x;\statval) = h(c(x; \statval); \statval)$,
  and in analogy to the update~\eqref{eqn:linear-f}
  we linearize $c$ without modifying $h$,
  defining
  \begin{equation*}
    f_x(y; \statval) \defeq h(c(x; \statval)
    + \nabla c(x; \statval)^T(y - x); \statval)
  \end{equation*}
  Again, conditions~\ref{item:convex-model}--\ref{item:subgrad-model}
  are immediate.
\end{example}

Lastly, we have stochastic proximal point methods for
weakly-convex functions.

\begin{example}[Stochastic proximal-point method]
  \label{example:stochastic-prox-point}
  We assume that the instantaneous function $f(\cdot; \statval)$ is
  $\lambda(\statval)$-weakly convex over $\xdomain$. In this case, for the
  model in the update~\eqref{eqn:model-based-minimization}, we set
  $f_x(y; \statval) = f(y; \statval) + \frac{\lambda(\statval)}{2}
  \norm{y - x}^2$.
\end{example}

\begin{example}[Guarded stochastic proximal-point method]
  \label{example:guarded-stochastic-prox-point}
  We assume that for some $\epsilon > 0$
  and all $x \in \xdomain$,
  the instantaneous function $f(\cdot; \statval)$ is
  $\lambda(\statval, x)$-weakly convex over $\xdomain
  \cap \{x + \epsilon \ball\}$. In this case, for the
  model in the update~\eqref{eqn:model-based-minimization}, we set
  \begin{equation}
    f_x(y; \statval) = f(y; \statval) + \frac{\lambda(\statval, x)}{2}
    \norm{y - x}^2 + \bigindic_{x + \epsilon \ball}(y),
    \label{eqn:guarded-stochastic-prox-point}
  \end{equation}
  which restricts the domain of the model function $f_x(\cdot; \statval)$ to
  a neighborhood of $x$ so that the
  update~\eqref{eqn:model-based-minimization} does not escape the region of
  convexity.  Again, by inspection, this satisfies
  conditions~\ref{item:convex-model}--\ref{item:subgrad-model}.
\end{example}

\subsection{The main convergence result}

The main theoretical result of this paper is to show that stochastic
algorithms based on the update~\eqref{eqn:model-based-minimization} converge
almost surely to the stationary points of the objective function $F(x) =
f(x) + \regularizer(x)$ over $\xdomain$. To state our results formally,
for $\epsilon > 0$ we
we define the function $\lipf_\epsilon : \xdomain \times \statdomain \to \R_+$
by
\begin{equation*}
  \lipf_\epsilon(x; \statval)
  \defeq \sup_{y \in \xdomain, \norm{y - x} \le \epsilon}
  \sup_{g \in \partial f(y; \statval)} \norm{g}.
\end{equation*}
We then make the following local Lipschitzian and convexity
assumptions on 
$f(\cdot; \statval)$.
\begin{assumption}
  \label{assumption:local-lipschitz}
  There exists $\epsilon_0 > 0$ such that
  $0 < \epsilon \le \epsilon_0$ implies that
  \begin{equation*}
    \E[\lipf_\epsilon(x; \statrv)^2] < \infty
    ~~ \mbox{for~all~} x \in \xdomain.
  \end{equation*}
\end{assumption}
\begin{assumption}
  \label{assumption:weak-convexity}
  There exists $\epsilon_0 > 0$ such that $0 < \epsilon \le \epsilon_0$
  implies that for all $x \in \xdomain$, there exists
  $\lambda(\statval, x) \ge 0$ such that
  \begin{equation*}
    y \mapsto f(y; \statval) + \frac{\lambda(\statval, x)}{2}
    \norm{y - x_0}^2
  \end{equation*}
  is convex on the set $x + \epsilon\ball$ for any $x_0$, and
  $\E[\lambda(\statrv, x)] < \infty$.
\end{assumption}
\noindent
As we shall see in Lemma~\ref{lemma:frechet-subgradients-of-f} later,
Assumptions~\ref{assumption:local-lipschitz}
and~\ref{assumption:weak-convexity} are sufficient to guarantee that
$\partial f(x)$ exists, is non-empty for all $x \in \xdomain$, and is outer
semi-continuous. In addition, it is immediate that for any
$\lambda \ge \E[\lambda(\statrv, x)]$, the function $f$ is $\lambda$-weakly
convex~\eqref{eqn:weak-convexity} on the $\epsilon$-ball around $x$.

With the assumptions in place, we can
now proceed to a (mildly) simplified version of our main result in this
paper. Let $\stationary$ denote the set of stationary points for the
objective function $F(x) = f(x) + \regularizer(x)$ over $\xdomain$.
Lemma~\ref{lemma:frechet-subgradients-of-f} to come implies that
$\partial F(x) = \partial f(x) + \partial \regularizer(x)$ for all
$x\in \xdomain$, so we can represent $\stationary$ as
\begin{equation}
  \label{eqn:stationary}
  \stationary
  \defeq
  \left\{x \in \xdomain \mid
  \exists g \in \partial f(x) + \partial \regularizer(x)
  ~ \mbox{with}~
  \<g, y - x\> \ge 0 ~ \mbox{for~all~} y \in \xdomain \right\}.
\end{equation}
Equivalently, $\partial f(x) + \partial \regularizer(x) \cap
-\normalcone_\xdomain(x) \neq \emptyset$, or $0 \in \partial f(x) + \partial
\regularizer(x) + \normalcone_\xdomain(x)$.
Important for us is the \emph{image} of the set of stationary
points, that is,
\begin{equation*}
  F(\stationary) \defeq \left\{f(x) + \regularizer(x)
  \mid x \in \stationary\right\}.
\end{equation*}

With these definitions, we have the following convergence result, which is a
simplification of our main convergence result,
Theorem~\ref{theorem:stationary-cluster-points-extension}, which we present in
Section~\ref{sec:as-convergence}.
\begin{theorem}
  \label{theorem:informal-theorem-stationary-cluster-points}
  Let Assumptions~\ref{assumption:local-lipschitz}
  and~\ref{assumption:weak-convexity} hold and assume
  $\xdomain$ is compact. Let $x_k$ be generated by any model-based update
  satisfying conditions
  \ref{item:convex-model}--\ref{item:upper-approximation} with
  stepsizes $\stepsize_k > 0$ satisfying
  $\sum_k \stepsize_k = \infty$ and $\sum_k \stepsize_k^2 < \infty$.
  Then with
  probability 1,
  \begin{equation}
    \label{eqn:subset-containment-limits}
    \big[\liminf_k F(x_k), \limsup_k F(x_k)\big]
    \subset F(\stationary).
  \end{equation}
\end{theorem}

We provide a few remarks on the theorem, as well as elucidating our
examples~\ref{example:sgd}--\ref{example:guarded-stochastic-prox-point} in
this context.  The limiting inclusion~\eqref{eqn:subset-containment-limits}
is familiar from the classical literature on stochastic subgradient
methods~\cite{Dorofeyev85,ErmolievNo98}, though in our case, it applies to
the broader family of model-based
updates~\eqref{eqn:model-based-minimization}, including
Examples~\ref{example:sgd}--\ref{example:guarded-stochastic-prox-point}.

To see that the theorem indeed applies to each of these
examples, we must verify Condition~\ref{item:upper-approximation}.
For Examples~\ref{example:sgd}, \ref{example:stochastic-prox-point},
and~\ref{example:guarded-stochastic-prox-point}, this is immediate
by taking the lower approximation function $\delta_\epsilon(x;
\statval) = \lambda(\statval, x)$ from Assumption~\ref{assumption:weak-convexity},
yielding the following
\begin{observation}
  Let Assumption~\ref{assumption:weak-convexity} hold. Then
  Condition~\ref{item:upper-approximation} holds for each of
  Examples~\ref{example:sgd}, \ref{example:stochastic-prox-point},
  and~\ref{example:guarded-stochastic-prox-point}.
\end{observation}

We also provide conditions on the composite optimization
problem~\eqref{eqn:convex-composite-stochastic}, that is, when $f(x;
\statval) = h(c(x;\statval);\statval)$, sufficient for
Assumptions~\ref{assumption:local-lipschitz}--\ref{assumption:weak-convexity}
and Condition~\ref{item:upper-approximation} to hold. Standard
results~\cite{DrusvyatskiyIoLe16} show that $\partial f(x;
\statval) = \nabla c(x; \statval) \partial h(c(x;\statval);\statval)$, so
Assumption~\ref{assumption:local-lipschitz} holds if $\sup_{\norm{y - x} \le
  \epsilon} \norm{\nabla c(x;\statval) \partial h(c(x;\statval);\statval)}$
is integrable (with respect to $\statval$).  For
Assumption~\ref{assumption:weak-convexity}, we assume
that there exists $\epsilon_0 > 0$ such that
if $0 < \epsilon \le \epsilon_0$,
there exist functions $\liph_\epsilon : \R^d \times \statdomain \to \R_+
\cup \{+\infty\}$ and $\lipc_\epsilon : \R^d \times \statdomain \to \R_+
\cup \{+\infty\}$ such that $c(\cdot; \statval)$ has $\lipc_\epsilon(x;
\statval)$-Lipschitz gradients in an $\epsilon$ neighborhood of $x$, that
is,
\begin{equation*}
  \opnorm{\nabla c(y; \statval) - \nabla c(y'; \statval)}
  \le \lipc_\epsilon(x, \statval) \norm{y - y'}
  ~~ \mbox{for~} \norm{y - x}, \norm{y' - x} \le \epsilon,
\end{equation*}
and that $h(\cdot; \statval)$ is
$\liph_\epsilon(x; \statval)$-Lipschitz continuous on the compact
convex neighborhood
\begin{equation*}
  \conv \Big\{
  c(y; \statval) + \nabla c(y; \statval)^T (z - y)
  + v \mid
  y, z \in x + \epsilon \ball,
  \norm{v} \le \frac{\lipc_\epsilon(x, \statval)}{2}
  \norm{y - z}^2 \Big\}.
\end{equation*}
We then have the following claim;
see Appendix~\ref{sec:proof-composite-weak-convexity} for a proof.
\begin{claim}
  \label{claim:composite-weak-convexity}
  If $\E[\liph_\epsilon(x;\statrv) \lipc_\epsilon(x; \statrv)]
  < \infty$ for all $x \in \xdomain$, then
  Assumption~\ref{assumption:weak-convexity} holds
  with $\lambda(\statval,x) = \liph_\epsilon(x;\statval)
  \lipc_\epsilon(\statval)$, and
  Condition~\ref{item:upper-approximation}
  holds with $\delta_\epsilon(x;\statval)
  = \liph_\epsilon(x;\statval) \lipc_\epsilon(x;\statval)$.
\end{claim}

Theorem~\ref{theorem:informal-theorem-stationary-cluster-points}
does not guarantee convergence of the iterates, though it does guarantee
cluster points of $\{x_k\}$ have limiting \emph{values} in the image of the
stationary set. A slightly stronger technical assumption, which rules out
pathological functions such as Whitney's construction~\cite{Whitney35}, is
the following assumption, which is related to Sard's results that the
measure of critical values of $\mc{C}^d$-smooth $f : \R^d \to \R$ is zero.
\begin{assumption}
  \label{assumption:countable-regularity}
  The set $(F(\stationary))^c$ is dense in $\R$.
\end{assumption}
\noindent
If $f$ is convex then $(f + \regularizer)(\stationary)$ is a singleton.
Moreover, if the set of stationary points $\stationary$ consists of a
(finite or countable) collection of sets $\stationary_1, \stationary_2,
\ldots$ such that $f + \regularizer$ is constant on each $\stationary_i$,
then $F(\xdomain\opt)$ is at most countable and
Assumption~\ref{assumption:countable-regularity} holds.  In subsequent work
to the first version of this paper, \citet{DavisDrKaLe18} give sufficient
conditions for Assumption~\ref{assumption:countable-regularity} to hold (see
also~\cite{Ioffe08,Ioffe17,BolteDaLeSh07}).  We have
\begin{corollary}
  \label{corollary:informal}
  In addition to the conditions of
  Theorem~\ref{theorem:informal-theorem-stationary-cluster-points}, let
  Assumption~\ref{assumption:countable-regularity}
  hold. Then $f(x_k) + \regularizer(x_k)$ converges, and all cluster points
  of the sequence $\{x_k\}$ belong to $\stationary$.
\end{corollary}


\section{Convergence Analysis of the Algorithm}
\label{sec:convergence-analysis}

In this section, we present the arguments necessary to prove
Theorem~\ref{theorem:informal-theorem-stationary-cluster-points} and its
extensions, beginning with a heuristic explanation.  By inspection and a
strong faith in the limiting behavior of random iterations, we expect that
the update scheme~\eqref{eqn:model-based-minimization}, as the stepsizes
$\stepsize_k \to 0$, are asymptotically approximately equivalent to
iterations of the form
\begin{equation*}
  \frac{x_{k + 1} - x_k}{\stepsize_k}
  \approx - \left[\subgrad(x_k)
  + v_k + w_k\right],
  ~
  \subgrad(x_k) \in \partial f(x_k),
  ~ v_k \in \partial \regularizer(x_{k+1}),
  ~ w_k \in \normalcone_X(x_{k+1}),
\end{equation*}
and the correction $w_k$ enforces $x_{k + 1} \in X$.  As $k \to
\infty$ and $\stepsize_k \to 0$, we may (again, deferring rigor) treat
$\lim_k \frac{1}{\stepsize_k} (x_{k + 1} - x_k)$ as a continuous time
process, suggesting that update schemes of the
form~\eqref{eqn:model-based-minimization} are asymptotically equivalent to a
continuous time process $t \mapsto x(t) \in \R^d$ that satisfies the
differential inclusion (a set-valued generalization of an ordinary
differential equation)
\begin{equation}
  \label{eqn:differential-inclusion}
    \dot{x}
    \in
    -\partial f(x)
    - \partial \regularizer(x)
    - \normalcone_X(x) \\
    = - \int \partial f(x; \statval) dP(\statval)
    - \partial \regularizer(x)
    - \normalcone_X(x).
\end{equation}

We develop a general convergence result showing that this limiting
equivalence is indeed the case and that the second equality of
expression~\eqref{eqn:differential-inclusion} holds. As part of this, we
explore in the coming sections how the weak convexity structure of $f(\cdot;
\statval)$ guarantees that the differential
inclusion~\eqref{eqn:differential-inclusion} is well-behaved.
%
%
We begin in Section~\ref{sec:preliminaries-differential-inclusions} with
preliminaries on set-valued analysis and differential inclusions we require,
which build on standard convergence results~\cite{AubinCe84, Kunze00}. Once
we have presented these preliminary results, we show how the stochastic
iterations~\eqref{eqn:model-based-minimization} eventually approximate
solution paths to differential inclusions
(Section~\ref{sec:functional-convergence}), which builds off of a number of
stochastic approximation results and the so-called ``ODE method'' Ljung
develops~\cite{Ljung77}, (see also~\cite{KushnerYi03, BenaimHoSo05,
  Borkar08}).  We develop the analytic properties of the composite
objective, which yields the uniqueness of trajectories
solving~\eqref{eqn:differential-inclusion} as well as a particular Lyapunov
convergence inequality
(Section~\ref{sec:diff-inclusion-properties}). Finally, we develop stability
results on the differential inclusion~\eqref{eqn:differential-inclusion},
which allows us to prove convergence as in
Theorem~\ref{theorem:informal-theorem-stationary-cluster-points}
(Section~\ref{sec:as-convergence}).

\subsection{Preliminaries: differential inclusions and set-valued analysis}
\label{sec:preliminaries-differential-inclusions}

\newcommand{\edosc}{$\epsilon$-$\delta$~o.s.c.\xspace}

We now review a few results in set-valued analysis and
differential inclusions~\cite{AubinCe84,Kunze00}.
Our notation and definitions follow closely the references of
\citet{RockafellarWe98} and \citet{AubinCe84}, and we cite a few results
from the book of \citet{Kunze00}.

Given a sequence of sets $A_n \subset \R^d$, the limit supremum of the sets
consists of limit points of subsequences $y_{n_k} \in A_{n_k}$, that is,
\begin{equation*}
  \limsup_n A_n \defeq \left\{y : \exists ~ y_{n_k} \in A_{n_k}
  ~ \mbox{s.t.} ~ y_{n_k} \to y ~ \mbox{as}~ k \to \infty \right\}.
\end{equation*}
We let $G : X \toto \R^d$ denote a set-valued mapping $G$ from $X$ to $\R^d$,
and we define $\dom G \defeq \{x : G(x) \neq \emptyset\}$.  Then $G$ is
\emph{outer semicontinuous (o.s.c.)} if for any sequence $x_n \to x \in \dom G$, we
have $\limsup_n G(x_n) \subset G(x)$.  One says that
$G$ is \emph{$\epsilon$-$\delta$
  outer semicontinuous}~\cite[Def.~1.1.5]{AubinCe84} if for all $x$ and
$\epsilon > 0$, there exists $\delta > 0$ such that
$G(x + \delta \ball) \subset G(x) + \epsilon \ball$. These notions coincide
when $G(x)$ is bounded.
Two standard examples of outer-semicontinuous mappings follow.
\begin{lemma}[Hiriart-Urruty and Lemar\'echal~\cite{HiriartUrrutyLe93ab},
  Theorem VI.6.2.4]
  \label{lemma:subg-outer-semicontinuity}
  Let $f : \R^d \to \R \cup \{+\infty\}$ be convex. Then
  the subgradient mapping $\partial f : \interior \dom f \toto \R^d$ is
  o.s.c.
\end{lemma}
\begin{lemma}[Rockafellar and Wets~\cite{RockafellarWe98},
  Proposition 6.6]
  \label{lemma:normal-outer-semicontinuity}
  Let $X$ be a closed convex set. Then the normal cone mapping
  $\normalcone_\xdomain : X \toto \R^d$ is
  o.s.c.\ on $X$.
\end{lemma}

The \emph{differential inclusion} associated with $G$
beginning from the point $x_0$, denoted
\begin{equation}
  \label{eqn:simple-differential-inclusion}
  \dot{x} \in G(x), ~~ x(0) = x_0
\end{equation}
has a solution if there exists an absolutely continuous function
$x : \R_+ \to \R^d$ satisfying $\frac{d}{dt} x(t) = \dot{x}(t) \in G(x(t))$ for all
$t \ge 0$. For $G : \mc{T} \toto \R^d$ and
a measure $\mu$ on $\mc{T}$,
\begin{equation*}
  \int_{\mc{T}} G d\mu =
  \int_{\mc{T}} G(t) d\mu(t) \defeq \left\{\int g(t) d\mu(t)
  \mid g(t) \in G(t) ~ \mbox{for~}t \in \mc{T},
  ~ g ~ \mbox{measurable} \right\}.
\end{equation*}
With these definitions, the
following results (with minor extension) on the existence and uniqueness of
solutions to differential inclusions are standard.
\begin{lemma}[Aubin and Cellina~\cite{AubinCe84}, Theorem 2.1.4]
  \label{lemma:aubin-existence}
  Let $G : X \toto \R^d$ be outer semicontinuous and compact-valued, and
  $x_0 \in X$. Assume there is $K < \infty$
  such that $\dist(0, G(x)) \le K$ for all $x$.
  Then there exists an absolutely
  continuous function $x : \R_+ \to \R^d$ such that $\dot{x}(t) \in G(x(t))$ and
  $x(t) \in x_0 + \int_0^t G(x(\tau)) d\tau$ for all $t \in \R_+$.
\end{lemma}
\begin{lemma}[Kunze~\cite{Kunze00}, Theorem 2.2.2]
  \label{lemma:kunze-existence}
  Let the conditions of Lemma~\ref{lemma:aubin-existence} hold and assume
  there
  exists $c < \infty$ such that
  \begin{equation*}
    \<x_1 - x_2, g_1 - g_2\> \le c \norm{x_1 - x_2}^2
    ~~ \mbox{for~} g_i \in G(x_i) ~ \mbox{and~all~}
    x_i \in \dom G.
  \end{equation*}
  Then the solution to the differential
  inclusion~\eqref{eqn:simple-differential-inclusion}
  is unique.
\end{lemma}

We recall basic Lyapunov theory for differential inclusions.  Let $V : X \to
\R_+$ be a non-negative function and $W : X \times \R^d \to \R_+$ be
continuous with $v \mapsto W(x, v)$ convex in $v$ for all $x$. A trajectory
$\dot{x} \in G(x)$ is \emph{monotone} for the pair $V, W$ if
\begin{equation*}
  V(x(T)) - V(x(0)) + \int_0^T W(x(t), \dot{x}(t)) dt \le 0
  ~~ \mbox{for~} T \ge 0.
\end{equation*}
The next lemma gives sufficient conditions for the existence of
monotone trajectories.
\begin{lemma}[Aubin and Cellina~\cite{AubinCe84}, Theorem 6.3.1]
  \label{lemma:monotone-trajectory}
  Let $G : X \toto \R^d$ be outer semicontinuous and compact-convex valued.
  Assume that for each $x$
  there exists
  $v \in G(x)$ such that $V'(x; v) + W(x; v) \le 0$. Then there exists
  a trajectory of the differential inclusion $\dot{x} \in G(x)$ such that
  \begin{equation*}
    V(x(T)) - V(x(0)) + \int_0^T W(x(t), \dot{x}(t)) dt \le 0.
  \end{equation*}
\end{lemma}


Finally, we present a lemma on the
subgradients of $f$ using our set-valued integral definitions.
The proof is somewhat technical and not the main focus of this
paper, so we defer it to Appendix~\ref{sec:proof-frechet-subgradients-of-f}.
\begin{lemma}
  \label{lemma:frechet-subgradients-of-f}
  Let $f(\cdot; \statval)$ satisfy
  Assumptions~\ref{assumption:local-lipschitz}
  and~\ref{assumption:weak-convexity}.
  Then
  \begin{equation*}
    \partial f(x) = \E_P[\partial f(x; \statrv)]
  \end{equation*}
  and $\partial f(\cdot; \statval) : \R^d \toto \R^d$ and
  $\partial f(\cdot) : \R^d \toto \R^d$ are closed compact convex-valued
  and outer semicontinuous.
\end{lemma}
\noindent
Lemma~\ref{lemma:frechet-subgradients-of-f} shows that
$\partial f(x; \statval)$ is compact-valued and o.s.c., and
we thus define the shorthand notation for the
subgradients of $f + \regularizer$ as
\begin{equation}
  \label{eqn:actual-subgradient-sets}
  G(x; \statval) \defeq \partial f(x; \statval) + \partial \regularizer(x)
  ~~ \mbox{and} ~~
  G(x) \defeq
  \int_\statdomain \partial f(x; \statval) dP(\statval)
  + \partial \regularizer(x),
\end{equation}
both of which are o.s.c.\ in $x$ and compact-convex valued
because $\regularizer$ is convex.

\subsection{Functional convergence of the iteration path}
\label{sec:functional-convergence}

With our preliminaries in place, we now establish a general functional
convergence theorem (Theorem~\ref{theorem:functional-convergence-general})
that applies to stochastic approximation-like algorithms that asymptotically
approximate differential inclusions. By showing the generic
algorithm~\eqref{eqn:model-based-minimization} has the form our theorem
requires, we conclude that each of
examples~\ref{example:sgd}--\ref{example:guarded-stochastic-prox-point}
converge to the appropriate
differential inclusion (Sec.~\ref{sec:stochastic-approximation-limit}).

\subsubsection{A general functional convergence theorem}
\label{sec:general-functional-convergence}

Let $\{g_k\}_{k \in \N}$ be a collection of set-valued mappings
$g_k : \R^d \toto \R^d$, $\{\stepsize_k\}_{k \in \N}$
be a sequence of positive stepsizes, 
$\{\noise_k\}_{k=1}^\infty$ be an
arbitrary $\R^d$-valued sequence (the noise sequence).
Consider the following iteration, which begins
from the initial value $x_0 \in \R^d$:
\begin{equation}
  \label{eqn:stochastic-approximation}
  x_{k+1} = x_k + \stepsize_k [y_k + \noise_{k+1}],
  ~~ \mbox{where}~ y_k \in g_k(x_k)
  ~~ \mbox{for~} k \ge 0.
\end{equation}
For notational convenience, define the
``times'' $t_m = \sum_{k=1}^m \stepsize_k$ as the partial
stepsize sums, and let $x(\cdot)$ be the linear 
interpolation of the iterates $x_k$:
\begin{equation}
  \label{eqn:interpolation}
  x(t) \defeq x_k + \frac{t-t_k}{t_{k+1} - t_k} (x_{k+1} - x_k)
  ~~\text{and}~~y(t) = y_k ~~\text{for}~~t\in [t_k, t_{k+1}).
\end{equation}
This path satisfies $\dot{x}(t) = y(t)$ for almost all $t$ and is
absolutely continuous on compact. For $t \in \R_+$, define the
time-shifted process $x^t(\cdot) = x(t + \cdot)$. We have the following
convergence theorem for the interpolation~\eqref{eqn:interpolation}
of the iterative process~\eqref{eqn:stochastic-approximation}, where we
recall that we metrize $\continuous(\R_+, \R^d)$ with $\fndist(f, g) =
\sum_{t = 1}^\infty 2^{-t} \sup_{\tau \in [0, t]} \norm{f(\tau) - g(\tau)}
\wedge 1$.
\begin{theorem}
  \label{theorem:functional-convergence-general}
  Let the following conditions hold: 
  \begin{enumerate}[label=(\roman*)]
  \item \label{item:bounded-iterates} The iterates
    are bounded, i.e.\
    $\sup_k \norm{x_k} < \infty$ and $\sup_k \norm{y_k} < \infty$.
  \item \label{item:summable-stepsizes}
    The stepsizes satisfy
    $\sum_{k=1}^\infty \stepsize_k = \infty$ and
    $\sum_{k=1}^\infty \stepsize_k^2 < \infty$. 
  \item \label{item:noise-summation}
    The weighted noise sequence converges:
    $\lim_n \sum_{k=1}^n \stepsize_k \noise_k = v$
    for some $v \in \R^d$.
  \item \label{item:convergence-set-mapping}
    There exists a closed-valued  $H : \R^d \toto \R^d$ such that
    for all $\{z_k\} \subset \R^d$ satisfying $\lim_k z_k = z$
    and all increasing subsequences $\{n_k\}_{k \in \N} \subset \N$,
    we have
    \begin{equation*}
      \lim_{n \to \infty} \dist \left(\frac{1}{n}
        \sum_{k=1}^n g_{n_k}(z_k), H(z)\right) = 0.
    \end{equation*}
  \end{enumerate}
  Then for any sequence $\left\{\tau_k\right\}_{k=1}^\infty \subset \R_+$,
  the sequence of functions $\left\{x^{\tau_k}(\cdot)\right\}$ is relatively
  compact in $\cont(\R_+, \R^d)$. If $\tau_k \to \infty$,
  all limit points of $\left\{x^{\tau_k}(\cdot)\right\}$ are in $\cont(\R_+,
  \R^d)$ and there exists $y : \R_+ \to \R^d$ satisfying
  $y(t) \in H(x(t))$ for all $t \in \R_+$ where
  \begin{equation*}
    \bar{x}(t) = \bar{x}(0) + \int_{0}^t y(\tau) d\tau
    ~~ \mbox{for~all~} t \in \R_+.
  \end{equation*}
\end{theorem}
The theorem is a generalization of Theorem 5.2 of~\citet{Borkar08}, where
the set-valued mappings $g_k$ are identical for all $k$; our
proof techniques are similar. For completeness, we provide a proof in 
Appendix~\ref{sec:proof-thm-functional-convergence}.

\subsubsection{Differential inclusion for
  stochastic model-based methods}
\label{sec:stochastic-approximation-limit}

With Theorem~\ref{theorem:functional-convergence-general} in place,
we can now show how the update
scheme~\eqref{eqn:model-based-minimization} is representable by the general
stochastic approximation~\eqref{eqn:stochastic-approximation}.
To do so, we must verify that any method satisfying
Conditions~\ref{item:convex-model}--\ref{item:upper-approximation}
satisfies the four conditions of
Theorem~\ref{theorem:functional-convergence-general}.  With this in mind, we
introduce a bit of new notation before proceeding. In
analogy to the gradient mapping from convex~\cite{Nesterov04}
and composite
optimization~\cite{DrusvyatskiyIoLe16}, we define a stochastic
gradient mapping $\gradmap$ and consider its limits. For fixed $x$ we define
\begin{equation}
  \label{eqn:gradmap}
  x_\stepsize^+ (\statval) \defeq \argmin_{y \in X}
  \left\{ f_x(y; \statval) + \regularizer(y)
  + \frac{1}{2 \stepsize} \norm{y - x}^2 \right\}
  ~~ \mbox{and} ~~
  \gradmap_\stepsize(x; \statval) \defeq
  \frac{1}{\stepsize} (x - x_\stepsize^+(\statval)),
\end{equation}
For any model $f_x(\cdot; \statval)$ we consider,
the update is well-behaved: it is measurable in
$\statval$~\cite[Lemma 1]{Rockafellar69}, and it is bounded,
as the next lemma shows.
\begin{lemma}
  \label{lemma:gradmap-bound}
  The update~\eqref{eqn:gradmap} guarantees that
  $\norm{\gradmap_\stepsize(x; \statval)} \le \norm{G(x; \statval)}$, where
  $G(x; \statval)$ is the subgradient~\eqref{eqn:actual-subgradient-sets}.
\end{lemma}
\begin{proof}
  For shorthand, write $x^+ = x^+_\stepsize(\statval)$ and let $g \in
  \partial f_x(x; \statval) \subset \partial f(x;\statval)$.  By the
  definition of the optimality conditions for $x^+$, there exists a vector
  $g^+$ that $g^+ \in \partial f_x(x^+; \statval)$ and another vector $v^+ \in
  \partial \regularizer(x^+)$ such that
  \begin{equation*}
    \<g^+ + \frac{1}{\stepsize}(x^+ - x) + v^+, y - x^+\> \ge 0
    ~~ \mbox{for~all~} y \in X.
  \end{equation*}
  Rearranging, we substitute $y = x$ to obtain
  \begin{equation*}
    \<g^+, x^+ - x\> + \frac{1}{\stepsize} \norm{x - x^+}^2
    + \<v^+, x^+ - x\> \le 0.
  \end{equation*}
  The subgradient mapping is monotone for
  $f_x(\cdot;\statval)$ and $\regularizer$, so
  $\<g^+, x - x^+\> \ge \<g, x - x^+\>$
  and $\<v^+, x - x^+\> \ge \<\partial \regularizer(x), x - x^+\>$.
  Thus
  \begin{equation*}
    \<g, x^+ - x\> + \frac{1}{\stepsize} \norm{x - x^+}^2
    + \<v, x^+ - x\> \le 0
  \end{equation*}
  for all $v \in \partial \regularizer(x)$.
  Cauchy-Schwarz implies
  $\norm{g + v} \norm{x^+ - x} \ge \frac{1}{\stepsize} \norm{x - x^+}^2$,
  which implies our desired result.
\end{proof}

To define the population counterpart of the
gradient mapping $\gradmap_\stepsize$, we require a result
showing that the gradient mapping is locally bounded and integrable.
To that end, for $x \in X$ and $\epsilon > 0$, define the Lipschitz constants
\begin{equation*}
  L_\epsilon(x; \statval) \defeq \sup_{x' \in X, \norm{x' - x} \leq \epsilon} 
  \norm{G(x'; \statval)}
  ~~ \mbox{and} ~~
  L_\epsilon(x) \defeq \E_P[L_\epsilon(x; \statrv)^2]^\half.
\end{equation*}
The following lemma shows these are not pathological
(see Appendix~\ref{sec:proof-property-of-Lipschitz-constant} for a proof).
\begin{lemma}
  \label{lemma:property-of-Lipschitz-constant}
  Let Assumptions~\ref{assumption:local-lipschitz}
  and~\ref{assumption:weak-convexity} hold.  Then $x \mapsto L_\epsilon(x;
  \statval)$ and $x \mapsto L_\epsilon(x)$ are upper semicontinuous on $X$
  and $L_\epsilon(x) < \infty$ for all $x \in X$.
\end{lemma}

As a consequence of this lemma
and Lemma~\ref{lemma:gradmap-bound}, $\gradmap_\stepsize(x;\statrv)$
is locally bounded by $L_\epsilon(x;\statval)$ and we may
define the mean subgradient mapping
\begin{equation*}
  \meangradmap_\stepsize(x)
  \defeq \E_P\left[\gradmap_\stepsize(x; \statrv)\right]
  = \int_{\statdomain} \gradmap_\stepsize(x; \statval) dP(\statval). 
\end{equation*}
Moreover, any update of the
form~\eqref{eqn:model-based-minimization}
(e.g.\ Examples~\ref{example:sgd}--\ref{example:guarded-stochastic-prox-point})
has representation
\begin{align}
  \label{eqn:stochastic-methods-as-approximation}
  x_{k+1} = x_k - \stepsize_k \gradmap_{\stepsize_k}(x_k; \statrv_k)
  = x_k - \stepsize_k \wb{\gradmap}_{\stepsize_k}(x_k) 
  - \stepsize_k \error_{\stepsize_k}(x_k; \statrv_k), 
\end{align}
where the noise vector
$\error_\stepsize(x; \statval) \defeq \gradmap_\stepsize(x; \statval) -
\wb{\gradmap}_\stepsize(x)$. Defining the filtration of
$\sigma$-fields
$\mc{F}_k \defeq \sigma(x_0, \statrv_1, \ldots, \statrv_{k-1})$, we
have $x_k \in \mc{F}_k$ and that $\error$ is a
square-integrable martingale difference sequence adapted to $\mc{F}_k$.
Indeed, for $\stepsize$ and $\epsilon > 0$ we have
\begin{equation*}
  \norm{\gradmap_\stepsize(x; \statval)}
  \le L_\epsilon(x; \statval)
  ~~ \mbox{and} ~~
  \norm{\meangradmap_\stepsize(x)} \le L_\epsilon(x)
\end{equation*}
by Lemma~\ref{lemma:gradmap-bound} and the definition
of the Lipschitz constant, and for any $x$ and $\stepsize > 0$, 
\begin{equation}
  \label{eqn:noise-upper-bound}
  \E_P\left[\norm{\error_\stepsize(x; \statrv)}^2 \right]
  \le \E_P\left[\norm{\gradmap_\stepsize(x; \statrv)}^2\right]
  \leq \E\left[L_\epsilon^2(x; \statrv)\right]
  = L_{\epsilon}(x)^2,
\end{equation}
because $\E[\gradmap_\stepsize] = \meangradmap_\stepsize$.
In the context of our iterative procedures, for
any $\stepsize > 0$, 
\begin{equation*}
  \E [\error_\stepsize(x_k; \statrv_k) \mid \mc{F}_k] = 0
  ~~ \mbox{and} ~~
  \E[\norm{\error_\stepsize(x_k; \statrv_k)}^2 \mid \mc{F}_k]
  \le L_\epsilon(x_k)^2.
\end{equation*}
The
(random) progress of each iterate of the algorithm $\gradmap$ is now the sum
of a mean progress $\meangradmap$ and a random noise perturbation $\error$
with (conditional) mean $0$ and bounded second moments. The update
form~\eqref{eqn:stochastic-methods-as-approximation} shows that all of our
examples---stochastic proximal point, stochastic prox-linear, and the
stochastic gradient method---have the
form~\eqref{eqn:stochastic-approximation} necessary for application of
Theorem~\ref{theorem:functional-convergence-general}.

\paragraph{Functional convergence for the stochastic updates}
Now that we have the
representation~\eqref{eqn:stochastic-methods-as-approximation}, it remains to
verify that the mean gradient mapping $\meangradmap$ and errors $\error$
satisfy the conditions necessary for application of
Theorem~\ref{theorem:functional-convergence-general}.  That is, we verify
\ref{item:bounded-iterates} bounded iterates,
\ref{item:summable-stepsizes} non-summable but square-summable stepsizes,
\ref{item:noise-summation} convergence of the weighted error sequence, and
\ref{item:convergence-set-mapping} the distance condition in the theorem.
Condition~\ref{item:summable-stepsizes} is trivial.  To address
condition~\ref{item:bounded-iterates}, we temporarily make the following
assumption, noting that the compactness of $X$ is sufficient for it
to hold (we give other sufficient conditions in Section~\ref{sec:as-convergence},
showing that it is not too onerous).
\begin{assumption}
  \label{assumption:boundedness-of-iterates}
  With probability 1, the iterates~\eqref{eqn:model-based-minimization} are
  bounded,
  \begin{equation*}
    \sup_k \norm{x_k} < \infty. 
  \end{equation*}
\end{assumption}
\noindent
A number of conditions, such as almost supermartingale convergence
theorems~\cite{RobbinsSi71}, are sufficient to guarantee
Assumption~\ref{assumption:boundedness-of-iterates}.
Whenever Assumption~\ref{assumption:boundedness-of-iterates}
holds, we have
\begin{equation*}
  \sup_k \sup_{\stepsize > 0} \norm{\meangradmap_{\stepsize}(x_k)}
  \le \sup_k L_\epsilon(x_k) < \infty,
\end{equation*}
by Lemmas~\ref{lemma:gradmap-bound}
and~\ref{lemma:property-of-Lipschitz-constant},
because the supremum of an upper
semicontinuous function on a compact set is finite. That is,
condition~\ref{item:bounded-iterates} of
Theorem~\ref{theorem:functional-convergence-general} on the boundedness of
$x_k$ and $y_k$ holds.

The error sequences $\error_{\stepsize_k}$ are also well-behaved for the
model-based updates~\eqref{eqn:model-based-minimization}. That is,
condition~\ref{item:noise-summation} of
Theorem~\ref{theorem:functional-convergence-general} is satisfied:
\begin{lemma}
  \label{lemma:summability-of-noise}
  Let Assumptions~\ref{assumption:local-lipschitz},
  \ref{assumption:weak-convexity},
  and~\ref{assumption:boundedness-of-iterates} hold. Then
  with probability 1, the limit $\lim_{n \to \infty}
  \sum_{k=1}^n \stepsize_k \error_{\stepsize_k}(x_k; \statrv_k)$
  exists
  and is finite.
\end{lemma}
\begin{proof}
  Ignoring probability zero events, by
  Assumption~\ref{assumption:boundedness-of-iterates} there is a
  random variable $B$, which is finite with probability $1$, such that
  $\norm{x_k} \leq B$ for all $k \in \N$.  As $L_{\epsilon}(\cdot)$ is upper
  semicontinuous (Lemma~\ref{lemma:property-of-Lipschitz-constant}), we know
  that $\sup \{L_{\epsilon}(x) \mid \norm{x} \le B, x \in X\} <
  \infty$. Hence, using inequality~\eqref{eqn:noise-upper-bound}, we have
  \begin{equation*}
    \sum_{k=1}^\infty 
    \E \left[\stepsize_k^2
      \norm{\error_{\stepsize_k}(x_k; \statrv_k)}^2
      \mid \mc{F}_k\right]
    \leq \sum_{k=1}^\infty \stepsize_k^2
    \sup_{\norm{x} \leq B, x \in X} L_{\epsilon}(x)^2 < \infty.
  \end{equation*}
  Standard convergence results for $\ell_2$-summable martingale difference
  sequences~\cite[Theorem 5.3.33]{Dembo16} immediately give the result.
\end{proof}

Finally, we verify the fourth technical condition
Theorem~\ref{theorem:functional-convergence-general} requires by
constructing an appropriate closed-valued mapping $H : \R^d \toto \R^d$ for
any update scheme of the form~\eqref{eqn:model-based-minimization}.  Recall
the definition~\eqref{eqn:actual-subgradient-sets} of the outer
semicontinuous mapping $G(x) = \E_P[\partial f(x; \statrv)] + \partial
\regularizer(x)$.  We then have the following limiting inclusion,
which is the key result allowing our limit statements.
\begin{lemma}
  \label{lemma:existence-of-G}
  Let the sequence $x_k \in X$ satisfy $x_k \to x \in X$ and
  Assumptions~\ref{assumption:local-lipschitz}
  and~\ref{assumption:weak-convexity} hold.  Let $\{i_k\} \subset \N$ be an
  increasing sequence. Then, for
  updates~\eqref{eqn:model-based-minimization} satisfying
  Conditions~\ref{item:convex-model}--\ref{item:upper-approximation},
  \begin{equation*}
    \lim_{n \to \infty} \dist\left(\frac{1}{n} \sum_{k = 1}^n
      \wb{\gradmap}_{\stepsize_{i_k}}(x_k),
      G(x) + \normalcone_X(x)  \right) = 0.
  \end{equation*}
\end{lemma}
\begin{proof}
  We begin with two intermediate lemmas on the continuity properties of the
  models $f_x$.  Both lemmas assume the conditions of
  Lemma~\ref{lemma:existence-of-G}.

  \begin{lemma}
    \label{lemma:intermediate-lipschitzian}
    There exists $\lipf_\epsilon'(x; \statval)$ such that $y \mapsto f_x(y;
    \statval)$ is $\lipf_\epsilon'(x; \statval)$-Lipschitz for $y\in x +
    (\epsilon/2) \ball$, and $\E[\lipf_\epsilon'(x;\statrv)] < \infty$.
  \end{lemma}
  \begin{proof}
    Let $\epsilon  > 0$, and let $g = \subgrad(x; \statval) \in
    \partial f_x(x; \statval) \subset \partial f(x; \statval)$.  We have
    that
    \begin{equation*}
      f_x(y; \statval) \ge f_x(x; \statval) + \<g, y - x\>
      \ge f(x;\statval) - \lipf_\epsilon(x; \statval) \norm{y - x}
    \end{equation*}
    by the
    local Lipschitz condition~\ref{assumption:local-lipschitz} on $f$.
    Condition~\ref{item:upper-approximation}
    and the Lipschitzian assumptions on $f$ also guarantee that for $y \in x
    + \epsilon \ball$,
    \begin{equation*}
      f_x(y; \statval)
      \le f(y; \statval) + \half \delta_\epsilon(x; \statval)
      \norm{y - x}^2
      \le f(x; \statval) + \left[\lipf_\epsilon(x; \statval)
        + \delta_\epsilon(x; \statval) \norm{x-y}\right]
      \norm{x - y}.
    \end{equation*}
    These two boundedness conditions and
    convexity of the model $f_x$ imply~\cite[Lemma
      IV.3.1.1]{HiriartUrrutyLe93ab} that $y \mapsto f_x(y; \statval)$ is $2
    \lipf_\epsilon(x;\statval) + \delta_\epsilon(x; \statval)
    \epsilon$-Lipschitz for $y \in x + (\epsilon/2)\ball$.
  \end{proof}

  \begin{lemma}
    \label{lemma:intermediate-outer-semicontinuity}
    Let $x_k, y_k \in \xdomain$ satisfy
    $x_k \to x, y_k \to x$, and let $g_k \in \partial f_{x_k}(y_k; \statval)$.
    Then there exists an integrable function
    $\lipf(\cdot)$ such that for large $k$,
    $\dist(g_k, \partial f(x; \statval))
    \le \lipf(\statval)$ for all $\statval$,
    and $\dist(g_k, \partial f(x;\statval)) \to 0$.
  \end{lemma}
  \begin{proof}
    By Lemma~\ref{lemma:intermediate-lipschitzian},
    we know that there exists an integrable $\lipf$
    such that $\norm{g_k} \le \lipf(\statval)$ for all large enough $k$.
    This gives the first claim of the lemma, as $f(\cdot; \statval)$
    is locally Lipschitz (Assumption~\ref{assumption:local-lipschitz}).
    Let $g_\infty$ be any limit point of the sequence $g_k$;
    by moving to a subsequence if necessary, we assume
    without loss of generality that $g_k \to g_\infty \in \R^d$.
    Now let $y \in x + \epsilon \ball$. Then for large $k$
    we have
    \begin{align*}
      f(y; \statval)
      \!\stackrel{(i)}{\ge}\! f_{x_k}(y; \statval)
      - \frac{\delta_\epsilon(x;\statval)}{2} \norm{y - x_k}^2
      & \!\ge\! f_{x_k}(x_k; \statval)
      + \<g_k, y - x_k\> - \frac{\delta_\epsilon(x;\statval)}{2}
      \norm{y - x_k}^2 \\
      & \to f(x; \statval) + \<g_\infty, y - x\>
      - \frac{\delta_\epsilon(x;\statval)}{2} \norm{y - x}^2,
    \end{align*}
    where inequality $(i)$ is a consequence of
    Condition~\ref{item:upper-approximation}.
    By definition of the Fr\'{e}chet subdifferential,
    we have $g_\infty \in \partial f(x;\statval)$ as desired.
  \end{proof}

  Now we return to the proof of Lemma~\ref{lemma:existence-of-G}.
  Let $x_k^+(\statval)$ be shorthand for the result of the
  update~\eqref{eqn:model-based-minimization}
  when applied with the stepsize
  $\stepsize = \stepsize_{i_k}$.  For any $\epsilon > 0$,
  Lemma~\ref{lemma:gradmap-bound} shows
  that
  $\norm{x_k^+(\statval) - x_k} \le \stepsize_{i_k} L_\epsilon(x;
  \statval)$. By the (convex) optimality conditions for
  $x_k^+(\statval)$, there exists a vector $\subgrad^+(x_k; \statval)$ such
  that
  \begin{equation*}
    \subgrad^+(x_k; \statval)
    \in \partial f_{x_k}(x_k^+(\statval); \statval)
  \end{equation*}
  and
  \begin{equation*}
    \gradmap_{\stepsize_{i_k}}(x_k; \statval)
    \in \subgrad^+(x_k; \statval) + \partial \regularizer(x_k^+(\statval))
    + \normalcone_X(x_k^+(\statval)).
  \end{equation*}
  Let $v^+_k(\statval) \in \partial \regularizer(x_k^+(\statval))$ and
  $w^+_k(\statval) \in \normalcone_X(x_k^+(\statval))$ be the vectors
  such that
  \begin{equation*}
    \gradmap_{\stepsize_{i_k}}(x_k; \statval)
    = \subgrad^+(x_k; \statval) + v^+_k(\statval) + w_k^+(\statval).
  \end{equation*}

  The three set-valued mappings $x \mapsto \partial f(x; \statval)$, $x
  \mapsto \partial \regularizer(x)$, and $x \mapsto \normalcone_X(x)$ are
  outer semicontinuous (see Lemmas~\ref{lemma:subg-outer-semicontinuity},
  \ref{lemma:normal-outer-semicontinuity},
  and~\ref{lemma:frechet-subgradients-of-f}).  Since $x_k^+(\statval) \to x$
  tends to $x$ as $k \to \infty$ (as $x_k \to x$), this outer semicontinuity
  and Lemma~\ref{lemma:intermediate-outer-semicontinuity} thus imply
  \begin{equation}
    \dist\!\left(\subgrad^+(x_k; \statval), \partial f(x; \statval)\right)
    \to 0, ~
    \dist\!\left(v_k^+(\statval), \partial \regularizer(x)\right) \to 0, ~
    \dist\!\left(w_k^+(\statval), \normalcone_X(x)\right) \to 0
    \label{eqn:user-outer-semi}
  \end{equation}
  as $k \to \infty$.  Because $x_k \to x$ and the Lipschitz constants
  $L_\epsilon(\cdot; \statval)$ are upper semicontinuous,
  Eq.~\eqref{eqn:user-outer-semi}
  and Lemma~\ref{lemma:gradmap-bound} also imply that
  \begin{align*}
    \limsup_k \norm{\subgrad^+(x_k; \statval)
    + v_k^+(\statval)} \leq L_\epsilon(x; \statval)
    ~~ \mbox{and} ~~
    \limsup_k \norms{\gradmap_{\stepsize_{i_k}}(x_k; \statval)} \leq L_\epsilon(x; \statval).
  \end{align*} 
  By the triangle inequality, we thus obtain
  $\limsup_k \norm{w_k^+(\statval)} \leq 2 L_\epsilon(x; \statval)$, and hence, 
  \begin{align*}
    \dist\left(w_k^+(\statval), \normalcone_\xdomain(x) \cap
    2L_\epsilon(x; \statval) \cdot \ball\right) \to 0.
  \end{align*}
  That
  $L_\epsilon(x) = \E[L_\epsilon(x; \statrv)^2]^\half$ yields
  $\normalcone_\xdomain(x) \cap 2 L_\epsilon(x) \ball \supset \int
  (\normalcone_\xdomain(x) \cap 2 L_\epsilon(x; \statval) \cdot \ball )
  dP(\statval)$, and the definition of the set-valued integral and convexity
  of $\dist(\cdot, \cdot)$
  imply that
  \begin{align}
    \lefteqn{\dist\left(\frac{1}{n} \sum_{k = 1}^n \meangradmap_{\stepsize_{i_k}}(x_k),
    G(x) + 
    \normalcone_X(x) \cap 2 L_\epsilon(x) \cdot \ball \right)} \nonumber \\
    & \le \frac{1}{n} \sum_{k = 1}^n
    \int \dist\!\left(\gradmap_{\stepsize_{i_k}}(x_k; \statval),
      \partial f(x; \statval)
      + \partial \regularizer(x)
      + \normalcone_X(x) \cap 2 L_\epsilon(x; \statval) \cdot \ball \right)
      dP(\statval).
      \label{eqn:thing-to-dct}
  \end{align}

  We now bound the preceding integral.  By the definition of Minkowski
  addition and the triangle inequality, we have the pointwise convergence
  \begin{align*}
    \lefteqn{\dist\!\left(\gradmap_{\stepsize_{i_k}}(x_k; \statval),
    \partial f(x; \statval) + \partial \regularizer(x)
    + \normalcone_X(x) \cap 2 L_\epsilon(x; \statval) \ball \right)} \\
    & \!\le\hspace{-.02em} \dist\!\left(\subgrad(x_k; \statval), \partial f(x; \statval)\right)
    + \dist\!\left(v_k^+(\statval), \partial \regularizer(x)\right)
    + \dist\!\left(w_k^+(\statval), \normalcone_X(x) \cap 2L_\epsilon(x; \statval)
    \ball\right) \!\rightarrow\! 0
  \end{align*}
  as $k \to \infty$ by the earlier outer semicontinuity convergence
  guarantee~\eqref{eqn:user-outer-semi}.  For suitably large $k$, the first
  term in the preceding sum is bounded by an integrable function
  $\lipf_\epsilon'(x; \statval)$ by
  Lemma~\ref{lemma:intermediate-outer-semicontinuity} and the latter two are
  bounded by $2 L_\epsilon(x; \statval)$, which is square integrable by
  Lemma~\ref{lemma:property-of-Lipschitz-constant}.  Lebesgue's dominated
  convergence theorem thus implies that the individual summands in
  expression~\eqref{eqn:thing-to-dct} converge to zero, and the analytic
  fact that the Ces\'{a}ro mean $\frac{1}{n} \sum_{k = 1}^n a_k \to 0$ if
  $a_k \to 0$ gives the result.
\end{proof}

With this lemma, we may now show the functional convergence of our
stochastic model-based update
schemes~\eqref{eqn:model-based-minimization}. We have verified that each of
the conditions
\ref{item:bounded-iterates}--\ref{item:convergence-set-mapping} of
Theorem~\ref{theorem:functional-convergence-general} hold with the mapping
$H(x) = -\normalcone_\xdomain(x) - G(x)$. Indeed, $H$ is closed-valued and
outer-semicontinuous as $G(\cdot)$ is convex compact o.s.c.\ and
$\normalcone_\xdomain(\cdot)$ is closed and o.s.c. Thus, with slight abuse
of notation, let $x(\cdot)$ be the linear
interpolation~\eqref{eqn:interpolation} of the iterates $x_k$ for either the
stochastic prox-linear algorithm or the stochastic subgradient algorithm,
where we recall that $x^t(\cdot) = x(t + \cdot)$. We have
\begin{theorem}
  \label{theorem:functional-convergence-prox-linear}
  Let Assumptions~\ref{assumption:local-lipschitz},
  \ref{assumption:weak-convexity},
  and~\ref{assumption:boundedness-of-iterates} hold.
  With probability one over the random sequence
  $\statrv_i \simiid P$ we have the following.
  For any sequence
  $\left\{\tau_k\right\}_{k=1}^\infty$, the function sequence 
  $\left\{x^{\tau_k}(\cdot)\right\}$ is relatively compact in 
  $\cont(\R_+, \R^d)$. In addition, for any sequence 
  $\tau_k \to \infty$, any limit point of 
  $\left\{x^{\tau_k}(\cdot)\right\}$ in $\cont(\R_+, \R^d)$ 
  satisfies
  \begin{equation*}
    \bar{x}(t) = \bar{x}(0) + \int_{0}^t y(\tau) d\tau
    ~~ \mbox{for~all~} t \in \R_+,
    ~~ \mbox{where} ~~
    y(\tau) \in - G(x(\tau)) - \normalcone_\xdomain(x(\tau)).
  \end{equation*}
\end{theorem}

\subsection{Properties of the limiting differential inclusion}
\label{sec:diff-inclusion-properties}

Theorem~\ref{theorem:functional-convergence-prox-linear} establishes that
the updates~\eqref{eqn:model-based-minimization}, which include stochastic
subgradient methods (Ex.~\ref{example:sgd}), stochastic prox-linear methods
(Ex.~\ref{example:stochastic-prox-linear}), or stochastic proximal point
methods
(Exs.~\ref{example:stochastic-prox-point}--\ref{example:guarded-stochastic-prox-point}), have sample paths asymptotically
approximated by the differential inclusion
\begin{equation*}
  \dot{x} \in -G(x) - \normalcone_X(x)
  ~~ \mbox{where} ~~
  G(x) = \partial f(x) + \partial \regularizer(x)
\end{equation*}
for the objective $f(x) = \E[f(x;\statrv)]$.
To establish convergence of the iterates $x_k$ themselves, we must
understand the limiting properties of trajectories of the preceding
differential inclusion. 

We define the minimal subgradient
\begin{equation*}
  \subgrad\opt(x) \defeq \argmin_g \left\{\norm{g}^2 \mid
  g \in \partial f(x) + \partial \regularizer(x) + \normalcone_X(x)
  \right\}
  = \project_{G(x) + \normalcone_X(x)}(0).
\end{equation*}
Before presenting the theorem on the differential inclusion, we need one 
regularity assumption on the objective function $F(x)$ and the constraint 
set $X$. Recall that a function $f$ is coercive if $f(x) \to \infty$ as 
$\norm{x} \to \infty$. 

\begin{assumption}
  \label{assumption:objective-coercive}
  The function $x \mapsto F(x) + \bigindic_X(x)$ is coercive. 
\end{assumption}

This assumption ensures that the sublevel sets of the objective function 
$F + \bigindic_X$ are compact. Now we have the following convergence theorem.

\begin{theorem}
  \label{theorem:monotone-trajectory}
  Let Assumptions~\ref{assumption:local-lipschitz},
  \ref{assumption:weak-convexity}, and~\ref{assumption:objective-coercive}
  hold.  Let $x(\cdot)$ be a solution to the differential inclusion $\dot{x}
  \in -\partial f(x) - \partial \regularizer(x) - \normalcone_X(x)$
  initialized at $x(0) \in X$. Then $x(t)$ exists and is in $X$ for all $t
  \in \R_+$, $\sup_t \norm{x(t)} < \infty$, $x(t)$ is Lipschitz in $t$, and
  \begin{equation*}
    f(x(t)) + \regularizer(x(t)) 
    + \int_0^t \norm{\subgrad\opt(x(\tau))}^2 d\tau
    \le f(x(0)) + \regularizer(x(0)).
  \end{equation*}
\end{theorem}

We prove the theorem in Section~\ref{sec:proof-monotone-trajectory},
giving a few corollaries to show that solutions to the
differential inclusion converge to stationary points of $f + \regularizer$.
\begin{corollary}
  \label{corollary:constant-is-stationary}
  Let $x(\cdot)$ be a solution to
  $\dot{x} \in -G(x) - \normalcone_X(x)$
  and assume that for some $t > 0$ we have $f(x(t)) = f(x(0))$. Then
  $\subgrad\opt(x(\tau)) = 0$ for all $\tau \in [0, t]$.
\end{corollary}
\begin{proof}
  By Theorem~\ref{theorem:monotone-trajectory}, we have that
  $\int_0^t \norm{\subgrad\opt(x(\tau))}^2 d\tau = 0$, so that
  $\subgrad\opt(x(\tau)) = 0$ for almost every $\tau \in [0, t]$.
  The continuity of $x(\cdot)$ and outer semi-continuity of
  $G$ extend this to all $\tau$.
\end{proof}

In addition, we can show that all cluster points of any trajectory
solving the differential inclusion~\eqref{eqn:differential-inclusion}
are stationary. First, we recall the following definition.
\begin{definition}
  \label{definition:cluster-points}
  Let $\{x(t)\}_{t \ge 0}$ be a trajectory. A point $x_\infty$ is a
  \emph{cluster point} of $x(t)$ if there exists an increasing sequence
  $t_n \to \infty$ such that $x(t_n) \to x_\infty$.
\end{definition}
\noindent
We have the following observation.
\begin{corollary}
  \label{corollary:almost-clusters-stationary}
  Let $x(\cdot)$ be the trajectory of $\dot{x} \in -G(x) - \normalcone_X(x)$
  and let $x_\infty$ be a cluster point of $x(\cdot)$. Then $x_\infty$ is
  stationary, meaning that $\subgrad\opt(x_\infty) = 0$.
\end{corollary}
\begin{proof}
  For $\epsilon > 0$, define $\mc{T}_\epsilon(x_\infty) = \{t \in \R_+
  \mid \norm{x(t) - x_\infty} \le \epsilon\}$, and
  let $\mu$ denote Lebesgue measure on $\R$.
  Because the trajectory $x(\cdot)$ is Lipschitz,
  we have that $\mu(\mc{T}_\epsilon(x_\infty) \cap \openright{T}{\infty})
  = \infty$ for all $\epsilon > 0$ and $T < \infty$
  (cf.~\cite[Proposition 6.5.1]{AubinCe84}).
  Let $\epsilon_n, \delta_n$ be
  sequences of positive numbers converging to 0.  Because
  $f(x(t)) + \regularizer(x(t))$ converges to
  $f(x_\infty) + \regularizer(x_\infty)$ (the sequence is decreasing
  and $f + \regularizer$ is continuous), we have
  $\int \norm{\subgrad\opt(x(t))}^2 dt < \infty$. Moreover, there exist
  increasing $T_n$ such that
  \begin{equation*}
    \int_{\mc{T}_{\epsilon_n}(x_\infty) \cap [T_n, \infty)}
    \norm{\subgrad\opt(x(t))}^2 dt \le \delta_n.
  \end{equation*}
  As $\mu(\mc{T}_{\epsilon_n}(x_\infty) \cap [T_n, \infty)) = \infty$, there
  must exist an increasing sequence $t_n \ge T_n$, $t_n \in
  \mc{T}_{\epsilon_n}(x_\infty)$,
  such that
  $\norm{\subgrad\opt(x(t_n))}^2 \le \delta_n$. By construction
  $x(t_n) \to x_\infty$, and we have a subsequence
  $\subgrad\opt(x(t_n)) \to 0$. The outer semi-continuity of
  $x \mapsto G(x) + \normalcone_X(x)$ implies that $0 \in G(x_\infty)
  + \normalcone_X(x_\infty)$.
\end{proof}

\subsubsection{Proof of Theorem~\ref{theorem:monotone-trajectory}}
\label{sec:proof-monotone-trajectory}

Our argument proceeds in three main steps. For shorthand, we define
$F(x) = f(x) + \regularizer(x)$. Our first step shows that the function
$V(x) \defeq F(x) + \bigindic_X(x) - \inf_{y \in X} F(y)$ is a Lyapunov
function for the differential inclusion~\eqref{eqn:differential-inclusion},
where we take the function $W$ in Lemma~\ref{lemma:monotone-trajectory} to
be $W(x, v) = \norm{v}^2$.  Once we have this, then we can use the existence
result of Lemma~\ref{lemma:aubin-existence} to show that a solution
$x(\cdot)$ exists in a neighborhood of $0$. The uniqueness of trajectories
(Lemma~\ref{lemma:kunze-existence}) then implies that the trajectory $x$ is
non-increasing for $V$, which then---combined with the assumption of
coercivity of $F + \bigindic_X$---implies that the trajectory $x$ is bounded
and we can extend uniquely it to all of $\R_+$.

\paragraph{Part 1: A Lyapunov function}
To develop a Lyapunov function, we compute directional derivatives of $f +
\regularizer$.
\begin{lemma}[\cite{HiriartUrrutyLe93ab}, Chapter VI.1]
  \label{lemma:smallest-subgradient}
  Let $h$ be convex and $g\opt
  = \argmin_{g \in \partial h(x)} \{\norm{g}\}$. Then the directional
  derivative satisfies
  $h'(x; -g\opt) = -\norm{g\opt}^2$.
\end{lemma}

Now, take $\subgrad\opt(x)$ as in the statement of the theorem and define
the Lyapunov-like function $V(x) = f(x) + \regularizer(x) + \bigindic_X(x) -
\inf_{y \in X} \{f(y) + \regularizer(y)\}$; we claim that
\begin{equation}
  \label{eqn:composite-lyapunov-function}
  V'(x; -\subgrad\opt(x))
  \le - \norm{\subgrad\opt(x)}^2.
\end{equation}
Before proving~\eqref{eqn:composite-lyapunov-function},
we note that it is
identical to that in Lemma~\ref{lemma:monotone-trajectory} on
monotone trajectories of differential inclusions. Thus
there exists a solution $x(\cdot)$
to the differential
inclusion $\dot{x} \in -G(x) - \normalcone_X(x)$ defined on $[0, T]$ for
some $T > 0$, where $x(\cdot)$ satisfies
\begin{equation}
  \label{eqn:intermediate-lyapunov}
  f(x(t)) + \regularizer(x(t)) + \bigindic_X(x(t))
  \le f(x(0)) + \regularizer(x(0))
  - \int_0^t \norm{\subgrad\opt(x(\tau))}^2 d\tau
\end{equation}
for all $t \in [0, T]$.  We return now to prove the
claim~\eqref{eqn:composite-lyapunov-function}.  Let $x \in \xdomain$ and
recall by Assumption~\ref{assumption:weak-convexity} that for all $\lambda
\ge \E[\lambda(\statrv, x)]$ that $f + \frac{\lambda}{2}
\norm{\cdot - x_0}^2$ is convex in an $\epsilon$-neighborhood of $x$. Now,
define $F_x(y) = f(y) + \regularizer(y) + \frac{\lambda}{2} \norm{y - x}^2$,
so that for $v$ with $\norm{v} = 1$ and $t \le \epsilon$, we have
\begin{equation*}
  |F(x + tv) - F(x)| \le |F_x(x + tv) - F(x)|
  + \frac{t^2 \lambda^2}{2} \norm{v}^2.
\end{equation*}
Because $\regularizer$ is convex and the error in the
approximation $f_x$ of $f$ is second-order, taking limits as
$u \to v, t \to 0$, we have for any fixed $x \in X$ that
\begin{align*}
  \lefteqn{\liminf_{t \downarrow 0, u \to v}
  \frac{F(x + t u) + \bigindic_X(x + t u) - F(x)}{t}} \\
  & = \liminf_{t \downarrow 0} \frac{F_x(x + tv)
  + \bigindic_X(x + tv)
  - F_x(x)}{t}
  = \sup_{g \in \partial f(x) + \partial \regularizer(x)
  + \normalcone_X(x)} \<g, v\>,
\end{align*}
where $F(x) = f(x) + \regularizer(x)$, and we have used
that the subgradient set of $y \mapsto F_x(y)$ at
$y = x$ is $\partial f(x) + \partial \regularizer(x)$. 
 Applying Lemma~\ref{lemma:smallest-subgradient} with
$v = -\subgrad\opt(x)$ gives
claim~\eqref{eqn:composite-lyapunov-function}.

\paragraph{Part 2: Uniqueness of trajectories}
Lemma~\ref{lemma:kunze-existence} shows that solutions to
$\dot{x} \in -G(x) - \normalcone_X(x)$ have unique trajectories almost
immediately.
By Assumption~\ref{assumption:weak-convexity}, for
any $x \in \xdomain$,
$f + \regularizer + \frac{\lambda}{2} \norm{\cdot}^2$ is
convex on the set $X \cap \{x + \epsilon \ball\}$
for all $\lambda \ge \E[\lambda(\statrv, x)]$.
Thus for points $x_1, x_2$ satisfying $\norm{x_i - x} \le \epsilon$ and
$g_i \in \partial f(x_i) + \partial \regularizer(x_i) + \normalcone_X(x_i)$,
\begin{equation*}
  \<g_1 + \lambda x_1 - g_2 - \lambda x_2,
  x_1 - x_2\> \ge 0
  ~~~ \mbox{or} ~~~
  \<-g_1 + g_2, x_1 - x_2\> \le \lambda \norm{x_1 - x_2}^2
\end{equation*}
because subgradients of convex functions are
increasing~\cite[Ch.~VI]{HiriartUrrutyLe93ab}.  Now, suppose that on an
interval $[0, T]$ the trajectory $x(t)$ satisfies $\norm{x(t)} \le B$, that
is, it lies in a compact subset of $\xdomain$. Then by taking a finite
subcovering of $B \ball \cap \xdomain$ as necessary, we may assume $f +
\regularizer + \frac{\lambda}{2} \norm{\cdot}^2$
is convex over $B \ball \cap \xdomain$.
This preceding display
is equivalent to the condition of Lemma~\ref{lemma:kunze-existence}, so
that for any $B$ and any interval $[0, T]$ for which the trajectory $x(t)$
satisfies $\norm{x(t)} \le B$ on $t \in [0, T]$, the trajectory
is unique. In particular, the Lyapunov
inequality~\eqref{eqn:intermediate-lyapunov} is satisfied on the interval
over which the trajectory $\dot{x} \in -G(x) - \normalcone_X(x)$ is defined.

\paragraph{Part 3: Extension to all times}
We argue that we may take $T \to \infty$. For any fixed
$T < \infty$, we know that
$f(x(T)) + \regularizer(x(T)) \le f(x(0)) + \regularizer(x(0))$, and the
coercivity of $f + \regularizer$ over $X$ implies that
there exists $B < \infty$ such that $\norm{x(t)} \le B$
on this trajectory (i.e.\ $t \in [0, T]$). The
compactness of $\partial f(x) + \partial \regularizer(x)$ for
$x \in X \cap \{y : \norm{y} \le B\}$ implies that
$\inf_g\{\norm{g} \mid g
\in \partial f(x) + \partial \regularizer(x) + \normalcone_X(x)\}$
is bounded (because $0 \in \normalcone_X(x)$).
The condition on existence of paths for all times $T$ in
Lemma~\ref{lemma:aubin-existence} applies.

The Lipschitz condition on $x(t)$ is an immediate consequence of the
boundedness of the subgradient sets
$\partial f(x) + \partial \regularizer(x)$ for bounded $x$.


\subsection{Almost sure convergence to stationary points}
\label{sec:as-convergence}

Thus far we have shown that the limit points of the stochastic model-based
iterations~\eqref{eqn:model-based-minimization} are asymptotically
equivalent to the differential inclusion~\eqref{eqn:differential-inclusion}
(Theorem~\ref{theorem:functional-convergence-prox-linear}) and that
solutions to the differential inclusion have certain uniqueness and
convergence properties (Theorem~\ref{theorem:monotone-trajectory}). Based on
those asymptotic equivalence results and convergence properties, this
section shows that cluster points of the iterates $x_k$ are stationary.
To provide a starting point, we state the main theorem of the section,
which applies to any sequence $x_k$ generated by
a model update~\eqref{eqn:model-based-minimization}
satisfying Conditions~\ref{item:convex-model}--\ref{item:upper-approximation}.
\begin{theorem}
  \label{theorem:stationary-cluster-points-extension}
  Let Assumptions~\ref{assumption:local-lipschitz},
  \ref{assumption:weak-convexity},
  \ref{assumption:boundedness-of-iterates}, and 
  \ref{assumption:objective-coercive} hold.  
  Then with probability 1,
  \begin{equation}
    \label{eqn:function-value-inclusion}
    \big[\liminf_k F(x_k), \limsup_k F(x_k)\big]
    \subseteq F(X\opt) = \{f(x): x\in X\opt\}.
  \end{equation} 
\end{theorem}

Let us discuss the theorem briefly.
Theorem~\ref{theorem:informal-theorem-stationary-cluster-points} is an
immediate consequence of the theorem, as
Assumptions~\ref{assumption:boundedness-of-iterates}
and~\ref{assumption:objective-coercive} are trivial when $\xdomain$ is
compact. To illustrate the theorem, we also establish
convergence of the iterates of $x_k$ to the stationary set $X\opt$ under the
weak Sard-type Assumption~\ref{assumption:countable-regularity}, giving
Corollary~\ref{corollary:informal} as a consequence.
\begin{corollary}
  \label{corollary:stationary-cluster-points-extension}
  Let
  Assumptions~\ref{assumption:local-lipschitz}--\ref{assumption:objective-coercive}
  hold. With probability 1, all cluster points of the sequence
  $\{x_k\}_{k=1}^\infty$ belong to the stationary set $X\opt$, and $F(x_k) =
  f(x_k) + \regularizer(x_k)$ converges.
\end{corollary}
\begin{proof}
  By Assumption~\ref{assumption:countable-regularity} (that $(F(X\opt))^c$
  is dense),
  Theorem~\ref{theorem:stationary-cluster-points-extension} implies that
  $F(x_k)$ converges. That all cluster points of $x_k$ belong to $X\opt$
  follows from Lemma~\ref{lemma:extreme-value-stationary-point}
  to come.
\end{proof}

\paragraph{Conditions for boundedness of the iterates}
Key to our theorems is the boundedness of the iterates $x_k$,
so it is important to give sufficient conditions that
Assumption~\ref{assumption:boundedness-of-iterates} holds
even when $\xdomain$ is unbounded.
We may develop examples by
considering the joint properties of the regularizer $\regularizer$ and
objectives $f(x; \statrv)$ in the stochastic updates of our methods.  We
mention two such examples, focusing for simplicity
on the stochastic subgradient method (Ex.~\ref{example:sgd}, using
subgradient $\subgrad(x;\statval) \in \partial f(x;\statval)$) in the
unconstrained case $\xdomain = \R^d$.
We first assume that $\regularizer(x) = \frac{\lambda}{2} \norm{x}^2$,
i.e.\ $\ell_2$ or Tikhonov regularization, common in
statistical learning and inverse problems. In addition, let us assume that
$f(x; \statval)$ is $L(\statval)$-Lipschitz in
$x$, where $L \defeq \E[L(\statrv)^2]^\half < \infty$, so that
$\norm{\subgrad(x; \statval)} \le L(\statval)$. This regularization is
sufficient
to guarantee boundedness:
\begin{observation}
  \label{observation:boundedness-tikhonov}
  Let the conditions of the preceding paragraph hold.
  Assume that $\E[L(\statrv)^2] < \infty$. Then
  with probability 1, $\sup_k \norm{x_k} < \infty$.
\end{observation}
\noindent
We provide the proof of Observation~\ref{observation:boundedness-tikhonov}
in
Appendix~\ref{sec:proof-boundedness-tikhonov}.
More quickly growing regularization
functions $\regularizer$ also yield bounded iterates. We
begin with a definition.
\begin{definition}
  \label{definition:coercivity}
  A function $\regularizer$ is \emph{$\beta$-coercive} if
  $\lim_{\norm{x} \to \infty} \regularizer(x) / \norm{x}^\beta = \infty$,
  and it is \emph{$(\lambda, \beta)$-regularly coercive} if it is 
  $\beta$-coercive and $\regularizer(x) \ge \regularizer(\lambda x)$ for 
  $\norm{x}$ large.
\end{definition}
\begin{observation}
  \label{observation:coercive-bounded}
  Let $\regularizer$ be $(\beta, \lambda)$-coercive with $\lambda \in
  \openright{0}{1}$. Assume that for all $\statval \in \statdomain$, $x
  \mapsto f(x; \statval)$ is $L (1 +
  \norm{x}^\nu)$-Lipschitz in a neighborhood of $x$, where $L < \infty$ is
  some constant, and $\nu < \beta - 1$. Then $\sup_k \norm{x_k} < \infty$.
\end{observation}
\noindent
\noindent
We provide the proof of Observation~\ref{observation:coercive-bounded} in 
Appendix~\ref{sec:proof-observation-coercive-bounded}.


\newcommand{\pre}{^-}
\newcommand{\subs}{^+}
\newcommand{\maxgradmap}{M}

\subsubsection{Proof of Theorem~\ref{theorem:stationary-cluster-points-extension}}
\label{sec:proof-stationary-cluster-points-extension}

We prove the theorem using two intermediate results: in the first part
(Lemma~\ref{lemma:extreme-value-stationary-point}), we show that if a
cluster point $x_\infty$ of the sequence $x_k$ is non-stationary, then the
iterates $F(x_k)$ must decrease through $F(x_\infty)$ infinitely often. A
consequence we show is that $\limsup_k F(x_k)$ and $\liminf_k F(x_k)$ belong
to $F(\xdomain\opt)$. We then show
(Lemma~\ref{lemma:lower-bound-time-difference}) that the
interpolated path $x(\cdot)$ of the iterates $x_k$
(recall definition~\eqref{eqn:interpolation}) cannot move too
quickly (Lemma~\ref{lemma:lower-bound-time-difference}). We finally use
this to show that all limting values of $f(x_k) + \regularizer(x_k)$ belong
to $F(X\opt)$.
In the statements of the lemmas, we implicitly assume
all of the conditions of the theorem (i.e.\
Assumptions~\ref{assumption:local-lipschitz}, \ref{assumption:weak-convexity},
\ref{assumption:boundedness-of-iterates},
and~\ref{assumption:objective-coercive}).


We start with a result on the boundaries of the sequences $F(x_k)$ and the
growth of the path $x(t)$ interpolating the iterates $x_k$ (recall the
definition~\eqref{eqn:interpolation}).
\begin{lemma}
  \label{lemma:extreme-value-stationary-point}
  With probability one, $\liminf_k F(x_k) \in F(\xdomain\opt)$ and
  $\limsup_k F(x_k) \in F(\xdomain\opt)$. For
  any increasing sequence $\{h_k\} \subset \R$ satisfying
  $h_k \to \infty$ and $\lim_k x(h_k) = x_\infty \not\in \xdomain\opt$
  and any sequence $\tau_k \to \tau > 0$,
  \begin{equation}
    \liminf_k F(x(h_k - \tau_k)) > F(x_\infty) > \limsup_k F(x(h_k + \tau_k)).
    \label{eqn:F-must-grow-around}
  \end{equation}
\end{lemma}
\begin{proof}
  We begin with the second claim~\eqref{eqn:F-must-grow-around} of the
  lemma, as the first is a nearly immediate consequence of the second.  Let
  the probability 1 events of
  Theorem~\ref{theorem:functional-convergence-prox-linear} hold, that is,
  the limit points of the shifted sequences $\{x^{\tau_k}(\cdot)\}$ satisfy
  the differential inclusion~\eqref{eqn:differential-inclusion}.  We introduce
  the left and right-shifted times
  \begin{equation*}
    h_k\pre = h_k -\tau_k
    ~~\text{and}~~h_k\subs = h_k + \tau_k~~\text{for $k \in \N$}.
  \end{equation*}
  To show the lemma, it suffices to show that, for any subsequence 
  $\{h_{k(m)}\}$ of the sequence $\{h_k\}$, there 
  exists a further subsequence $\{h_{k(m(n))}\}_{n \in \N}$
  such that
  \begin{equation}
    \label{eqn:target-strict-decrease-trajectory}
    \lim_{n \to \infty}
    F(x_{h_{k(m(n))}\pre})
    > F(x_\infty)
    > \lim_{n \to \infty} F(x_{h_{k(m(n))}\subs}).
  \end{equation}
  Now, fix a subsequence $\{h_{k(m)}\}_{m \in \N}$.  By
  Assumption~\ref{assumption:boundedness-of-iterates}, both sequences
  $\{x(h_{k(m)}\pre)\}$ and $\{x(h_{k(m)}\subs)\}$ are relatively compact in
  $\R^d$, and Theorem~\ref{theorem:functional-convergence-prox-linear}
  implies that the sequence of shifted functions
  $\{x^{h_{k(m)}\pre}(\cdot)\}_{m \in \N}$ is relatively compact in
  $\cont(\R_+, \R^d)$. As a consequence, there exists a further subsequence
  $\{h_{k(m(n))}\}_n$ such that for $u_n = x(h_{k(m(n))}\pre)$ and
  $v_n = x(h_{k(m(n))}\subs)$, there are points
  $u_\infty$ and $v_\infty$ and a function $\wb{x} \in \cont(\R_+, \R^d)$
  such that
  \begin{equation*}
    \lim_{n} u_n = u_\infty,
    ~~
    \lim_n v_n = v_\infty,
    ~~ \mbox{and} ~~
    \lim_{n \to \infty} x^{h_{k(m(n))}\pre}(\cdot) = \wb{x}(\cdot).
  \end{equation*}
  By this equation, that $\tau_k \to \tau$ as $k \to \infty$, and the assumption
  in the lemma that $\lim_k x(h_k) = \lim_k x^{h_k\pre}(\tau_k) = x_\infty$, we
  have $\wb{x}(0) = u_\infty$, $\wb{x}(\tau) = x_\infty$, and $\wb{x}(2\tau) =
  v_\infty$.  Theorem~\ref{theorem:functional-convergence-prox-linear} shows
  that $\wb{x}$ satisfies the differential
  inclusion~\eqref{eqn:differential-inclusion},
  which
  has monotone trajectory by Theorem~\ref{theorem:monotone-trajectory}.
  As $\wb{x}(\tau) = x_\infty \not\in X\opt$,
  Corollary~\ref{corollary:constant-is-stationary}
  implies the strict decrease
  \begin{equation*}
    F(u_\infty) = F(\wb{x}(0))
    > F(\wb{x}(\tau)) > F(\wb{x}(2\tau)) = F(v_\infty),
  \end{equation*}
  yielding inequality~\eqref{eqn:target-strict-decrease-trajectory}
  and thus inequality~\eqref{eqn:F-must-grow-around}.

  Now we show the first claim of the lemma.  Let $y = \liminf_k F(x_k)$ (the
  proof for case $y = \limsup_k F(x_k)$ is, \emph{mutatis mutandis},
  identical). As the sequence $\{x_k\}_{k=1}^\infty$ is bounded and the
  function $F$ is continuous on $\xdomain$, there is a subsequence
  $\{x_{k(m)}\}_{m \in \N}$ with $x_{k(m)} \to x_\infty$ and $\lim_m
  F(x_{k(m)}) = F(x_\infty) = y$. Recall that $x_{k} = x(t_k)$ for $t_k =
  \sum_{i = 1}^k \stepsize_i$. If $x_\infty \not \in \xdomain\opt$, then for
  any $\tau > 0$ and for $h_k = t_k$,
  inequality~\eqref{eqn:F-must-grow-around} implies $F(x_\infty) > \limsup_m
  F(x(h_{k(m)} + \tau)) \ge \liminf_k F(x_k)$, an absurdity, so we must have
  $x_\infty \in \xdomain\opt$.
\end{proof}

Our second intermediate result shows that the interpolated paths
$x(\cdot)$ cannot move too quickly.
\begin{lemma}
  \label{lemma:lower-bound-time-difference}
  For any two sequences $\{h_k\}_{k=1}^\infty$ and $\{h_k'\}_{k=1}^\infty$
  satisfying $h_k' > h_k$, $\lim_k h_k' = \lim_k h_k = \infty$ and
  $\liminf_k \norm{x(h_k') - x(h_k)} > 0$, we have with probability 1 that
  $\liminf_k (h_k' - h_k) > 0$.
\end{lemma}  
\begin{proof}
  As in the proof of
  Lemma~\ref{lemma:extreme-value-stationary-point}, fix the sample
  $\statrv_1, \statrv_2, \ldots$ so that the probability 1 conclusions of
  Theorem~\ref{theorem:functional-convergence-prox-linear} hold.
  Now, for $h \in \R_+$ define
  \begin{equation*}
    \klow(h) = \max\{k \in \N : t_k \le h\}
    ~~ \mbox{and} ~~
    \khigh(h) = \min\{k \in \N : t_k \ge h\},
  \end{equation*}
  where we recall the interpolation times $t_k = \sum_{i = 1}^k \stepsize_i$.
  As $\stepsize_k \to 0$,  the
  statement $\liminf_{k\to \infty} (h_k - h_k') > 0$ is equivalent to 
  the statement $\liminf_{k\to \infty} (t_{\khigh(h_k)} - t_{\klow(h_k')}) > 0$.  
  For any $m \le n \in \N$, we have
  \begin{align*}
    \norm{x(t_n) - x(t_m)}
    & = \normbigg{\sum_{i=m+1}^n \stepsize_i 
      \wb{\gradmap}_{\stepsize_i}(x_i) + \sum_{i=m+1}^n \stepsize_i \noise_i}
    \\
    & \leq (t_n - t_m) \sup_i \norm{\wb{\gradmap}_{\stepsize_i}(x_i)}
    + \normbigg{\sum_{i=m+1}^n \stepsize_i \noise_i}.
  \end{align*}
  Let $\maxgradmap =\sup_i \norm{\wb{\gradmap}_{\stepsize_i}(x_i)} < \infty$
  (use Lemmas~\ref{lemma:gradmap-bound}
  and~\ref{lemma:property-of-Lipschitz-constant} to see that
  $\maxgradmap < \infty$). Lemma~\ref{lemma:summability-of-noise} 
  implies that $\lim_{m\to \infty} \sup_{n \geq m}\norms{\sum_{i=m+1}^n \stepsize_i
    \noise_i} = 0$. Thus, we obtain that for any $\epsilon > 0$, there exists 
  $N \in \N$ such that for all $m, n \geq N$,
  \begin{equation}
    \label{eqn:time-separations}
    (t_n - t_m) \maxgradmap \ge \norm{x(t_n) - x(t_m)} - \normbigg{\sum_{i = m + 1}^n
      \stepsize_i \noise_i}
    \ge \norm{x(t_n) - x(t_m)} - \epsilon.
  \end{equation}
  As $x(\cdot)$ are linear interpolations of $x_k = x(t_k)$
  and $h_k, h_k' \to \infty$,
  for any $\epsilon > 0$ there exists 
  exists $K \in \N$ such that $k \geq K$ implies
  \begin{align*}
    \norm{x(h_k') - x(h_k)}
    & \leq \max\left\{
    \norm{x(t_n) - x(t_m)} : n, m \in [\klow(h_k'), \khigh(h_k)]\right\} \\
    & \le \left(t_{\khigh(h_k)} - t_{\khigh(h_k')}\right) \maxgradmap
    + \epsilon.
  \end{align*}
  Since $\liminf_k \norm{x(h_k') - x(h_k)} > 0$,
  inequality~\eqref{eqn:time-separations}
  gives the result.
\end{proof}

\newcommand{\high}{^{\mathsf{hi}}}
\newcommand{\low}{^{\mathsf{lo}}}

To prove the theorem, we assume $(\liminf_k F(x_k), \limsup_k F(x_k))$ is
non-empty, as otherwise the result is trivial.  As in the proof of
Lemma~\ref{lemma:extreme-value-stationary-point}, fix the sample $\statrv_1,
\statrv_2, \ldots$ so that the probability 1 conclusions of
Theorem~\ref{theorem:functional-convergence-prox-linear} hold.

Suppose for the sake of contradiction that
$y\high \in (\liminf_k F(x_k), \limsup_k F(x_k))$ satisfies $y\high
\not\in F(X\opt)$. Let
$y\low < y\high$, $y\low \in (\liminf_k F(x_k), \limsup_k
F(x_k))$.
We claim we may choose sequences
$\{h_k\low\}$ and $\{h_k\high\}$ with $h_k\low < h_k\high$,
$\lim_k h_k\low = \lim_k h_k\high = \infty$, and
\begin{equation}
  \label{eqn:sequence-times}
  \begin{split}
    F(x(h_k\low)) = y\low, ~~F(x(h_k\high)) = y\high, ~~\text{and}~~ 
    y\low < F(x(t)) < y\high ~~ \text{for}~~ t \in (h_k\low, h_k\high).
  \end{split}
\end{equation}
To see that sequences satisfying condition~\eqref{eqn:sequence-times} exist,
we consider traversals of the interval $[y\low, y\high]$ (see
Figure~\ref{fig:upcrossings}).  As $\liminf_k F(x_k) < y\low < y\high <
\limsup_k F(x_k)$, there exist increasing sequences $\wt{h}_k'$ and
$\wt{h}_k$ with
\begin{equation*}
  F(x(\wt{h}_k')) = y\low, ~~ F(x(\wt{h}_k)) = y\high
  ~~\text{and}~~ \wt{h}_k' < \wt{h}_k.
\end{equation*}
Then we define the last entrance and first subsequent exit times
\begin{equation*}
  h_k\low \defeq \sup \{ {h \in [\wt{h}_k', \wt{h}_k]}: 
  f(x(h)) \leq y\low \}
  ~~ \mbox{and} ~~
  h_k\high \defeq \inf \{h \in [h_k\low, \wt{h}_k] : f(x(h)) \ge y\high \}.
\end{equation*}
The continuity of $F$ and $x(\cdot)$ show that the
conclusion~\eqref{eqn:sequence-times} holds.

\begin{figure}[t]
  \begin{center}
    \begin{overpic}[width=.6\columnwidth]{
        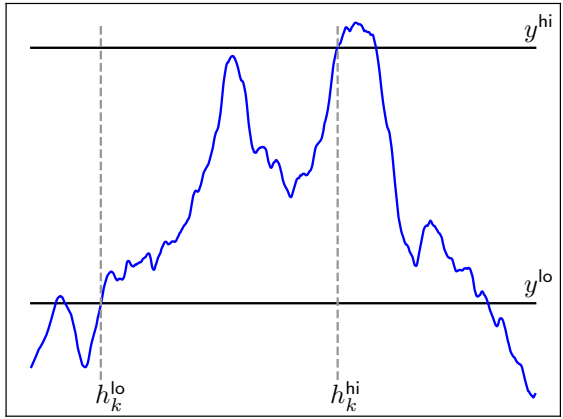}
      \put(91,67.5){$y\high$}
      \put(91,24){$y\low$}
      \put(18,5){$h_k\low$}
      \put(58,5){$h_k\high$}
    \end{overpic}
    \caption{\label{fig:upcrossings} Illustration of proof of
      Theorem~\ref{theorem:stationary-cluster-points-extension}. The erratic
      line represents a trajectory $F(x(t))$, with last entrance time
      $h_k\low$ and first exit time $h_k\high$. Such upcrossings must be
      separated in time by the strict decreases
      in Lemma~\ref{lemma:extreme-value-stationary-point}.}
  \end{center}
\end{figure}

By taking a subsequence if necessary, we assume w.l.o.g.\ that $x(h_k\high) \to
x_\infty$. By continuity, we have $y\high = F(x_\infty)$
and $x_\infty \not\in X\opt$ as $y\high \not \in F(X\opt)$.  Now,
fix some $\tau > 0$, and take $y_\tau = \liminf_{k} F(x(h_k\high - \tau))$, which
satisfies $y_\tau > F(x_\infty) = y\high$ by
Lemma~\ref{lemma:extreme-value-stationary-point}, because
$x_\infty \not\in X\opt$.  Consider the value gap $\Delta =
\half \min\{|y_\tau - y\high|, |y\high - y\low|\} > 0$. The continuity of $F$
implies for some $\delta > 0$, we have $|F(x) - y\high| < \Delta$ for $x \in
\xdomain \cap \{x_\infty + \delta \ball\}$.
As $\liminf_k |F(x(h_k\low)) - F(x_\infty)| =
|y\low - y\high| > \Delta$ and $\liminf_k |F(x(h_k\high - \tau)) - F(x_\infty)| =
|y_\tau - y\high| > \Delta$, by continuity of $F$ and our choice of $\delta$, we
must have the separation
\begin{equation}
  \label{eqn:key-step-one-stationary-points}
  \liminf_k \norm{x(h_k\low) - x_\infty} > \delta, ~\text{and}~ 
  \liminf_k \norm{x(h_k\high - \tau) - x_\infty} > \delta.
\end{equation}
For this value $\delta > 0$,
consider the sequence $\{h_k^\delta\}_{k=1}^\infty$ defined by
\begin{equation*}
  h_k^\delta = \max_t\left\{t
  \mid t < h_k\high, \norm{x(t) - x(h_k\high)} = \delta\right\}.
\end{equation*}
Then using $x(h_k\high) \to x_\infty$, we have
$\liminf_k \norm{x(h_k\low) - x(h_k\high)} > \delta$ and so
\begin{equation}
  \label{eqn:separations-to-contradict}
  h_k^\delta \in [h_k\low, h_k\high]
  ~ \mbox{eventually, and} ~
  F(x(h_k^\delta)) \in [y\low,y\high]
\end{equation}
by definition~\eqref{eqn:sequence-times} of the upcrossing times.

By Eq.~\eqref{eqn:key-step-one-stationary-points} and 
that $x(h_k\high) \to x_\infty$, we have $h_k^\delta >
\max\{h_k\low, h_k\high - \tau\}$ for large enough $k$.  In particular, this
implies that $\limsup_k (h_k\high - h_k^\delta) \le \tau$.  Because the paths
$x(\cdot)$ cannot move too quickly by
Lemma~\ref{lemma:lower-bound-time-difference}, the quantity $\tau(\delta)
\defeq \liminf_k (h_k\high - h_k^\delta) \in \openleft{0}{\tau}$.  By taking
subsequences if necessary, we may assume w.l.o.g.\ that the sequence
$h_k^\delta - h_k\high \to \tau_\infty \in [\tau(\delta), \tau]$, so that
$h_k^\delta = h_k\high - \tau_k$ for $\tau_k \to \tau_\infty > 0$.  As $x(h_k\high)
\to x_\infty \not\in X\opt$, Lemma
\ref{lemma:extreme-value-stationary-point} implies that $\liminf_k
F(x(h_k^\delta)) > y\high$, contradicting the
containments~\eqref{eqn:separations-to-contradict}. This is
the desired contradiction, which gives the theorem.


\section{Experiments}
\label{sec:experiments}

The asymptotic results in the previous sections provide somewhat limited
guidance for application of the methods. To that end, in this section we
present experimental results explicating the performance of the methods as
well as comparing their performance to the deterministic prox-linear
method~\eqref{eqn:prox-linear-step-deterministic} (adapted
from~\cite[Section 5]{DrusvyatskiyLe18}). \citet{DrusvyatskiyLe18} provide a
convergence guarantee for the deterministic method that after $O(1 /
\epsilon^2)$ iterations, the method can output an $\epsilon$-approximate
stationary point, that is, a point $\what{x}$ such that there exists $x_0$
with $\norms{\what{x} - x_0} \le \epsilon$ and $\min\{\norm{g} : g \in
\partial f(x_0)\} \le \epsilon$.  These comparisons provide us a somewhat
better understanding of the practical advantages and disadvantages of the
stochastic methods we analyze.

We consider the following problem. We have
observations $b_i = \<a_i, x^\star\>^2$, $i = 1, \ldots, n$, for an unknown
vector $x^\star \in \R^d$, and we wish to find $x^\star$. This is a quadratic
system of equations, which arises (for example) in phase retrieval problems in
imaging science as well as in a number of combinatorial
problems~\cite{ChenCa15,CandesLiSo15}. The natural exact penalty form of this
system of equations yields the minimization problem
\begin{equation}
  \label{eqn:robust-pr}
  \minimize_x ~ f(x) \defeq
  \frac{1}{n} \sum_{i=1}^n |\<a_i, x\>^2 - b_i|,
\end{equation}
which is certainly of the form~\eqref{eqn:convex-composite-stochastic} with
the function $h(t) = |t|$ and $c_i(x) = \<a_i, x\>^2 - b_i$, so we may take the
sample space $\statdomain = \{1, \ldots, n\}$.  In somewhat more
general noise models, we may also assume we observe
$b_i = \<a_i, x^\star\>^2 + \noise_i$ for some noise sequence $\noise_i$; in
this case the problem~\eqref{eqn:robust-pr} is a natural robust analogue of
the typical phase retrieval problem, which uses the smooth objective
$(\<a_i, x\>^2 - b_i)^2$.  While there are a number of specialized
procedures for solving such quadratic equations~\cite{ChenCa15}, we view
problem~\eqref{eqn:robust-pr} as a natural candidate for exploration of our
algorithms' performance.

The stochastic prox-linear update of
Example~\ref{example:stochastic-prox-linear} is reasonably straightforward
to compute for the problem~\eqref{eqn:robust-pr}. Indeed, as $\nabla_x
(\<a_i, x\>^2 - b_i) = 2 \<a_i, x\> a_i$, by rescaling by the stepsize
$\stepsize_k$ we may simplify the problem to minimizing $|b + \<a, x\>| +
\half \norm{x - x_0}^2$ for some scalar $b$ and vectors $a, x_0 \in \R^d$. A
standard Lagrangian calculation shows that
\begin{equation*}
  \argmin_x \left\{ |b + \<a, x\>| + \half \norm{x - x_0}^2
    \right\}
    = x_0 - \pi(\lambda) a
    ~~ \mbox{where} ~
    \lambda = \frac{\<x_0, a\> + b}{\norm{a}^2}
\end{equation*}
and $\pi(\cdot)$ is the projection of its argument into the interval $[-1,
  1]$. The full proximal step (Example~\ref{example:stochastic-prox-point})
is somewhat more expensive, and for general weakly convex functions, it may
be difficult to estimate $\rho(\statval)$, the weak-convexity constant;
nonetheless, in Section~\ref{sec:stepsize-robustness} we use it to evaluate
its merits relative to the prox-linear updates in terms of robustness to
stepsize.  Each iteration $k$ of the deterministic prox-linear
method~\cite{Burke85,DrusvyatskiyLe18} requires solving the quadratic
program
\begin{equation}
  \label{eqn:qp-iterates}
  x_{k + 1} = \argmin_x \bigg\{
  \frac{1}{n} \sum_{i = 1}^n |\<a_i, x_k\>^2 + 2 \<a_i, x_k\> \<a_i, x -
  x_k\> - b_i| + \frac{1}{2 \stepsize} \norm{x - x_k}^2\bigg\},
\end{equation}
which we perform using Mosek via the \texttt{Convex.jl} package in
\texttt{Julia}~\cite{UdellMoZeHoDiBo14}.

Before we present our results, we describe our choices for all parameters in
our experiments. In each experiment, we let $n = 500$ and $d = 50$, and we
choose $x\opt$ uniformly from the unit sphere $\mathbb{S}^{d - 1}$.  The noise
variables $\noise_i$ are i.i.d.\ Laplacian random variables with mean $0$ and
scale parameter $\sigma$, which we vary in our experiments.  We construct the
design matrix $A \in \R^{n \times d}, A = [a_1 ~ \cdots ~ a_n]^T$, where each
row is a measurement vector $a_i$, as follows: we choose
$U \in \R^{n\times d}$ uniformly from the orthogonal matrices in
$\R^{n\times d}$, i.e., $U^T U = I_{d \times d}$. We then make one of two
choices for $A$, the first of which controls the condition number of $A$ and
the second the regularity of the norms of the rows $a_i$. In the former case,
we set $A = UR$, where $R \in \diag(\R^d) \subset \R^{d \times d}$ is diagonal
with linearly spaced diagonal elements in $[1, \kappa]$, so that
$\kappa \ge 1$ gives the condition number of $A$. In the latter case, we set
$A = RU$, where $R \in \diag(\R^n) \subset \R^{n \times n}$ is again diagonal
with linearly spaced elements in $[1, \kappa]$. Finally, in each of our
experiments, we set the stepsize for the stochastic methods as
$\stepsize_k = \stepsize_0 k^{-\beta}$, where $\stepsize_0 > 0$ is the initial
stepsize and $\beta \in (\half, 1)$ governs the rate of decrease in stepsize.
We present three experiments in more detail in the coming subsections: (i)
basic performance of the algorithms, (ii) the role of conditioning in the data
matrix $A$, and (iii) an analysis of stepsize sensitivity for the different
stochastic methods, that is, an exploration of the effects of the choices of
$\stepsize_0$ and $\beta$ in the stepsize choice.

\subsection{Performance for well-conditioned problems}

\begin{figure}[t]
  \begin{center}
    \begin{tabular}{cc}
      \hspace{-.4cm}
      \begin{overpic}[width=.47\columnwidth]{
          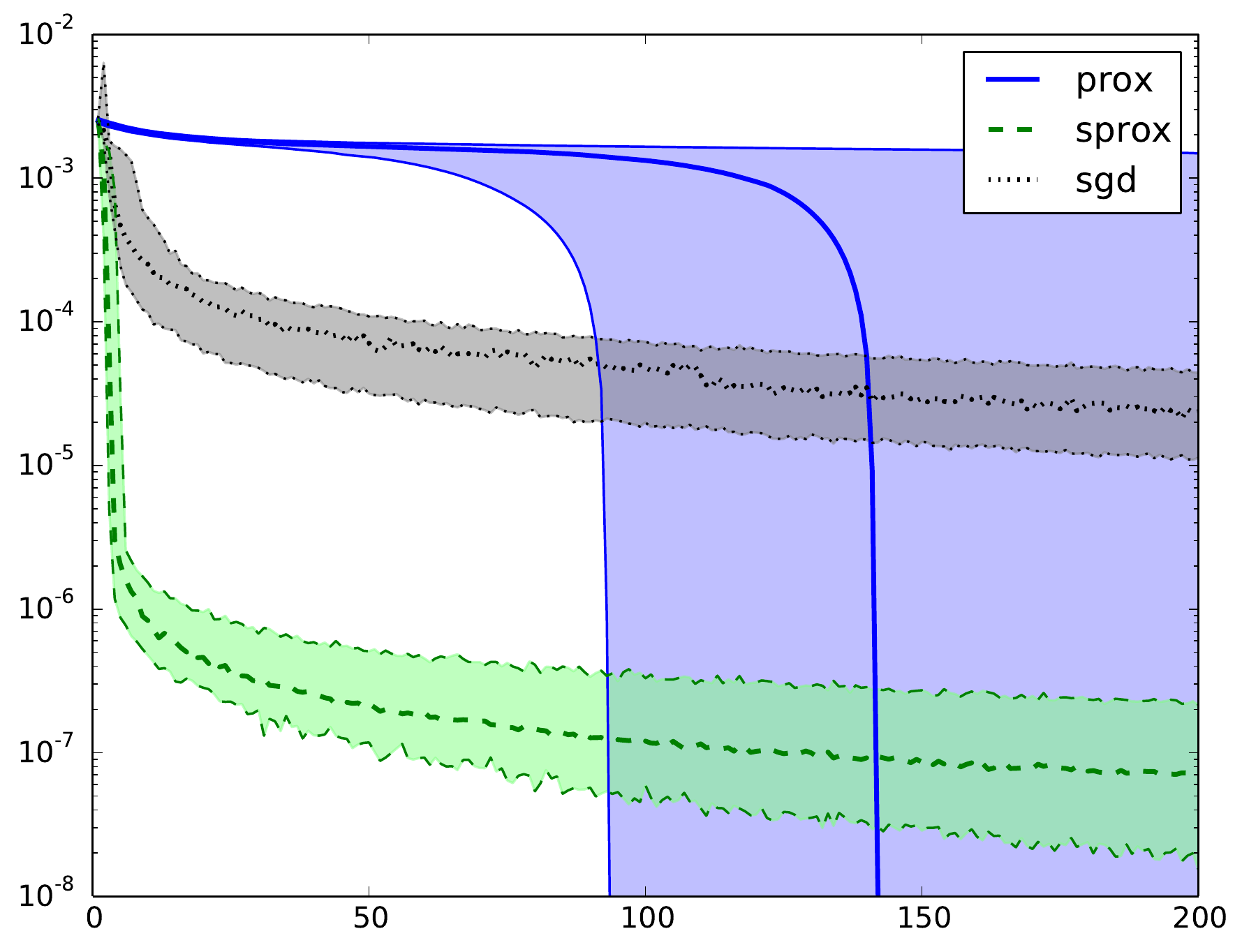}
        \put(83,60){
          \tikz{\path[draw=white,fill=white] (0, 0) rectangle (.6cm,.7cm);}}
        \put(84.4,69.4){\scriptsize {\sf prox}}
        \put(84.4,65.55){\scriptsize {\sf s.prox}}
        \put(84.4,61.7){\scriptsize {\sf sgm}}
      \end{overpic}
      & \hspace{-.4cm}
      \begin{overpic}[width=.47\columnwidth]{
          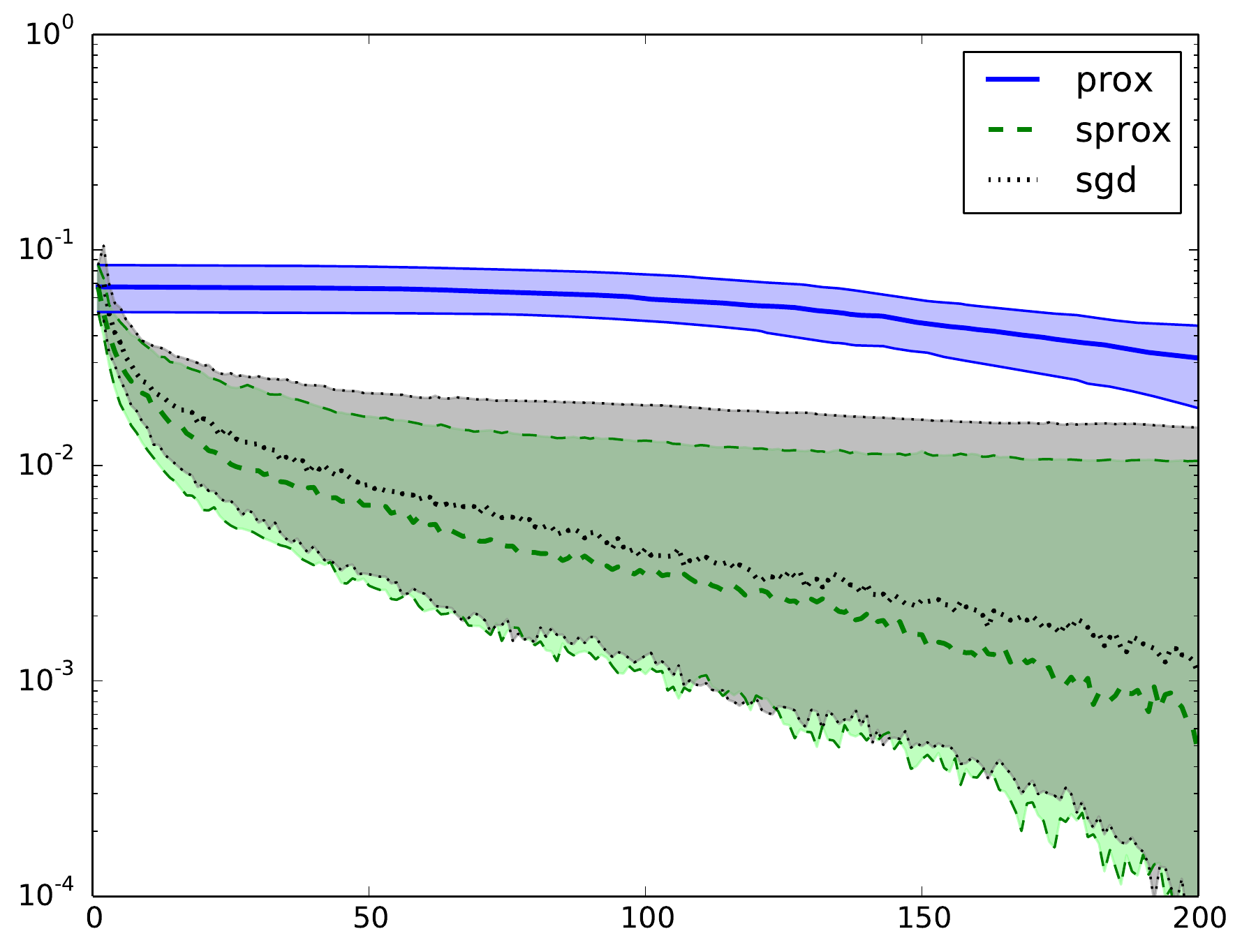}
        \put(83,60){
          \tikz{\path[draw=white,fill=white] (0, 0) rectangle (.6cm,.7cm);}}
        \put(84.4,69.4){\scriptsize {\sf prox}}
        \put(84.4,65.55){\scriptsize {\sf s.prox}}
        \put(84.4,61.7){\scriptsize {\sf sgm}}
      \end{overpic} \\
      (a) Noiseless & (b) Noisy \\
      \hspace{-.4cm}
      \begin{overpic}[width=.47\columnwidth]{
          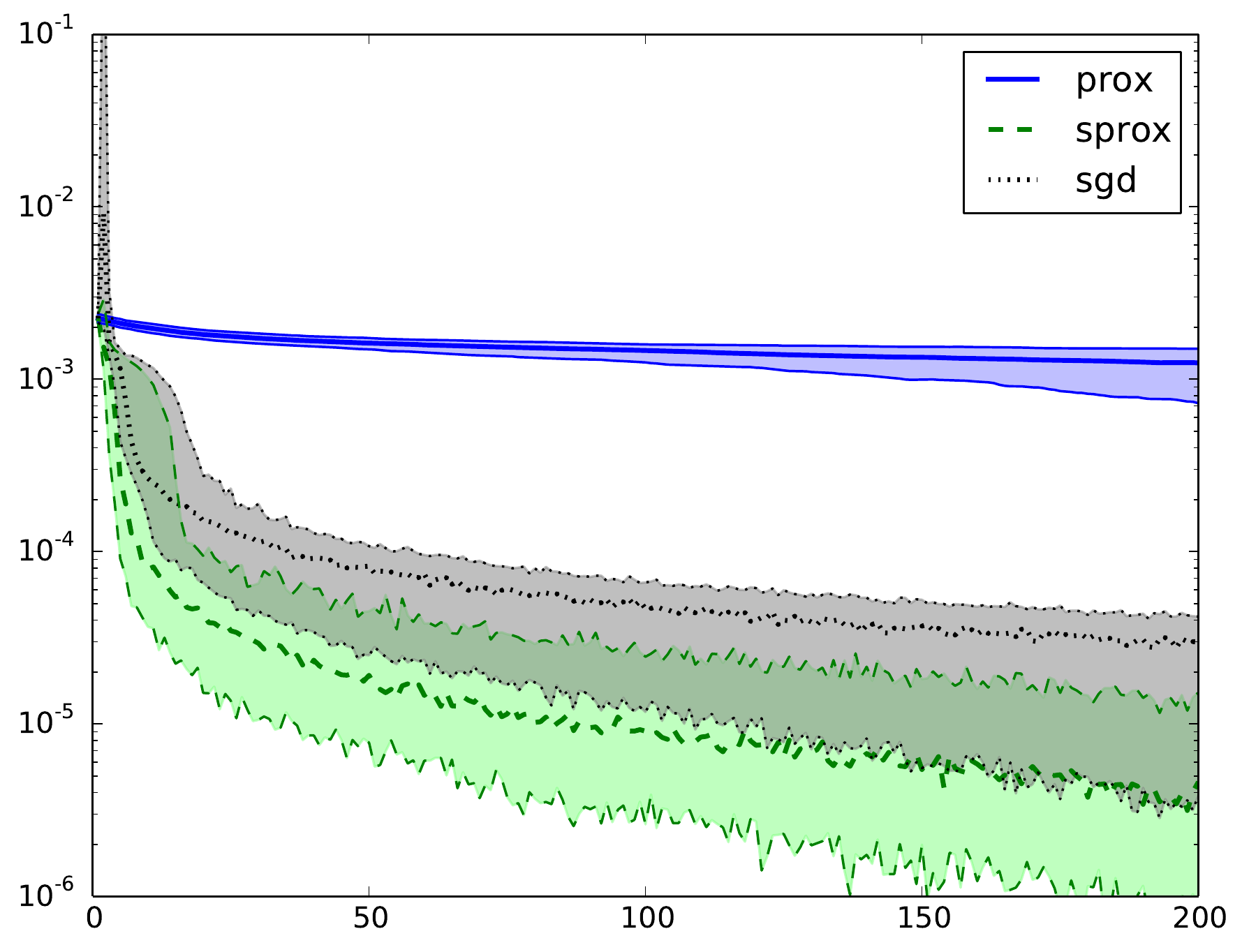}
        \put(83,60){
          \tikz{\path[draw=white,fill=white] (0, 0) rectangle (.6cm,.7cm);}}
        \put(84.4,69.4){\scriptsize {\sf prox}}
        \put(84.4,65.55){\scriptsize {\sf s.prox}}
        \put(84.4,61.7){\scriptsize {\sf sgm}}
      \end{overpic}
      &
        \hspace{-.3cm} \begin{minipage}{.5\columnwidth}
          \vspace{-4cm}
          \caption{\label{fig:well-conditioned} Experiments with
            well-conditioned $A$ matrices. The vertical axis shows $f(x_k) -
            f(x\opt)$, the horizontal axis iteration count for the
            prox-linear iteration~\eqref{eqn:qp-iterates} or the $k / n$-th
            iteration for the stochastic methods. The key: {\sf prox} is the
            prox-linear iteration~\eqref{eqn:prox-linear-step-deterministic},
            {\sf s.prox} is the stochastic prox-linear
            method (Ex.~\ref{example:stochastic-prox-linear}), {\sf sgm} is
            the stochastic subgradient method (Ex.~\ref{example:sgd}).  (a)
            Methods with no noise.  (b) $\noise_i \simiid$ Laplacian with
            scale $\sigma = 1$.  (c) Proportion $p = .1$ of observations
            $b_i$ corrupted arbitrarily.}
        \end{minipage} \\
      (c) Corrupted &
    \end{tabular}
  \end{center}
\end{figure}
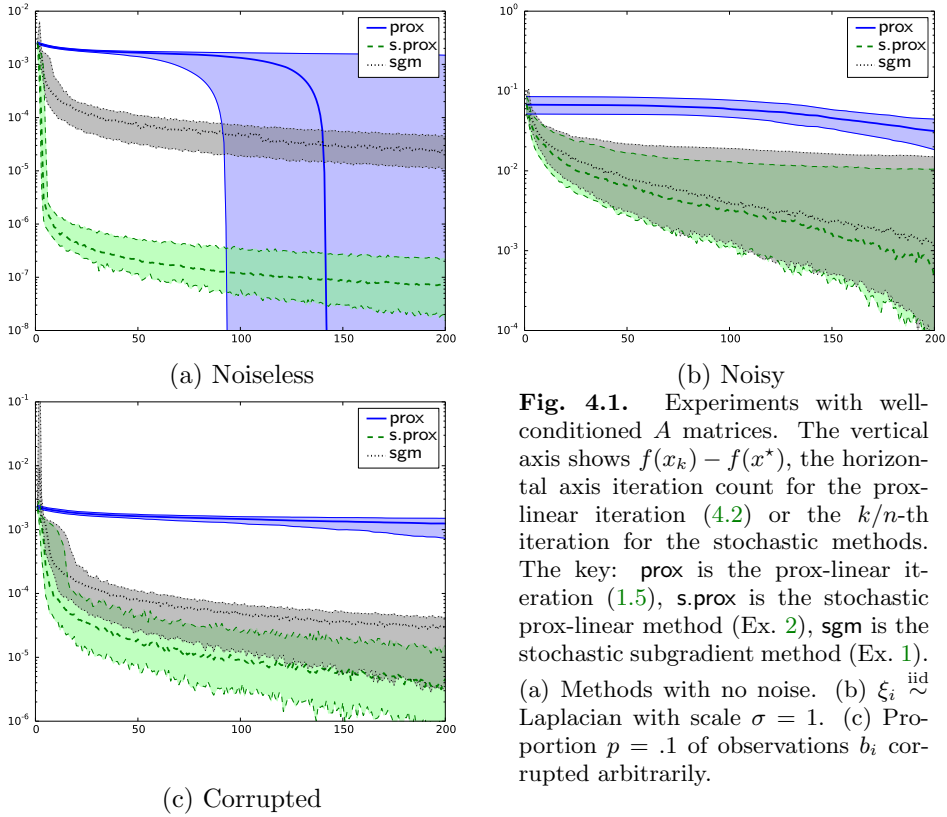

In our first group of experiments, we investigate the performance of the three
algorithms under noiseless and noisy observational situations. In each of
these experiments, we set the condition number $\kappa = \kappa(A) = 1$. We
consider three experimental settings to compare the procedures: in the first,
we have noiseless observations $b_i = \<a_i, x\opt\>^2$; in the second, we set
$b_i = \<a_i, x\opt\>^2 + \noise_i$ where $\noise_i$ are Laplacian with scale
$\sigma = 1$; and in the third, we again have noiseless observations
$b_i = \<a_i, x\opt\>^2$, but for a fraction $p = .1$ of the observations, we
replace $b_i$ with an independent $\normal(0, 25)$ random variable, so that
$n/10$ of the observations provide no information. Within each experimental
setting, we perform $N = 100$ independent tests, and in each individual test
we allow the stochastic methods to perform $N = 200 n$ iterations (so
approximately 200 loops over the data). For the deterministic prox-linear
method~\eqref{eqn:qp-iterates}, we allow 200 iterations.  Each deterministic
iteration is certainly more expensive than $n$ (sub)gradient steps or
stochastic prox-linear steps, but it provides a useful benchmark for
comparison.  The stochastic methods additionally require specification of the
initial stepsize $\stepsize_0$ and power $\beta$ for
$\stepsize_k = \stepsize_0 k^{-\beta}$, and to choose this, we let
$\stepsize_0 \in \{1, 10, 10^2, 10^3\}$ and $\beta \in \{.6, .7, .8, .9\}$,
perform $3n$ steps of the stochastic method with each potential pair
$(\stepsize_0, \beta)$, and then perform the full $N = 200n$ iterations with
the best performing pair.  We measure performance of the methods within each
test by plotting the gap $f(x_k) - f(x\opt)$, where we approximate $x\opt$ by
taking the best iterate $x_k$ produced by any of those methods.  While the 
problem is non-convex and thus may have
spurious local minima, these gaps provide a useful quantification of
(relative) algorithm performance.

We summarize our experimental results in Figure~\ref{fig:well-conditioned}.
In each plot, we plot the median of the excess gap $f(x_k) - f(x\opt)$ as well
as its $10\%$ and $90\%$ confidence intervals over our $N = 100$ tests.  In
order to compare the methods, the horizontal axis scales as iteration $k$
divided by $n$ for the stochastic methods and as iteration for the
deterministic method~\eqref{eqn:qp-iterates}.  Each of the three methods is
convergent in these experiments, and the stochastic methods exhibit fast
convergence to reasonably accurate (say $\epsilon \approx 10^{-4}$) solutions
after a few passes through the data. Eventually (though we do not always plot
such results) the deterministic prox-linear algorithm achieves substantially
better accuracy, though its progress is often slower.  This corroborates
substantial experience from the convex case with stochastic
methods~\cite[c.f.][]{NemirovskiJuLaSh09, DuchiHaSi11}.

There are differences in behavior for the different methods, which we can
heuristically explain. In Fig.~\ref{fig:well-conditioned}(a), the stochastic
prox-linear method (Ex.~\ref{example:stochastic-prox-linear}) converges
substantially more quickly than the stochastic subgradient method. Intuitively,
we expect this behavior because each data point $(a_i, b_i)$ should have
$\<a_i, x\>^2 = b_i$ exactly, and the precise stepping of the prox-linear
method achieves this more easily. In Fig.~\ref{fig:well-conditioned}(b),
where $b_i = \<a_i, x\opt\>^2 + \noise_i$ the two methods have similar
behavior; in this case, the population expectation $f_{\rm pop}(x) = \E[|b -
  \<a, x\opt\>|^2]$ is smooth, because the noise $\noise$ has a density, so
gradient methods are likely to be reasonably effective. Moreover, with
probability $1$ we have $\<a_i, x\opt\>^2 \neq b_i$, so that the precision
of the prox-linear step is unnecessary. Finally,
Fig.~\ref{fig:well-conditioned}(c) shows that the methods are robust to
corruption, but because we have $\<a_i, x\opt\>^2 = b_i$ for the majority of
$i \in \{1, \ldots, n\}$, there is still benefit to using the more exact
(stochastic) prox-linear iteration. We note in passing that the gap in
function values $f(x_k)$ between the stochastic prox-linear method and
stochastic subgradient method (SGM)
is statistically significantly positive at the $p = 10^{-2}$ level for
iterations $k = 1, \ldots, 20$, and that at each iteration $k$, the
prox-linear method outperforms SGM for at least 77 of the $N = 100$
experiments (which is statistically significant for rejecting the hypothesis
that each is equally likely to achieve lower objective value than the other
at level $p = 10^{-6}$).

\subsection{Problem conditioning and observation irregularity}

\begin{figure}[ht]
  \begin{center}
    \begin{tabular}{cc}
      \hspace{-.4cm}
      \begin{overpic}[width=.48\columnwidth]{
          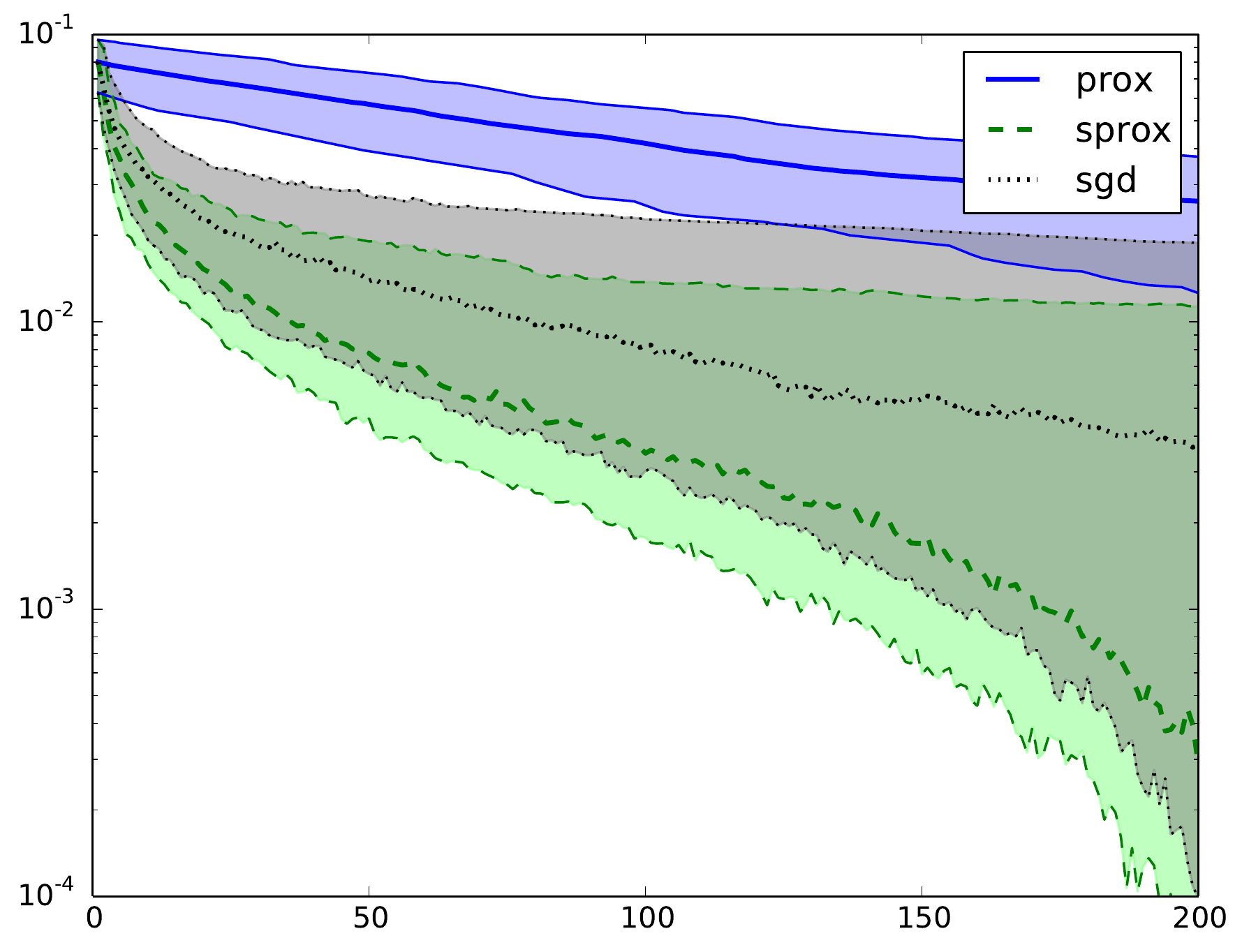}
        \put(83,60){
          \tikz{\path[draw=white,fill=white] (0, 0) rectangle (.6cm,.7cm);}}
        \put(84.4,69.4){\scriptsize {\sf prox}}
        \put(84.4,65.55){\scriptsize {\sf s.prox}}
        \put(84.4,61.7){\scriptsize {\sf sgm}}
      \end{overpic}
      & \hspace{-.4cm}
      \begin{overpic}[width=.48\columnwidth]{
          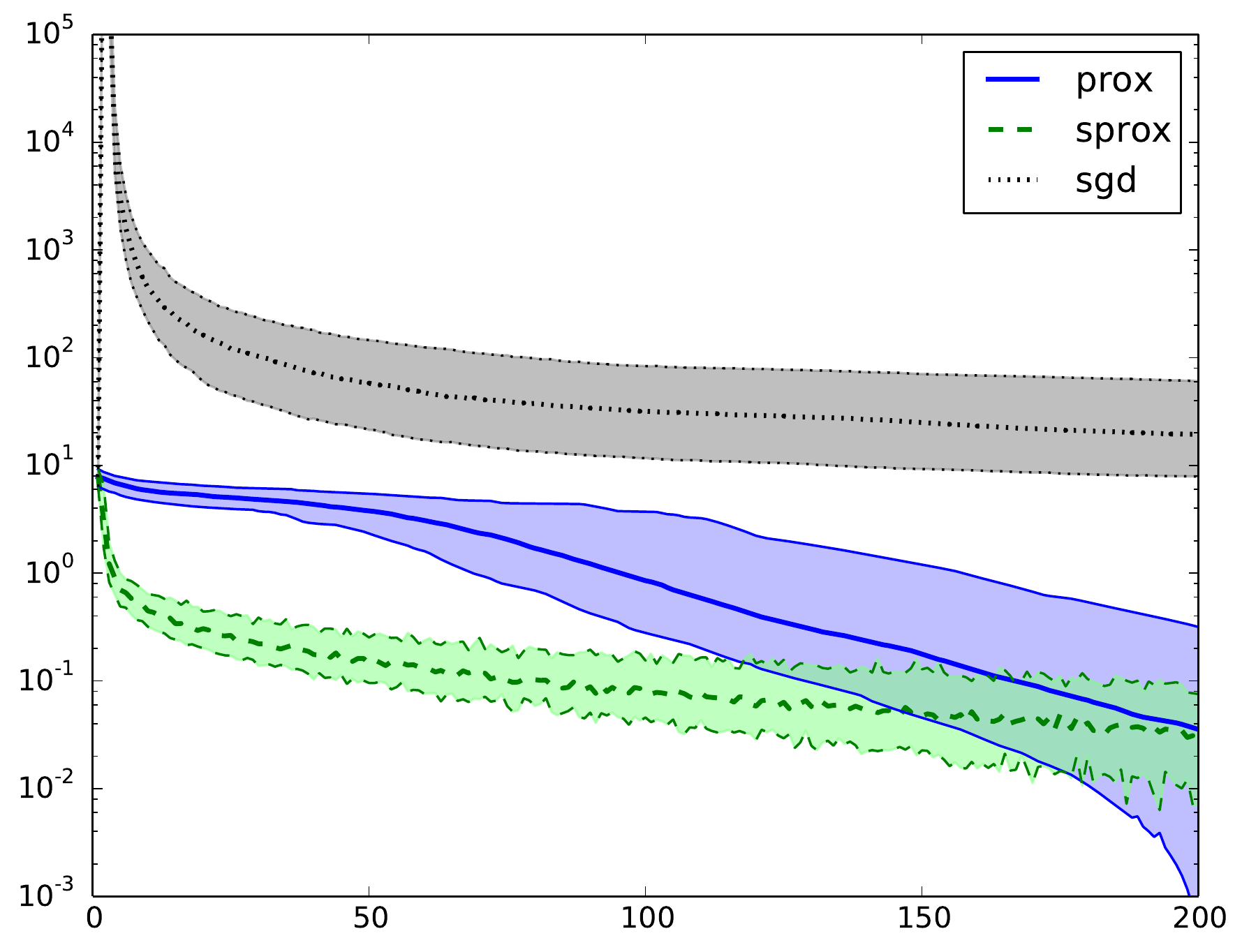}
        \put(83,60){
          \tikz{\path[draw=white,fill=white] (0, 0) rectangle (.6cm,.7cm);}}
        \put(84.4,69.4){\scriptsize {\sf prox}}
        \put(84.4,65.55){\scriptsize {\sf s.prox}}
        \put(84.4,61.7){\scriptsize {\sf sgm}}
      \end{overpic}
      \\
      (a) $A = UR$, $\kappa(A) = 10$ & (b) $A = UR$, $\kappa(A) = 100$ \\
      \hspace{-.4cm}
      \begin{overpic}[width=.48\columnwidth]{
          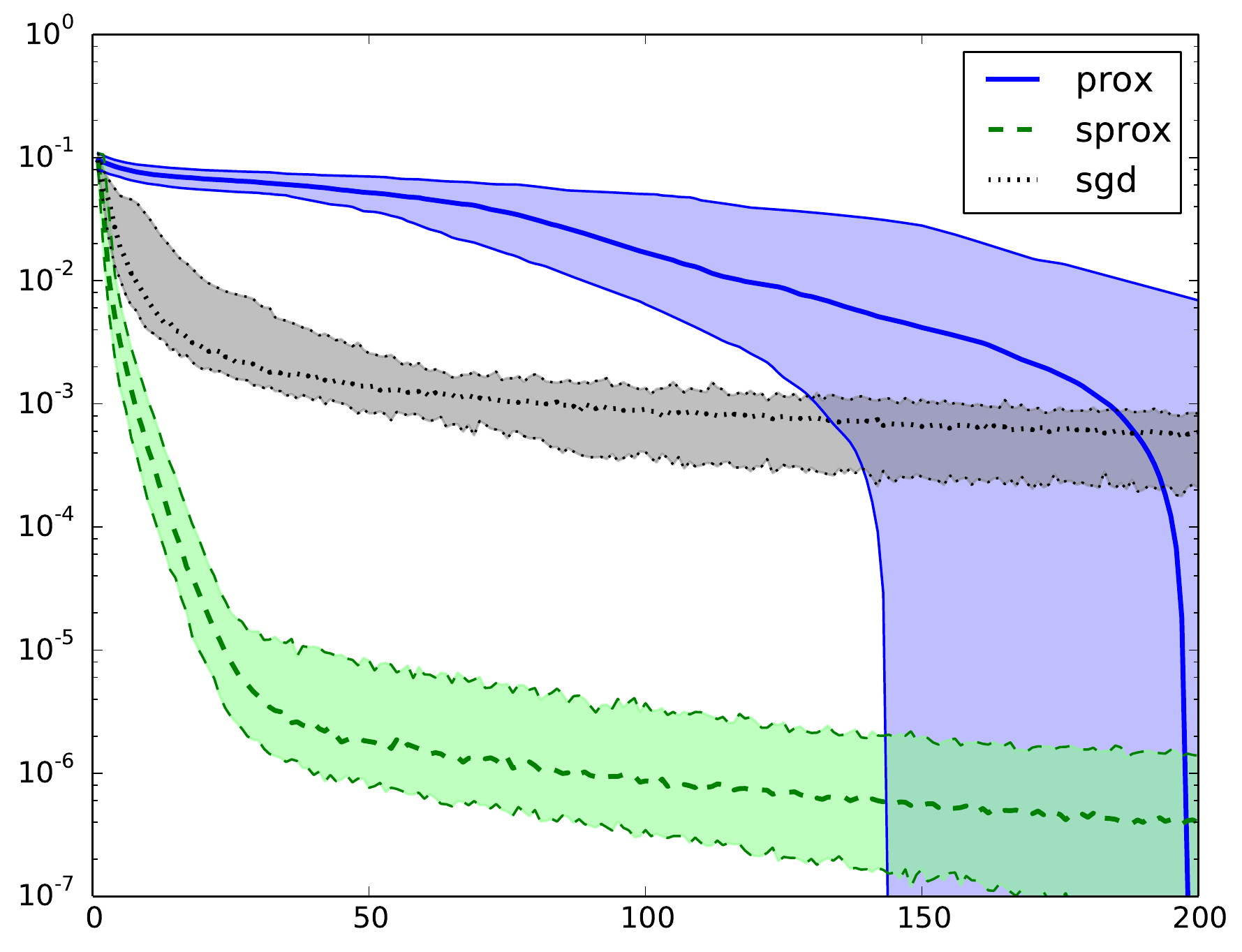}
        \put(83,60){
          \tikz{\path[draw=white,fill=white] (0, 0) rectangle (.6cm,.7cm);}}
        \put(84.4,69.4){\scriptsize {\sf prox}}
        \put(84.4,65.55){\scriptsize {\sf s.prox}}
        \put(84.4,61.7){\scriptsize {\sf sgm}}
      \end{overpic}
      & \hspace{-.4cm}
      \begin{overpic}[width=.48\columnwidth]{
          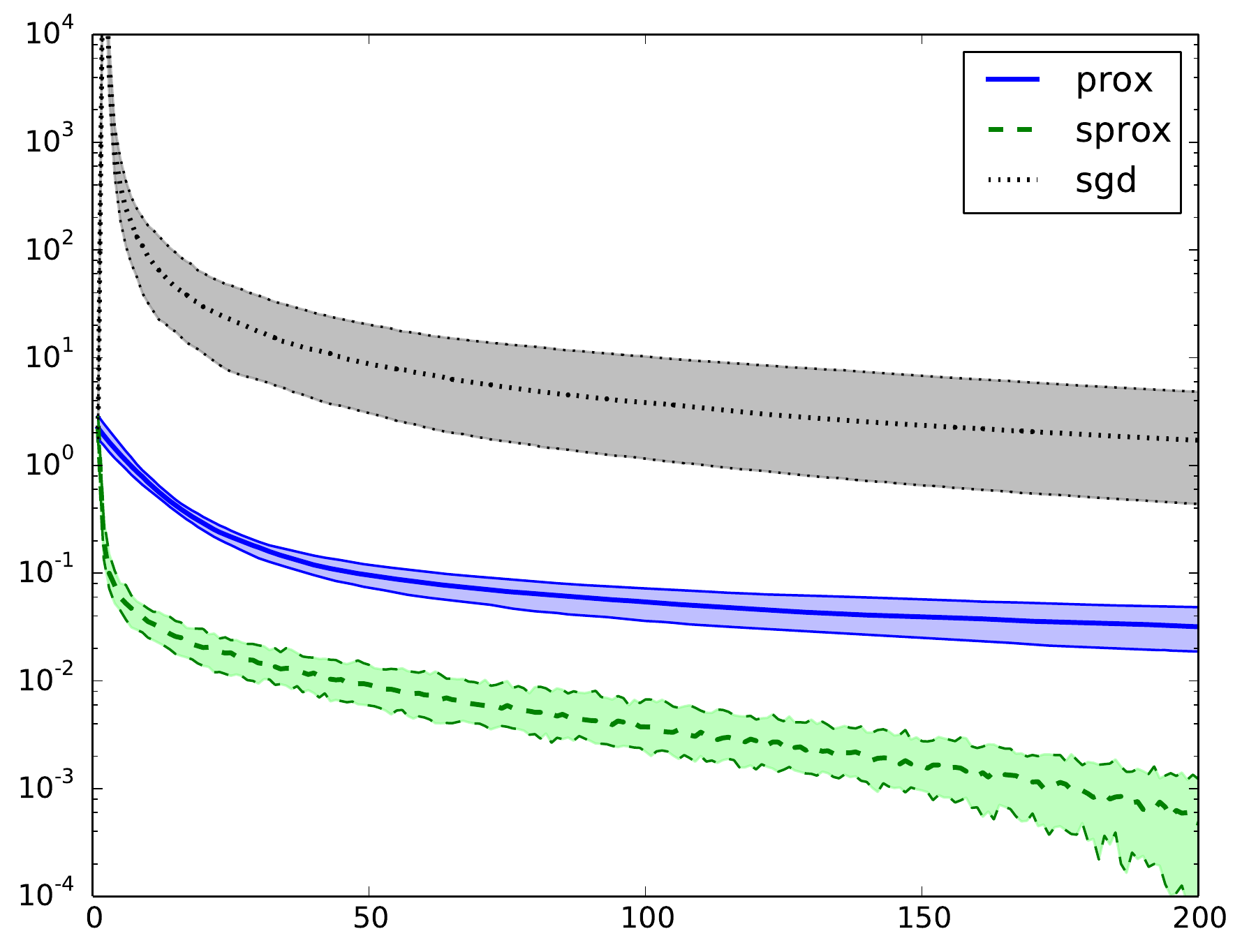}
        \put(83,60){
          \tikz{\path[draw=white,fill=white] (0, 0) rectangle (.6cm,.7cm);}}
        \put(84.4,69.4){\scriptsize {\sf prox}}
        \put(84.4,65.55){\scriptsize {\sf s.prox}}
        \put(84.4,61.7){\scriptsize {\sf sgm}}
      \end{overpic}
      \\
      (c) $A = RU$, $b_i = \<a_i, x\opt\>^2$
      & (d)
        $A = RU$, $b_i = \<a_i, x\opt\>^2 + \noise_i$
    \end{tabular}
    \caption{\label{fig:condition-number} Experiments with 
    $A$ matrices of varying condition number and irregularity
    in row norms.}
  \end{center}
\end{figure}
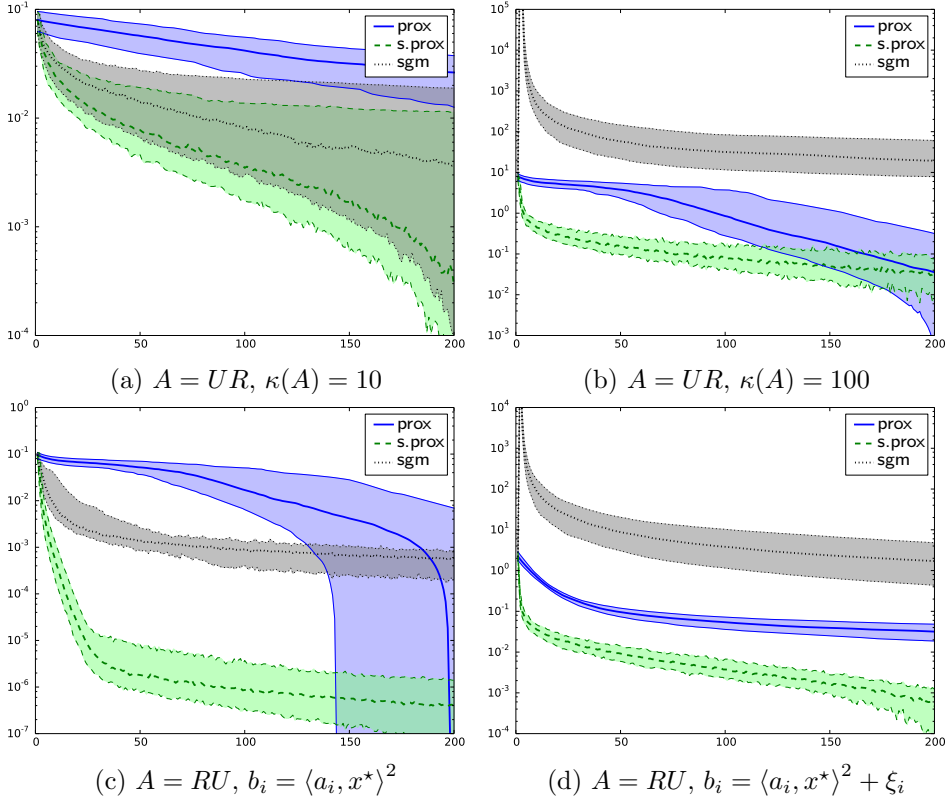

In our second set of experiments, we briefly investigate conditioning of the
problem~\eqref{eqn:robust-pr} by modifying the condition number
$\kappa = \kappa(A)$ of the measurement matrix $A \in \R^{n \times d}$ or by
modifying the relative norms of the rows $\norm{a_i}$ of $A$.  In each of the
experiments, we choose the initial stepsize $\stepsize_0$ and power $\beta$ in
$\stepsize_k = \stepsize_0 k^{-\beta}$ using the same heuristic as the
previous experiment for the stochastic methods (by considering a grid of
possible values and selecting the best after $3n$ iterations).  We present
four experiments, whose results we summarize in
Figure~\ref{fig:condition-number}. As in the previous experiments, we plot the
gaps $f(x_k) - f(x\opt)$ versus iteration $k$ (for the deterministic
prox-linear method) and versus iteration $k/n$ for the stochastic methods.  In
the first two, we use observations $b_i = \<a_i, x\opt\>^2 + \noise_i$, where
the noise variables are i.i.d.\ Laplacian with scale $\sigma = 1$, and we set
$A = UR$ where $R$ is diagonal, in the first
(Fig.~\ref{fig:condition-number}(a)) scaling between $1$ and $\kappa = 10$ and
in the second (Fig.~\ref{fig:condition-number}(b)) scaling between $1$ and
$\kappa = 100$.  Each method's performance degrades as the condition number
$\kappa = \kappa(A)$ increases, as one would expect. The performance of SGM
degrades substantially more quickly with the conditioning of the matrix $A$,
in spite of the fact that noisy observations improve its performance relative
to the other methods (in the case $\sigma = 0$, SGM's relative performance is
worse).

In the second two experiments, we set $A = RU$, where $R$ is diagonal with
entries linearly spaced in $[1, \kappa]$ for $\kappa = 10$, so that the norms
$\norm{a_i}$ are irregular (varying by approximately a factor of
$\kappa = 10$). In the first of the experiments
(Fig.~\ref{fig:condition-number}(c)), we set the observations
$b_i = \<a_i, x\opt\>^2$ with no noise, while in the second
(Fig.~\ref{fig:condition-number}(d)) we set
$b_i = \<a_i, x\opt\>^2 + \noise_i$ for $\noise_i$ i.i.d.\ Laplacian with
scale $\sigma = 1$. In both cases, the stochastic prox-linear method has
better performance---this is to be expected, because its more exact updates
involving the linearization
$h(c(x_k; \statval) + \nabla c(x_k; \statval)^T(x - x_k); \statval)$ are more
robust to scaling of $\norm{a_i}$.
As we explore more carefully in the next set of experiments, one implication
of these results is that the robustness and stability of the stochastic
prox-linear algorithm with respect to problem conditioning is reasonably good,
while the behavior of stochastic subgradient methods can be quite sensitive
to conditioning behavior of the design matrix $A$.

\subsection{Robustness of stochastic methods to stepsize}
\label{sec:stepsize-robustness}

\begin{figure}[t]
  \begin{center}
    \begin{tabular}{cc}
      \hspace{-1cm}
      \begin{overpic}[width=.55\columnwidth,clip]{
          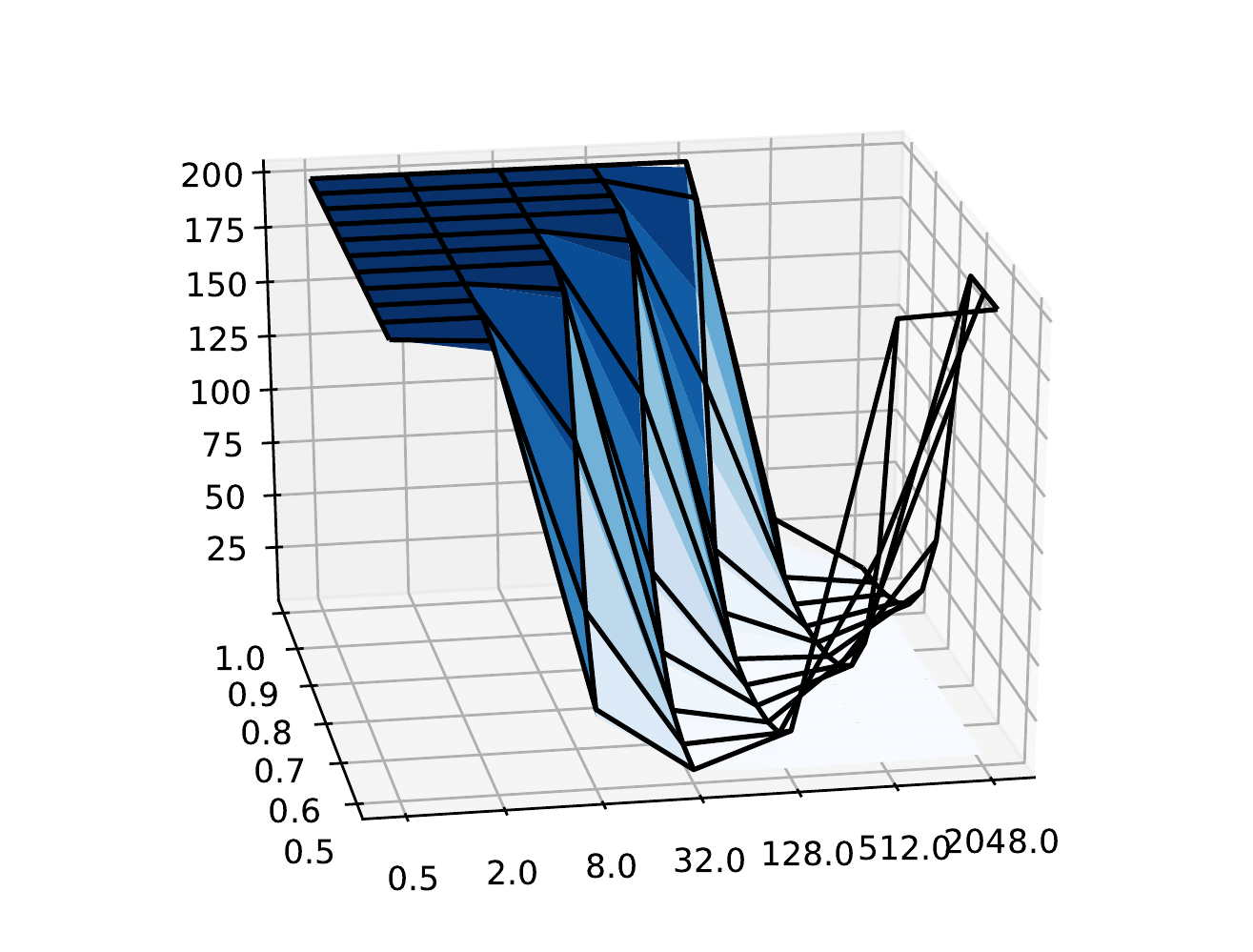}
        \put(55,3){$\stepsize_0$}
        \put(14,11){$\beta$}
        \put(9,39){\rotatebox{90}{$T(\epsilon)$}}
      \end{overpic}
      &
      \hspace{-.5cm}
      \begin{overpic}[width=.48\columnwidth]{
          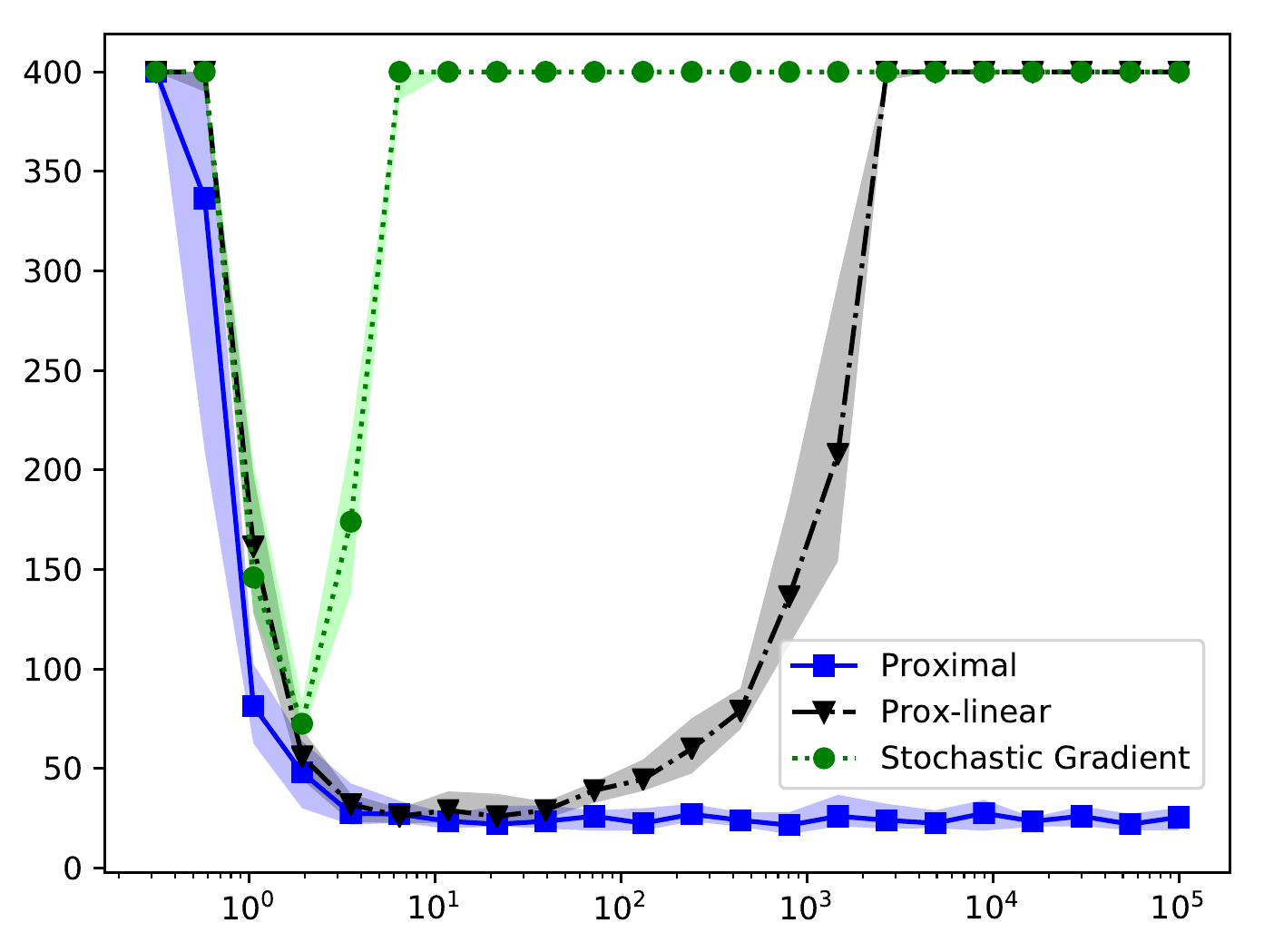}
        \put(51,0){$\stepsize_0$}
        \put(-5,35){\rotatebox{90}{$T(\epsilon)$}}
        \put(68,13){
          \tikz{\path[draw=white,fill=white] (0, 0) rectangle (1.5cm,.7cm);}}
        \put(69,22){\tiny St.\ prox-point}
        \put(69,18){\tiny St.\ prox-linear}
        \put(69,14){\tiny St.\ gradient}
      \end{overpic}
    \end{tabular}
    \caption{\label{fig:stopping-time} Iterations to achieve
      $\epsilon$-accuracy. Left: time to convergence versus $\stepsize_0$
      and $\beta$ for stochastic subgradient (wireframe) and stochastic
      prox-linear (solid surface) methods. Right: time to convergence versus
      initial stepsize $\stepsize_0$ with $\beta = \half$ for stochastic
      proximal, prox-linear, and gradient methods.}
    \vspace{-.2cm}
  \end{center}
\end{figure}
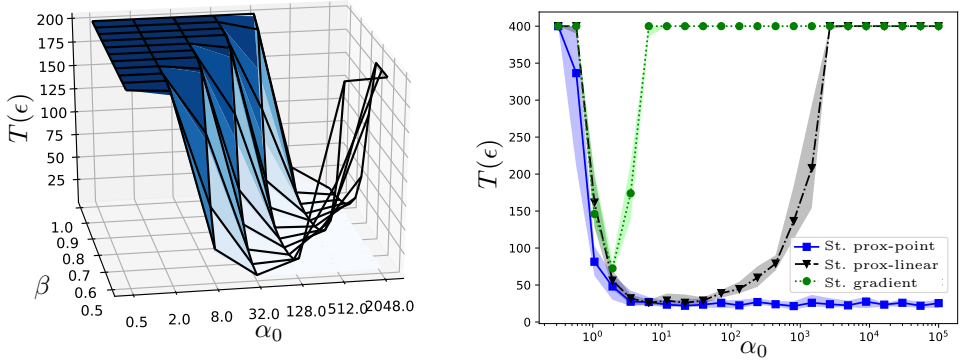


In our final experiment, we investigate the effects of stepsize parameters for
the behavior of our stochastic methods. For stepsizes
$\stepsize_k = \stepsize_0 k^{-\beta}$, the stochastic methods require
specification of both the parameter $\stepsize_0$ and $\beta$, so it is
interesting to investigate the robustness of the stochastic prox-linear method
and SGM to various settings of $\stepsize_0$ and $\beta$.  In each of these
experiments, we set the condition number $\kappa(A) = 1$ and have no noise,
i.e.\ $b_i = \<a_i, x\opt\>^2$.  We vary the initial stepsize
$\stepsize_0 \in \{2^{-1}, 2^1, 2^3, \ldots, 2^{11}\}$ and the power
$\beta \in \{0.5, 0.55, 0.6, \ldots, 1\}$.
In this experiment, we have $f(x\opt) = 0$, and we investigate
the number of iterations
\begin{equation*}
  T(\epsilon)
  \defeq \inf\left\{k \in \N \mid f(x_k) \le \epsilon\right\}
\end{equation*}
required to find and $\epsilon$-optimal solution. (In our experiments, the
stochastic methods always find such a solution eventually.)  We perform $N =
250$ tests for each setting of the pairs $\stepsize_0, \beta$, and in each
test, we implement all three of the stochastic gradient
(Ex.~\ref{example:sgd}), prox-linear
(Ex.~\ref{example:stochastic-prox-linear}), and proximal-point
(Ex.~\ref{example:stochastic-prox-point}) methods, each for $k = 200n$
iterations, setting $T(\epsilon) = 200n$ if no iterate $x_k$
satisfies $f(x_k) \le \epsilon$.

Figure~\ref{fig:stopping-time} illustrates the results of these experiments,
where the vertical axis gives the median time $T(\epsilon)$ to $\epsilon =
10^{-2}$-accuracy over all the $N = 250$ tests.  The left plot demonstrates
convergence time of the stochastic prox-linear and subgradient
methods versus the initial stepsize $\stepsize_0$ and power $\beta$,
indicated on the horizontal axes. The solid white-to-blue surface, with thin
lines, corresponds to the iteration counts for the stochastic prox-linear
method; the transparent surface with thicker lines corresponds to the
iteration counts for the stochastic subgradient method.
Figure~\ref{fig:stopping-time} shows that the stochastic prox-linear
algorithm consistently has comparable or better performance than SGM for the
same choices of parameters $\stepsize_0, \beta$. The right plot shows
convergence of the stochastic proximal-point method (see Example
\ref{example:stochastic-prox-point} in Sec~\ref{sec:example}), stochastic
prox-linear 
method, and stochastic subgradient method versus stepsize on a log-plot of
initial stepsizes, with $\beta = \half$ fixed. The most salient aspect of
the figures is that the stochastic prox-linear and proximal-point methods
are more robust to stepsize (mis-)specification than is SGM. Indeed,
Fig.~\ref{fig:stopping-time} makes apparent, the range of stepsizes yielding
good performance for SGM is a relatively narrow valley, while the
prox-linear and proximal-point methods enjoy reasonable
performance for broad choices of (often large) stepsizes $\stepsize_0$, with
less sensitivity to the rate of decrease $\beta$ in the stepsize as
well. This behavior is expected: the iterations of the stochastic
prox-linear and proximal-point methods
(Exs.~\ref{example:stochastic-prox-linear}--\ref{example:stochastic-prox-point})
guard more carefully against wild swings that result from aggressive
stepsize choices, yielding more robust convergence and easier stepsize
selection.
%

\appendix

\section{Proof of Theorem~\ref{theorem:functional-convergence-general}}
\label{sec:proof-thm-functional-convergence}

To prove the theorem, we require a few pieces of additional notation. 
Define the indices
\begin{equation*}
  \klow(t) = \max\{k \in \N : t_k \le t\}
  ~~ \mbox{and} ~~
  \khigh(t) = \min\{k \in \N : t_k \ge t\}
  ~~ \mbox{for} ~ t \in \R_+.
\end{equation*}
For $t \in \R_+$, let $x_t(\cdot)$ denote the solution to
the differential equation
\begin{equation}
  \label{eqn:diff-eq-no-noise}
  \dot{x}_t(\cdot) = y(\cdot) ~~\text{and}~~x_t(t) = x(t),
\end{equation}
which is evidently $x_t(\tau) = x(t) + \int_t^\tau y(u) du$.  We divide the
proof into two main parts: first, we show that the function family
$x^\tau(\cdot)$ is relatively compact in $\cont(\R_+, \R^d)$ using the
Arzel\`{a}-Ascoli theorem. In the second part, we apply a few
functional-analytic results on weak convergence to establish that
$x^\tau(\cdot)$ has the limiting differential inclusion properties claimed in
the theorem.

\paragraph{Part I: Relative compactness}
We begin with a lemma that shows an equivalence between the time-shifted and
interpolated~\eqref{eqn:interpolation} sequence $x^t(\cdot) = x(t + \cdot)$
and the differential sequence $x_t(\cdot)$ as $t \to \infty$, showing that the
noise and discretization effects in the
iteration~\eqref{eqn:stochastic-approximation} are essentially negligible.  As
we show presently, the relative compactness of $x_t(\cdot)$ is reasonably
straightforward, so this equivalence simplifies our arguments.
\begin{lemma}
  \label{lemma:asymptotic-equivalence}
  For any $T \ge 0$, we have 
  \begin{equation*}
    \lim_{t \rightarrow \infty} \sup_{\tau \in [t, t+T]} 
    \norm{x(\tau) - x_t(\tau)} = 0.
  \end{equation*}
\end{lemma}
\begin{proof}
  Fix $\tau \in [t, t+T]$ and
  let $p = \khigh(t)$ and $q = \klow(\tau)$. Then
  by definition of $x_t$ we have
  \begin{align}
    x_t(\tau)
    &= x(t) + \int_{t}^{t_p} y(u) du + \int_{t_p}^{t_q} y(u) du
      + \int_{t_q}^\tau y(u) du
    \nonumber \\
    &= x(t_p) + \sum_{i=p}^{q-1} \stepsize_i y_i + 
      \left(x(t) + \int_{t}^{t_p} y(u) du - x(t_p)\right) + \int_{t_q}^\tau
      y(u) du
    \nonumber \\
    & = x(t_q) - \sum_{i=p}^{q-1} \stepsize_i \noise_i + 
      \left(x(t) + \int_{t}^{t_p} y(u) du - x(t_p)\right) + \int_{t_q}^\tau
      y(u) du
    \nonumber \\
     \begin{split} 
    & = x(\tau) - \sum_{i=p}^{q-1} \stepsize_i \noise_i + 
      \left(x(t) + \int_{t}^{t_p} y(u) du - x(t_p)\right)  \\
    &~~~~~~~~~~~~~~~~~~~~~~~~
    	 + \left(x(t_q) + \int_{t_q}^\tau y(u) du - x(\tau)\right).
      \label{eqn:x-sub-x-real}
     \end{split}
  \end{align}
  Taking norms and appling the triangle inequality, we have
  \begin{align*}
    \norm{x(t) + \int_{t}^{t_p} y(u) du - x(t_p)}
    &\leq \norm{x(t_{p}) - x(t_{p-1})} + \int_{t}^{t_p} \norm{y(u)} du \\ 
    &\leq \stepsize_{p-1}\left( \norm{\noise_{p-1}} + 2 \norm{y_{p-1}}\right),
  \end{align*}
  and similarly
  \begin{align*}
    \norm{x(t_q) + \int_{t_q}^\tau y(u) du - x(\tau)}
    &\leq \norm{x(t_{q+1}) - x(t_q)} + \int_{t_q}^\tau \norm{y(u)} du \\
    &\le \stepsize_{q} \left(\norm{\noise_q} + 2 \norm{y_q}\right).
  \end{align*}
  The convergence of the sum
  $\sum_{k = 1}^\infty \stepsize_k \noise_k$ implies the uniform guarantees
  \begin{align*}
    \lim_{p \to \infty} \sup_{q\geq p} \normbigg{
    \sum_{i=p}^{q-1} \stepsize_i \noise_i} = 0
    ~~\text{and}~~\lim_{p\to \infty} \sup_{q\geq p}
    \stepsize_q(\norm{\noise_q} + \norm{y_q} = 0.
  \end{align*}
  Substituting these into the bound~\eqref{eqn:x-sub-x-real} gives the result.
\end{proof}

Now we may show the relative compactness of the function sequence
$x^{t_k}(\cdot)$ for any sequence $t_k$.  Using that
$\sup_{t} \norm{y(t)} = \sup_k \norm{y_k} < \infty$ by assumption, we have
that $\{x_t(\cdot)\}_{t \in \R_+}$ is a family of uniformly equicontinuous (even
Lipschitz) and pointwise bounded functions. The Arzel\`{a}-Ascoli theorem
implies it is therefore relatively compact in $\cont(\R_+, \R^d)$. In turn,
this implies that the family of shifted interpolant functions
$\{x^t(\cdot)\}_{t\in \R_+} \in \cont(\R_+, \R^d)$ is also relatively
compact. Indeed, pick an arbitrary sequence $\{t_k\}_{k=1}^\infty$; then one
of the following two cases must occur.
\begin{enumerate}
\item The sequence $\{t_k\}_{k=1}^\infty$ has a bounded subsequence. 
  Pick any cluster point $t$ of $t_k$ (w.l.o.g.\ we may assume $t_k \to t$). 
  Then by the uniform (even Lipschitz) continuity of $x(\cdot)$ on $[0, t + T]$ for 
  any $T \ge 0$, we have that 
  $x^{t_k}(\cdot) \to x^t(\cdot)$ in $\cont(\R_+, \R^d)$.
\item The sequence $\{t_k\}_{k=1}^\infty$ satisfies $t_k \to \infty$.
  Since $x_{t_k}(\cdot)$ is relatively compact in $\cont(\R_+, \R^d)$,
  it has a convergence subsequence $x_{t_{k_n}}(\cdot) \to \wb{x}(\cdot)$ for
  some limit function $\wb{x}(\cdot)$.
  Lemma~\ref{lemma:asymptotic-equivalence} shows that 
  $x^{t_{k_n}}(\cdot) \to \wb{x}(\cdot)$ as well.
\end{enumerate}
These cases combined yield that $\{x^t(\cdot)\}_{t \in \R_+}$
is relatively compact.

\paragraph{Part II: Limiting differential inclusion}
We now establish the remainder of the theorem.  Following our notational
conventions, we define the shifts $y^t(\cdot) = y(t + \cdot)$.  Let
$\{u_k\}_{k \in \N}$ satisfy $u_k \to \infty$ and fix $T > 0$, and w.l.o.g.\
assume that $x^{u_k}(\cdot) \to \wb{x}(\cdot)$ in $\cont(\R_+, \R^d)$ for some
continuous $\wb{x}(\cdot)$.  Viewing
$\mc{Y}_T \defeq \{y^{u_k}(\tau), \tau \in [0, T]\}_{k \in \N}$ as a subset of
the Hilbert space $L^2([0, T])$, $\mc{Y}_T$ is bounded and thus weakly
sequentially compact by the Banach-Alaoglu theorem. Thus the sequence
$\{u_k\}$ has a subsequence $u_{k_n}$ such that
$y^{u_{k_n}}(\cdot) \to \wb{y}(\cdot)$ weakly in $L^2([0, T])$ for some
$\wb{y}(\cdot) \in L^2([0, T])$.  Lemma~\ref{lemma:asymptotic-equivalence}
also implies that the integrated sequence
$x_{u_k}(u_k + \cdot) \to \wb{x}(\cdot)$ in $\cont(\R_+, \R^d)$.  Using the
definition~\eqref{eqn:diff-eq-no-noise} of $x_t(t + \cdot)$, we have for all
$t, \tau \in \R_+$ that
\begin{equation*}
  x_t(t + \tau) = x_t(t) + \int_0^\tau y^t(u) du.
\end{equation*}
Substituting $t = u_{k_n}$, and taking limits as $k_n \to \infty$ on both
sides of this equality, we obtain that for $\tau \le T$ we have
\begin{equation}
  \label{eqn:get-to-bar-y-integral}
  \wb{x}(\tau)  = \wb{x}(0) + \int_{0}^\tau \wb{y}(u) du
\end{equation}
by the weak convergence of $y^{u_{k_n}}$. It remains to show
that $\wb{y}(u) \in H(\wb{x}(u))$ for (almost) all $u \le T$.

As the set $\mc{Y}_T \subset L^2([0, T])$ is bounded, the Banach-Saks theorem
implies the existence of a further subsequence
$\{k_{n_j}\}_{j \in \N}$, which to simplify notation we now assume
without loss of generality is simply the original sequence $k$, such that 
\begin{equation*}
  \frac{1}{N} \sum_{k=1}^N y^{u_k} (\cdot) \to \wb{y}(\cdot)
\end{equation*}
in $L^2([0, T])$ as well as almost everywhere. (We could replace the original
sequence $u_k$ with $u_{k_{n_j}}$ and repeat, \emph{mutatis mutandis}, our
preceding argument.)  For $t \in \R_+$, define the left values
$l_k(t) \defeq \klow(u_k +t)$. Then we have by definition of the interpolated
sequence $y(\cdot)$ and the sequence of mappings $g_k : \R^d \toto \R^d$ that
\begin{equation*}
  y^{u_k}(\tau) \in g_{l_k(\tau)}(x(t_{l_k(\tau)}))
\end{equation*}
where we recall that $t_l = \sum_{k = 1}^l \stepsize_k$.
By the definition $x^t(\tau) = x(t + \tau)$, the triangle inequality,
and the left indices $l_k(\tau)$,
for any fixed $\tau$ we find that
\begin{align*}
  \norm{x(t_{l_k(\tau)}) - \wb{x}(\tau)}
  & \le \norm{x(t_{l_k(\tau)}) - x(u_k + \tau)} + \norm{x^{u_k}(\tau) - \wb{x}(\tau)}
  \\
  & \le \norm{x(t_{l_k(\tau) + 1}) - x(t_{l_k(\tau)})}
    + \norm{x^{u_k}(\tau) - \wb{x}(\tau)} \\
  & \le \stepsize_{l_k(\tau)} \norm{\noise_{l_k(\tau)}}
    + \stepsize_{l_k(\tau)} \norm{y_{l_k(\tau)}}
    + \norm{x^{u_k}(\tau) - \wb{x}(\tau)},
\end{align*}
and this quantity converges to zero as $k \to \infty$ because
$y(\cdot)$ is bounded by condition~\eqref{item:bounded-iterates} of
the theorem, $\stepsize_k \norm{\noise_k} \to 0$ by
the convergence condition~\eqref{item:noise-summation}, and
$x^{u_k}(\tau) \to \wb{x}(\tau)$ by the assumption that $x^{u_k} \to \wb{x}$
in $\cont(\R_+, \R^d)$. That is, $x(t_{l_k(\tau)}) - \wb{x}(\tau) \to 0$ for
all $\tau$ as $k \to \infty$.

Lastly, we apply item~\eqref{item:convergence-set-mapping} in the conditions
for the theorem. Using $y^{u_k}(\tau) \in g_{l_k(\tau)}(x(t_{l_k(\tau)}))$, we
have
\begin{align*}
  \dist(\wb{y}(\tau), H(\wb{x}(\tau)))
  & \le \underbrace{\normbigg{\frac{1}{N}
      \sum_{k = 1}^N y^{u_k}(\tau) - \wb{y}(\tau)}}_{\to 0}
    + \dist\left(\frac{1}{N} \sum_{k = 1}^N y^{u_k}(\tau),
    H(\wb{x}(\tau))\right) \\
  & \le \dist\left(\frac{1}{N} \sum_{k = 1}^N g_{l_k(\tau)}(x(t_{l_k(\tau)})),
    H(\wb{x}(\tau))\right) + o(1) \to 0
\end{align*}
as $N \to \infty$ by item~\eqref{item:convergence-set-mapping}. As
$H(\wb{x}(\tau))$ is closed, we obtain $\wb{y}(\tau) \in H(\wb{x}(\tau))$
for almost every $\tau \in [0, T]$, as desired. Modifying
$\wb{y}(\tau)$ on a suitable null set and recognizing that $T$ is
arbitrary then yields the theorem.

\section{Technical Proofs and Results}

\subsection{Proof of Claim~\ref{claim:composite-weak-convexity}}
\label{sec:proof-composite-weak-convexity}

Fix $\statval \in \statdomain$; we let $h = h(\cdot; \statval)$ and $c =
c(\cdot; \statval)$ for notational simplicity.
Then for any $y, z$ with $\norm{y - x} \le \epsilon$
and $\norm{z - x} \le \epsilon$
and some vector $v$ with $\norm{v} \le \lipc_\epsilon \norm{y - z}^2 / 2$,
we have
\begin{align*}
  h(c(y)) &
  = h(c(z) + \nabla c(z)^T(y - z) + v) \\
  & \stackrel{(i)}{\ge}
  h(c(z) + \nabla c(z)^T (y - z)) - \liph_\epsilon(x) \norm{v} \\
  & \stackrel{(ii)}{\ge}
  h(c(z)) + \partial h(c(z))^T \nabla c(z)^T(y - z)
  - \frac{\liph_\epsilon(x) \lipc_\epsilon(x)}{2} \norm{z - y}^2,
\end{align*}
where inequality $(i)$ follows from the local Lipschitz continuity of
$h$ and $(ii)$ because $h$ is subdifferentiable.
Let $\lambda \ge \liph_\epsilon(x) \lipc_\epsilon(x)$.
Then adding the quantity $\frac{\lambda}{2} \norm{y - x_0}^2$ to both
sides of the preceding inequalities,
we obtain for any $g \in \partial h(c(x))$ that
\begin{align*}
  h(c(y)) + & \frac{\lambda}{2} \norm{y - x_0}^2
  \ge h(c(z)) + (\nabla c(z) g)^T (y - z)
  - \frac{\lambda}{2} \norm{z - y}^2
  + \frac{\lambda}{2} \norm{y -x_0}^2 \\
  & = h(c(z))
  + \frac{\lambda}{2} \norm{z - x_0}^2
  + \<\nabla c(z) g, y - z\>
  + \lambda \<z - x_0, y - z\>.
\end{align*}
That is, the function $y \mapsto h(c(y)) + \frac{\lambda}{2} \norm{y -
  x_0}^2$ has subgradient $\nabla c(z) g + \lambda (z - x_0)$ at $y = z$
for all $z$ with $\norm{z - x} \le \epsilon$; any function with non-empty
subdifferential everywhere on a compact convex set must be convex on that
set~\cite{HiriartUrrutyLe93ab}. In particular,
we see that $y \mapsto f(y; \statval)$ is $\lambda(\statval, x)
= \liph_\epsilon(x, \statval) \lipc_\epsilon(x, \statval)$-weakly convex
in an $\epsilon$-neighborhood of $x$, giving the result.

The final result on Condition~\ref{item:upper-approximation}
is nearly immediate: we have
\begin{equation*}
  h(c(y;\statval);\statval)
  \ge h(c(x;\statval) + \nabla c(x;\statval)^T(y - x);\statval)
  - \frac{\liph_\epsilon(x;\statval) \lipc_\epsilon(x;\statval)}{2}
  \norm{y - x}^2
\end{equation*}
for $y$ in an $\epsilon$-neighborhood of $x$ by
the Lipschitz continuity of $h$ and $\nabla c$.

\subsection{Proof of Lemma~\ref{lemma:frechet-subgradients-of-f}}
\label{sec:proof-frechet-subgradients-of-f}

Recall Assumption~\ref{assumption:weak-convexity} that for all $x \in
\xdomain$ and some $\epsilon > 0$, there exists $\lambda(\statval, x)$ such
that $y \mapsto f(y; \statval) + \frac{\lambda(\statval,x)}{2} \norm{y-x}^2$
is convex for $\norm{y - x} \le \epsilon$, and $\E[\lambda(\statrv, x)] <
\infty$ for all $x$.  Then $f(\cdot; \statval)$ has a Fr\'{e}chet
subdifferential $\partial f(x; \statval)$ and the directional
derivative of $f(\cdot; \statval)$ in the direction $v$ is $f'(x; \statval;
v) = \sup_{g \in \partial f(x; \statval)} \<g, v\>$
(cf.~\cite[Ch.~8]{RockafellarWe98}).  Let $\lambda(\statval) =
\lambda(\statval, x)$ for shorthand, as $x$ is fixed throughout our
argument.  Fix $v \in \R^d$, and let $u$ be near $v$ with $t < \epsilon /
(\norm{u}+\norm{v})$. Then
\begin{align*}
  f(x + tu)
  = \int \left[f(x + tu; \statval) + \frac{\lambda(\statval) t^2}{2}
    \norm{u}^2 \right] dP(\statval)
  - \frac{t^2}{2} \norm{u}^2 \E[\lambda(\statrv)].
\end{align*}
Because $u \mapsto f(x + tu; \statval) + \frac{\lambda(\statval)}{2}
\norm{t u}^2$ is a normal convex integrand~\cite[Ch.~14]{RockafellarWe98},
the dominated convergence theorem implies that
\begin{align*}
  \frac{f(x + tu) - f(x)}{t}
  & = \int \left[\frac{f(x + tu; \statval) - f(x; \statval)}{t}
    + t \frac{\lambda(\statval)}{2} \norm{u}^2 \right] dP(\statval)
  - \frac{t \norm{u}^2}{2} \E[\lambda(\statrv)] \\
  & \to \int f'(x; \statval; v) dP(\statval)
  ~~~ \mbox{as}~ t \to 0, u \to v.
\end{align*}
That is, $f'(x; v) = \int f'(x, \statval; v) dP(\statval)$.
An argument parallel to that of \citet[Propositions
2.1--2.2]{Bertsekas73} yields that
$\partial f(x) = \int \partial f(x; \statval) dP(\statval)$
and that $\partial f(x)$ is compact.

Now we show that $\partial f(\cdot)$ is outer semi-continuous.
Because the support function of the subdifferential
$\partial f(x)$ is the directional derivative of $f$,
the outer semi-continuity of $\partial f$ is equivalent to
\begin{equation}
  \limsup_{k \to \infty} f'(x_k; v)
  \le f'(x; v)
  ~~ \mbox{for~all~} \norm{v} = 1 ~ \mbox{and} ~ x_k \to x \in X
  \label{eqn:limsup-upper-continuity}
\end{equation}
(cf.~\cite[Proposition V.3.3.9]{HiriartUrrutyLe93ab}).  The
sets
$\partial_y (f(y; \statval) + (\lambda(\statval)/2) \norm{y-x}^2)$
are bounded for $y$ in a neighborhood of $x$ because the function
$f$ is weakly convex, where $\lambda(\cdot)$ is $P$-integrable.
Let $\lambda = \E_P[\lambda(\statrv, x)] < \infty$, and define
$g(y) = f(y) + \frac{\lambda}{2} \norm{y - x}^2$.
Then $g$ is convex and continuous near $x$~\cite{Bertsekas73},
and we have~\cite[Corollary VI.6.2.5]{HiriartUrrutyLe93ab} that
\begin{equation*}
  g'(x; v) = \limsup_{y \to x} g'(y; v)
\end{equation*}
for all $v \in \R^d$.
But for convex $g$, we have $g'(x; v) = \lim_{t \downarrow 0}
(g(x + tv) - g(x))/t$, and so the preceding display implies
that as $y \to x$ we have
\begin{align*}
  \lefteqn{-o(1)  \le
    g'(x; v) - g'(y; v)} \\
  & = \lim_{t \downarrow 0} \!
  \left[\!\frac{f(x + tv) - f(x)}{t}
    + \frac{\lambda t \norm{v}^2}{2}\!\right]
  - \lim_{t \downarrow 0} \!
  \left[\!\frac{f(y + tv) - f(y)}{t}
    + \lambda \<v, y - x\>
    + \frac{\lambda t \norm{v}^2}{2} \!\right] \\
  & = \lim_{t \downarrow 0}
  \frac{f(x + tv) - f(x)}{t}
  - \lim_{t \downarrow 0} \frac{f(y + tv) - f(y)}{t}
  - \lambda \<v, y - x\>.
\end{align*}
In particular, the preceding limits exist,
we have $f'(x; v) = \lim_{t \downarrow 0} (f(x + tv) - f(x))/t$,
and inequality~\eqref{eqn:limsup-upper-continuity} holds by taking
$y \to x$. The preceding argument works, of course, for any
weakly convex function, and so applies to $f(\cdot; \statval)$ as well.

The final claim of the lemma is
a standard calculation~\cite{DrusvyatskiyLe18,DrusvyatskiyIoLe16}.

\subsection{Proof of Lemma~\ref{lemma:property-of-Lipschitz-constant}}
\label{sec:proof-property-of-Lipschitz-constant}

That $L_\epsilon(x; \statval) < \infty$ for all $x$ is immediate, as
$G(x; \statval) = \partial \regularizer(x) + \partial f(x; \statval)$, and
subdifferentials of convex functions ($\regularizer$) are compact convex
sets.

We first show the upper semicontinuity of the function
$L_\epsilon(\cdot; \statval)$. Suppose for the sake of contradiction that for
some $x\in \xdomain$ and for some sequence $\{x_k\}_{k=1}^\infty \subset X$
converging to $x$, we have $\lim_{k\to \infty} L_\epsilon(x_k; \statval)$
exists and there is some $\delta > 0$ such that
\begin{equation*}
  \lim_{k\to \infty} L_\epsilon(x_k; \statval) \geq 
  L_\epsilon(x; \statval) + \delta.
\end{equation*}
By definition of $L_\epsilon$ we may choose $x_k' \in X$ such that
$\norm{x_k - x_k'} \le \epsilon$ and subgradient vectors
$p_k(\statval) \in \partial f(x_k'; \statval)$ and
$q_k(\statval) \in \partial \regularizer(x_k')$ satisfying
\begin{equation*}
  L_\epsilon(x_k; \statval) \leq \norm{p_k(\statval) + q_k (\statval)}
  + \delta/2
  ~~ \mbox{for~all~} k.
\end{equation*}
Since 
the sequence $\{x_k'\} \subset X$ is bounded, it has accumulation points
and we may assume w.l.o.g.\ that $x_k' \to x' \in X$, where
$x'$ satisfies
$\norm{x' - x} \leq \epsilon$.
The outer semicontinuity of the subdifferential for weakly convex
functions (Lemmas~\ref{lemma:subg-outer-semicontinuity}
or~\ref{lemma:frechet-subgradients-of-f}) shows that
there must be a subsequence $\{n_k\}$ satisfying
$p_{n_k}(\statval) \to p(\statval) \in \partial f(x'; \statval)$
and $q_{n_k}(\statval) \to q(\statval) \in \partial \regularizer(x')$.
In particular,
\begin{align*}
  \lim_{k\to \infty} L_\epsilon(x_k; \statval) = 
  \limsup_{k \to \infty} L_\epsilon(x_{n_k}; \statval) 
  & \le \limsup_{k \to \infty}
    \norm{p_{n_k}(\statval) + q_{n_k}(\statval)}
    + \frac{\delta}{2} \\
  & = \norm{p(\statval) + q(\statval)} + \frac{\delta}{2}
    \le L_\epsilon(x; \statval) + \frac{\delta}{2},
\end{align*}
which is a contradiction. Thus $L_\epsilon(\cdot; \statval)$ and
$L_\epsilon^2(\cdot; \statval)$ are
upper semicontinuous.

To see that $L_\epsilon(\cdot)$ is upper semicontinuous, we construct an
integrable envelope for the function and then apply Fatou's lemma.  Indeed,
using the assumed $\lipf_\epsilon(x; \statval)$-local Lipschitz continuity
of $y \mapsto f(y; \statval)$ for $y$ near $x$, we have
\begin{equation*}
  \norm{G(y; \statval)} \le \lipf_\epsilon(x; \statval)
  + \norm{\partial \regularizer(y)}
\end{equation*}
for $y$ with $\norm{y - x} \le \epsilon$.
This quantity is integrable, and we may apply Fatou's lemma
and Assumption~\ref{assumption:local-lipschitz}
to obtain
\begin{equation*}
  \limsup_{y \to x} L_\epsilon(y)
  \le \E\left[\limsup_{y \to x} L_\epsilon(y)\right]
  \le \E[L_\epsilon(x; \statrv)]
  \le \sqrt{\E[L_\epsilon(x; \statrv)^2]}.
\end{equation*}

\subsection{Proof of Observation~\ref{observation:boundedness-tikhonov}}
\label{sec:proof-boundedness-tikhonov}

In the case that
$\regularizer(x) = \frac{\lambda}{2} \norm{x}^2$, the stochastic
update in Ex.~\ref{example:sgd} becomes
$x_{k + 1} = \frac{1}{1 + \stepsize_k \lambda} x_k - \frac{\stepsize_k}{1 +
  \stepsize_k \lambda} g_k$, and we have the recursion
\begin{equation*}
  \norm{x_{k + 1}} \le \frac{\norm{x_k}}{1 + \stepsize_k \lambda}
  + \frac{\stepsize_k}{1 + \stepsize_k \lambda} L(\statrv_k)
  \le \prod_{i = 1}^k (1 + \stepsize_i \lambda)^{-1} \norm{x_1}
  + \sum_{i = 1}^k \stepsize_i L(\statrv_i) \prod_{j = i}^k (1 + \stepsize_j \lambda)^{-1}.  
\end{equation*}
Let $L_i = L(\statrv_i)$ for shorthand and define
$\noise_i = L_i - \E[L(\statrv_i)]$, noting $\noise_i$ are i.i.d.\ and mean zero.
Assume that $\lambda = 1$ and $\E[L(\statrv)^2] = 1$
without loss of generality. Defining
$Z_k = \sum_{i = 1}^k \stepsize_i L_i \prod_{j = i}^k (1 + \stepsize_j
\lambda)^{-1}$,
so that
$Z_k - \E[Z_k] = \sum_{i = 1}^k \stepsize_i \noise_i \prod_{j = i}^k (1 +
\stepsize_j \lambda)^{-1}$
and noting that
\begin{equation*}
  \E[Z_{k+1}]
  = \frac{\E[Z_k] + \stepsize_{k + 1}}{1 + \stepsize_{k + 1}}
  = \begin{cases} \le \E[Z_k] & \mbox{if~} \E[Z_k] > 1 \\
    \le 1 & \mbox{if~} \E[Z_k] \le 1 \end{cases}
\end{equation*}
we have that $\sup_k \E[Z_k] < \infty$. Moreover, if we let
$M_k = \sum_{i = 1}^k \stepsize_i \prod_{j = 1}^{i - 1} (1 + \stepsize_j)
\noise_i$,
then $Z_k - \E[Z_k] = \prod_{j = 1}^k (1 + \stepsize_j)^{-1} M_k$, and $M_k$
is a martingale adapted to the filtration
$\mc{F}_k = \sigma(\statrv_1, \ldots, \statrv_k)$. Noting that
$M_{k + 1} - M_k = \stepsize_{k + 1} \prod_{j = 1}^k (1 + \stepsize_j)
\noise_{k + 1}$, we have
\begin{equation*}
  \sum_{k = 1}^\infty \frac{1}{\prod_{j = 1}^k (1 + \stepsize_j)^2}
  \E[(M_{k + 1} - M_k)^2 \mid \mc{F}_k]
  = \sum_{k = 1}^\infty \stepsize_{k + 1}^2 \E[\noise_{k + 1}^2]
  < \infty.
\end{equation*}
Applying standard $L_2$-martingale convergence results (e.g.~\cite[Exercise
5.3.35]{Dembo16}) gives that
$M_k / \prod_{j = 1}^k (1 + \stepsize_j) \cas 0$, and thus
$Z_k \cas \E[Z_k]$, while certainly
$\limsup_k \norm{x_k} \le \limsup_k Z_k$.

\subsection{Proof of Observation~\ref{observation:coercive-bounded}}
\label{sec:proof-observation-coercive-bounded}

The standard first-order optimality conditions for minimization
of (strongly) convex problems 
immediately yield
\begin{equation}
  \label{eqn:one-step-shrink-regularizer}
  \frac{1}{\stepsize_k} \norm{x_k - x_{k + 1}}^2
  + \<g_k, x_{k  +1}\> + \regularizer(x_{k + 1})
  \le \<g_k, x_k\> + \regularizer(x_k),
\end{equation}
and as the iterations are unconstrained, we have
$x_{k + 1} - x_k = - \stepsize_k g_k - \stepsize_k v_{k + 1}$ for some
$v_{k + 1} \in \partial \regularizer(x_{k + 1})$.  We first assume that the
regularizer $\regularizer$ increases in iteration $k$, so
$\regularizer(x_k) \le \regularizer(x_{k + 1})$. Then
inequality~\eqref{eqn:one-step-shrink-regularizer} implies that
\begin{equation*}
  \norm{x_{k + 1} - x_k}
  \norm{g_k + v_{k + 1}}
  = \frac{1}{\stepsize_k} \norm{x_{k + 1} - x_k}^2
  \le \<g_k, x_k - x_{k + 1}\>
  \le \norm{g_k} \norm{x_k - x_{k + 1}},
\end{equation*}
or $\norm{g_k + v_{k + 1}} \le \norm{g_k}$ and we have
$\norm{v_{k + 1}} \le 2 \norm{g_k}$.
The local Lipschitz assumption on $h(c(x; \statval); \statval)$
guarantees $\norm{g_k} \le L (1 + \norm{x_k}^\nu)$.
Choose the constant $1 \le B < \infty$ such that if $\norm{y} \ge B$ then
$\norm{v} \ge 4 L \norm{y}^{\beta - 1}$ for all $v \in \partial
\regularizer(y)$, which must exist by our coercivity assumption,
and $B^\frac{\beta - 1}{\nu} \ge (1/\lambda) B$, where
$\lambda \in \openleft{0}{1}$ is the regularity parameter of
$\regularizer$
for which $\regularizer(\lambda x) \le \regularizer(x)$ for
all $x$ with $\norm{x} \ge B$ (recall $\beta - 1 > \nu$).
Then we have that
\begin{equation*}
  4L \norm{x_{k + 1}}^{\beta - 1} \le 2 \norm{g_k}
  \le 2 L (1 + \norm{x_k}^\nu)
  \le 4 L \cdot \begin{cases}
    \norm{x_k}^\nu & \mbox{if~} \norm{x_k} \ge 1 \\
    1 & \mbox{otherwise}
  \end{cases}
\end{equation*}
if $\norm{x_{k + 1}} \ge B$. Otherwise, we have
$\norm{x_{k + 1}} < B$, so our assumptions on $B$ guarantee
\begin{equation*}
  \norm{x_{k + 1}}
  \le \max\left\{B, 1, \norm{x_k}^\frac{\nu}{\beta - 1}\right\}
  \le \max\left\{B, \norm{x_k}^\frac{\nu}{\beta - 1}\right\}
  \le \max\left\{B, \lambda \norm{x_k}\right\}.
\end{equation*}
We assumed that $\regularizer(x_{k + 1}) \ge \regularizer(x_k)$, so
we see that (because of the $\lambda$ term above)
\begin{equation}
  \label{eqn:coercive-fun}
  \regularizer(x_{k + 1}) \le \regularizer(x_k)
  ~~ \mbox{or} ~~
  \norm{x_{k + 1}} \le B.
\end{equation}

As $\regularizer$ is continuous (it is convex and defined on $\R^d$),
there exists $B'$ such that $\regularizer(x) \le B'$ for all $\norm{x} \le
B$,
and thus inequality~\eqref{eqn:coercive-fun} implies
\begin{equation*}
  \regularizer(x_{k + 1}) \le \regularizer(x_k) \vee B'.
\end{equation*}
We may assume w.l.o.g.\ that $\regularizer(x_1) \le B'$, because otherwise
we could simply increase $B' = \regularizer(x_1)$.
An inductive argument implies that
$\regularizer(x_{k + 1}) \le B'$ for all $k$,
which implies the result.

\bibliography{bib}
\bibliographystyle{abbrvnat}

\end{document}